\documentclass[11pt]{amsart}

\usepackage{amscd,amssymb,amsmath,latexsym,enumerate}
\usepackage[utf8]{inputenc}
\usepackage{amscd,amssymb,amsmath,amsthm,bbm}
\usepackage[mathscr]{euscript}
\usepackage{mathrsfs}
\usepackage{longtable}

\numberwithin{equation}{section}

\usepackage[breaklinks=true,colorlinks=true,
linkcolor=black,urlcolor=black,citecolor=black,% PRINT
bookmarks=true,bookmarksopenlevel=2]{hyperref}

\DeclareMathAlphabet{\mathpzc}{OT1}{pzc}{m}{it}

\usepackage{xcolor}

\usepackage{tikz}
\usetikzlibrary{decorations,arrows}
\usetikzlibrary{decorations.pathmorphing}
\usepgflibrary{decorations.pathreplacing} 
\usepackage{pgf}
\usepackage{tikz-3dplot}
\usetikzlibrary{patterns}

\newtheorem{definition}{Definition}[section]
\newtheorem{theorem}[definition]{Theorem}
\newtheorem{proposition}[definition]{Proposition}
\newtheorem{lemma}[definition]{Lemma}

\newtheorem{corollary}[definition]{Corollary}

\theoremstyle{remark}
\newtheorem{example}[definition]{Example}
\newtheorem{remark}[definition]{Remark}

\newtheorem{notation}[definition]{Notation}
\newtheorem{construction}[definition]{Construction}
\newtheorem{problem}[definition]{Problem}

\renewcommand{\L}{\mathfrak}
\newtheorem{no}[definition]{\S}
\newcommand{\cB}{{\mathcal B}}
\newcommand{\CC}{{\mathbb C}}
\newcommand{\KK}{{\mathbb K}}
\newcommand{\LL}{{\mathbb L}}
\newcommand{\NN}{{\mathbb N}}
\newcommand{\QQ}{{\mathbb Q}}
\newcommand{\RR}{{\mathbb R}}
\newcommand{\ZZ}{{\mathbb Z}}

\newcommand{\Aa}{{\mathcal A}}
\newcommand{\Dd}{{\mathcal D}}
\newcommand{\Ss}{{\mathcal S}}

\newcommand{\Pp}{{\mathcal P}}

\newcommand{\supp}{{\operatorname{supp}}}			%support
\newcommand{\qand}{\quad \text{and} \quad}

\renewcommand{\mid}{:}
\renewcommand{\subset}{\subseteq}
\allowdisplaybreaks

\newcommand{\Hmm}[1]{\leavevmode{\marginpar{\tiny%
$\hbox to 0mm{\hspace*{-0.5mm}$\leftarrow$\hss}%
\vcenter{\vrule depth 0.1mm height 0.1mm width \the\marginparwidth}%
\hbox to 0mm{\hss$\rightarrow$\hspace*{-0.5mm}}$\\\relax\raggedright #1}}}

\begin{document}
\title{Symbolic substitution systems beyond abelian groups}
\author{Siegfried Beckus, Tobias Hartnick, Felix Pogorzelski}

\address{Institut für Mathematik\\
Universit\"at Potsdam\\
Potsdam, Germany}
\email{beckus@uni-potsdam.de}

\address{Institut f\"ur Algebra und Geometrie\\
KIT\\
Karlsruhe, Germany}
\email{tobias.hartnick@kit.edu}

\address{Mathematisches Institut\\
Universit\"at Leipzig\\
Leipzig, Germany}
\email{felix.pogorzelski@math.uni-leipzig.de}

\begin{abstract} In this article we construct the first examples of strongly aperiodic linearly repetitive Delone sets in non-abelian Lie groups by means of symbolic substitutions. In particular, we find such sets in all $2$-step nilpotent Lie groups with rational structure constants such as the Heisenberg group. More generally, we consider the class of $1$-connected nilpotent Lie groups whose Lie algebras admit a rational form and a derivation with positive eigenvalues. Any group in this class admits a lattice which is invariant under a natural family of dilations, and this allows us to construct primitive non-periodic symbolic substitutions. We show that, as in the abelian case, the associated subshift (and hence the induced Delone dynamical system) is minimal, uniquely ergodic and weakly aperiodic and consists of linearly repetitive configurations. In the $2$-step nilpotent case, it is even strongly aperiodic.
\end{abstract}

%%%%%%%%%%%%%%%%%%%%%%%%%%%%%%%%%%%%%%%%%%%%%%%%%%%%%%%%%%%%%%%%%%%%

\maketitle

%\tableofcontents

\section{Introduction}

\subsection{Motivation}
The theory of aperiodic order is an important source of examples of dynamical systems over locally compact abelian groups: If $\Lambda$ is an aperiodic Delone set in a locally compact abelian group $A$ and $\Omega_\Lambda$ denotes its orbit closure in the space of all such Delone sets (with respect to the Chabauty--Fell topology), then  $A \curvearrowright \Omega_\Lambda$ is a topological dynamical system. In principle, such systems can be studied for arbitrary Delone sets $\Lambda$, but in practice they have been studied mostly for three (overlapping) classes of Delone sets: 
\begin{itemize}
\item aperiodic Meyer sets, i.e.\ Delone sets arising from cut-and-project constructions;
\item aperiodic Delone sets arising from symbolic substitutions inside a lattice in $A$;
\item aperiodic Delone sets arising from geometric substitutions.
 \end{itemize}
The present article continues a general program initiated in \cite{BjHa18,BjHaPo18} and developed further in \cite{FishExtensionsOfSchreiber, BjHStulemeijer, BP18, BjHAnalyticProperties, Machado1, CordesHTonic, Hrushovski, BjHaPoII, BjHaPoIII, Machado2, BjHa3, Machado3, BHP25-LR} which aims to expand the theory of aperiodic order to non-abelian locally compact groups and to study the resulting dynamical systems. 
 Thanks to the work cited above there is now a rich theory of Meyer sets in non-abelian groups; the dynamical properties of the resulting systems can sometimes be surprisingly different from the abelian case. The goal of the present article is to extend key aspects of the theory of symbolic substitution dynamical systems associated with lattices in abelian groups \cite{Que87, Sol97,Sol98,AnPu98,Dur00,Fogg02,DaLe06,BaakeGrimm13,KeLe13,BarOl14} to symbolic substitution systems associated with lattices in certain non-abelian group. Using our approach we construct the first examples of \emph{strongly aperiodic linearly repetitive Delone sets} in non-abelian locally compact groups. (Note that, in contrast to the abelian case, there is so far no construction of such Delone sets using cut-and-project methods for non-abelian locally compact groups.) Our construction applies to a large class of lattices in nilpotent Lie groups which contains hundreds of examples even in relatively small dimensions (say, $\leq 10$); see \cite{SolomonSmilansky} and  \cite{BL21,BL23} concerning recent work on substitutions in spaces associated with other non-abelian (semi-)groups.
 
\subsection{The case of the Heisenberg group}
The smallest non-abelian example to illustrate our results is the $3$-dimensional Heisenberg group $\mathbb{H}_3(\mathbb{R})$ which we think of as $\mathbb{R}^3$ with non-commutative multiplication given by\footnote{The factor $\tfrac{1}{2}$ can be replaced by any non-zero real number without changing the isomorpism type of the group; our normalization follows a convention explained in more details below.}

\[
(x,y,z)\ast(x', y', z') = (x+x', y+y', z+z' + \tfrac{1}{2} (xy' - yx')).
\]
We equip $\mathbb{H}_3(\mathbb{R})$ with the Cygan--Korányi metric
\[
d((x,y,z),(x',y',z'))= \sqrt[4]{(|x'-x|^2 + |y'-y|^2)^2 + |z'-z-\tfrac{1}{2} (xy' - yx')|^2}.
\]
The points with coordinates in the ring $2\mathbb{Z}$ then form a lattice $\Gamma < \mathbb{H}_3(\mathbb{R})$, and for this specific lattice our main results (Theorems~\ref{Thm-ExPrimNonPerSubst}, \ref{MainTheorem}, \ref{Thm-strongAP} and Corollary~\ref{Cor-exLR_Del} below) specialize as follows; see \S \ref{Par-BasTerm} for terminology:
\begin{theorem}
\label{Thm-HeisenbergMain}
 For every finite alphabet $\Aa$ with $|\Aa|\geq 2$ there exists a linearly repetitive configuration $\omega \in \Aa^{\Gamma}$ whose $\Gamma$-orbit closure is minimal, uniquely ergodic and strongly aperiodic (i.e.\ all elements in the orbit closure have trivial stabilizer).
\end{theorem}
The proof of Theorem \ref{Thm-HeisenbergMain} is constructive and provides very explicit configurations using symbolic substitutions. On the geometric level, we obtain:
\begin{corollary} 
\label{Cor-UniqErg_Min_StrAper-Heisenberg}
There is a linearly repetitive Delone set in the Heisenberg group $\mathbb{H}_3(\mathbb{R})$ whose hull is uniquely ergodic, minimal and strongly aperiodic.
\end{corollary}
While our approach follows the classical construction for lattices in abelian Lie groups, several substantial difficulties arise in the non-abelian setting, and these are already apparent in the case of the Heisenberg group:
\begin{enumerate}
\item Compatibility with $\Gamma$ requires working with dilations of $\mathbb H_3(\mathbb R)$ that scale different directions in the Heisenberg group by distinct factors. 
\item Ensuring covariance of the substitution under the group action causes patches in $\mathbb H_3(\mathbb R)$ to grow much more irregularly under iteration than in $\mathbb R^3$; see Figure~\ref{Fig-Heisenberg_Supp}. This leads to a technical theory of support growth (Section~\ref{SecGrowth}).
\item Establishing strong aperiodicity is considerably more delicate. We introduce a criterion for weak aperiodicity, apparently new even in the abelian setting, which allows the construction of linearly repetitive aperiodic configurations with minimal orbit closure in great generality. In the abelian case, strong aperiodicity then follows directly, whereas in the Heisenberg case additional arguments are required.
\item The proof of unique ergodicity is likewise more involved and relies on a property of the Cygan--Korányi metric known as \emph{exact polynomial growth}, using results from~\cite{BHP25-LR}.
\end{enumerate}

\subsection{An axiomatic framework for substitutions}\label{IntroAxioms}
To state Theorem \ref{Thm-HeisenbergMain} in its natural generality, we introduce an abstract axiomatic framework for symbolic substitutions. As an input we are going to need two kinds of data:

\medskip

\noindent \textbf{Geometric data:} Pick a locally compact, second countable Hausdorff (lcsc) group $G$ and a lattice $\Gamma < G$. We assume that for some left-invariant metric $d$ on $G$ (inducing the given topology) there 
exists a one-parameter group $(D_{\lambda})_{\lambda>0}$ of dilations of $(G,d)$ such that
\[
d(D_{\lambda}(g), D_{\lambda}(h)) = \lambda\, d(g,h)
\quad\text{and}\quad 
D_{\lambda_0}(\Gamma) \subseteq \Gamma
\]
for some $\lambda_0>1$. We also fix a bounded Borel fundamental domain $V \subset G$ for $\Gamma$ containing $e$ and with non-empty interior and refer to $\mathcal D = (G, d, (D_{\lambda})_{\lambda>0}, \Gamma, V)$ as a \emph{dilation datum} over $G$ (see Definition~\ref{Def-DilationDatum}). For some results we will need to assume that the dilation datum satisfies a technical assumption called \emph{exact polynomial growth}, see \S \ref{EPG}.

\medskip

\noindent \textbf{Combinatorial data:} 
Given a dilation datum $\Dd$ as above we choose a finite alphabet $\mathcal A$ and a \emph{substitution rule}
\[
S_0: \mathcal A \to \mathcal A^{D_{\lambda_0}(V) \cap \Gamma},
\]
where the \emph{stretch factor} $\lambda_0 \gg 1$ is chosen such that $D_{\lambda_0}(\Gamma) \subseteq \Gamma$ and assumed to satisfy a technical condition which we call ``$V$-sufficiency'' (see Definition~\ref{Def-stretch-factor}). If $V$ contains an identity neighbourhood, then this condition reduces to $\lambda_0$ being ``sufficiently large'' (Proposition \ref{Prop-SuffCriter:suff_large-old}). We call $\Ss = (\mathcal A, \lambda_0, S_0)$ a \emph{substitution datum} over $\Dd$.

Given a substitution datum $\Ss$ over a dilation datum $\Dd$, the underlying substitution rule admits a unique $D_{\lambda_0}$-equivariant, locally defined extension to a \emph{substitution map} 
$S: \mathcal A^{\Gamma} \to \mathcal A^{\Gamma}$ (cf.\ Proposition~\ref{Prop-SubstitutionAxioms}); the construction is more involved than in the abelian case. 

\subsection{Existence of good substitution data}
Let us call a lcsc group $G$ a \emph{substitution group} if there exists a dilation datum $\Dd$ over $G$. Such groups are quite similar to the example of the Heisenberg example above: They are nilpotent Lie groups, whose underlying 
manifold is diffeomorphic to $\mathbb{R}^n$ for some $n$ and whose multiplication is given by polynomials in suitable local coordinates. We discuss various characterizations of such groups in Section~\ref{SecDilSubExamples} and in more detail in Section~\ref{Sec-Lie}. Given a substitution group one can ask for the existence of dilation and substitution data with special properties, such as \emph{primitivity} (Definition \ref{Def-Primitive}) and \emph{non-periodicity} (Definition~\ref{Def-S_NonPeriodic}).

\begin{theorem} 
\label{Thm-ExPrimNonPerSubst}
Let $G$ be a substitution group and let $\mathcal A$ be a finite alphabet. If $\dim G \geq 2$ and $|\mathcal A| \geq 2$, then there exists a dilation datum $\Dd$ of exact polynomial growth over $G$ and a primitive, non-periodic substitution datum $\mathcal S$ over $\Dd$ with alphabet $\mathcal A$.
\end{theorem}
The proof of Theorem~\ref{Thm-ExPrimNonPerSubst}, as discussed in Section~\ref{Sec-Ex_GoodSubstData} below, is constructive. Using the structure theory of substitution groups one first shows that $G$ admits a special kind of dilation datum $\mathcal{D}$ which we call \emph{split dilation datum} (Definition \ref{Def-SplitDilationDatum}). Given a split substitution datum we can then choose an alphabet $\mathcal A$ (with $|\mathcal A| \geq 2$) and a suitable stretch factor $\lambda_0$, and as long as our substitution rule satisfies a few explicit rules (as stated in Definition~\ref{Def-GoodSubst}) the resulting substitution datum will be primitive and non-periodic. There is thus a huge flexibility in our construction: Not only can we easily construct hundreds of examples of substitution groups, but for each of these groups we can easily construct hundreds of primitive, non-periodic substitution data.

\subsection{The subshift defined by a substitution}
From now on let $S : \mathcal A^\Gamma \to  \mathcal A^\Gamma$ be a substitution map constructed from a substitution datum $\mathcal S$ over a dilation datum $\mathcal D$. 

\begin{no} The lattice $\Gamma$ acts by homeomorphisms on the compact metrizable space $\mathcal A^{\Gamma}$ via
\[
 \gamma \omega(\eta) = \omega(\gamma^{-1}\eta) \quad (\omega \in \mathcal{A}^\Gamma, \gamma,\eta \in \Gamma),
\]
and we refer to a closed $\Gamma$-invariant subset $\Omega \subset \mathcal A^\Gamma$ as a \emph{subshift}. Exactly as in the abelian case, the substitution rule $S$ gives rise to a subshift $\Omega(S) \subset \mathcal A^{\Gamma}$ consisting of those configurations $\omega \in \mathcal A^{\Gamma}$ whose patches can be generated by iterating $S$ on a single letter of $\mathcal A$ (see Construction~\ref{Def-SubstitutionSystem}). We refer to $\Gamma \curvearrowright \Omega(S)$ as the \emph{substitution system} generated by $S$.
\end{no}
\begin{no}
\label{Par-BasTerm} Recall that a subshift $\Omega$ is called \emph{minimal} if every $\Gamma$-orbit in $\Omega$ is dense, and \emph{uniquely ergodic} if it admits a unique $\Gamma$-invariant probability measure. We say that $\Omega$ is \emph{weakly aperiodic} if there exists some $\omega \in \Omega$ such that $\mathrm{Stab}_\Gamma(\omega) := \{ \gamma \in \Gamma: \gamma \omega = \omega \} = \{e\}$; if the latter condition holds even for \emph{all} $\omega\in \Omega$, then we say that $\Omega$ is \emph{strongly aperiodic}.
We say that $\omega \in \mathcal A^\Gamma$ is \emph{linearly repetitive} if there exists $C>0$ such that for all $R > 0$ every patch of $\omega$ of size $R$ occurs in every ball of radius $C\cdot R$ in $G$; here ``size'' is measured with respect to the given invariant metric $d$ on $G$.
\end{no}

\begin{theorem}[Non-abelian primitive symbolic substitution systems]
\label{MainTheorem} 
The substitution system $\Gamma \curvearrowright \Omega(S)$ has the following properties:
\begin{enumerate}[(a)]
\item $\Omega(S) \subset \mathcal A^\Gamma$ is an $S$-invariant subshift.
\item Some power of $S$ has a fixpoint in $\Omega(S)$; in particular $\Omega(S)$ is non-empty.
\end{enumerate}
If the underlying substitution rule $\Ss$ is primitive, then
\begin{enumerate}[(a)]
\setcounter{enumi}{2}
\item every element in $\omega \in \Omega(S)$ is linearly repetitive.
\item the action $\Gamma \curvearrowright \Omega(S)$ is minimal.
\item the action $\Gamma \curvearrowright \Omega(S)$ is uniquely ergodic provided $\Dd$ has exact polynomial growth.
\end{enumerate}
If the underlying substitution rule $\Ss$ is non-periodic, then
\begin{enumerate}[(a)]
\setcounter{enumi}{5}
\item $\Omega(S)$ is weakly aperiodic.
\end{enumerate}
\end{theorem}

Theorem \ref{MainTheorem}  combines Proposition~\ref{Prop-Subshift}, Proposition~\ref{Prop-LegalFixpoints}, Theorem~\ref{Thm-LinearRep}, and Theorem~\ref{Thm-weakAP-ex} below. The fact that $\Omega(S)$ is nonempty relies crucially on the aforementioned sufficiency condition on the stretch factor, which guarantees that iterations of the substitution rule generate arbitrarily large patches (see Proposition~\ref{Prop-SupportGroth}). The proof of unique ergodicity is based on a result of our companion paper \cite{BHP25-LR} which is applicable due to exact polynomial growth of metric balls in $G$;  we refer to Section~\ref{Sec-LR+UniqErg} for details. Parts~(a)--(d) of Theorem~\ref{MainTheorem} are classical when the underlying metric group $(G,d)$ is Euclidean; see e.g.\ \cite{Sol97,Sol98, AnPu98, Dur00, Fogg02, DaLe06, BaakeGrimm13} and references therein. Also Part~(e) has been established in the Euclidean setting; see~\cite{Dur00,DaLe01,Len02b,LaPl03,DaLe06}. Our notion of non-periodic substitutions appears to be new even in the Euclidean case of symbolic substitutions. 
It adapts related ideas from the geometric setting~\cite{Sol97,Sol98,AnPu98}, where injectivity of the substitution ensures aperiodic tilings. 
By contrast, weak aperiodicity for symbolic dynamical systems over a lattice in the Euclidean space is characterized by the existence of proximal pairs~\cite{BarOl14}; see also~\cite{BaakeGrimm13,KeLe13} for higher-dimensional discussions.

It is an interesting question to find conditions under which a given substitution system $\Gamma \curvearrowright \Omega(S)$ is actually \emph{strongly} aperiodic. If $\Gamma$ happens to be abelian, then it is easy to see that every minimal and weakly aperiodic subshift is automatically strongly aperiodic; thus assuming primitivity and non-periodicity of the substitution rule is sufficient in this case. However, the corresponding argument breaks down in the non-abelian case. If the ambient substitution group $G$ happens to be $2$-step nilpotent (as in the case of the Heisenberg group), then there is a distinguished class of dilation data for $G$ which we call \emph{special dilation data} (see \S \ref{SpecialSub} below), and in this case we can establish strong aperiodicity by reduction to the abelian case:

\begin{theorem}[Strong aperiodicity]
\label{Thm-strongAP}
Let $\Dd$ be a special dilation datum over a $2$-step nilpotent substitution group. Then for every non-periodic primitive substitution datum $\Ss$ over $\Dd$ the associated substitution system $\Omega(S)$ is strongly aperiodic.
\end{theorem}

It is likely that for suitable substitution data strong aperiodicity can also be established for certain (maybe even all) higher step examples by an inductive argument similar to ours, but since the set-up gets quite technical if one increases the step-size, we will not pursue this here.

\subsection{Application to Delone sets} 
Our symbolic results can be used to establish corresponding results for Delone sets in the usual way. Given a Delone subset $\Lambda$ of a Lie group $G$ we denote by $\Omega_\Lambda$ the $G$-orbit closure of $\Lambda$ with respect to the Chabauty-Fell topology; from Theorems \ref{Thm-ExPrimNonPerSubst},  \ref{MainTheorem}  and \ref{Thm-strongAP} we can then deduce the following corollary by standard methods (see Section~\ref{Sec-Delone}).

\begin{corollary}[Linearly repetitive Delone sets] 
\label{Cor-exLR_Del} 
Every substitution group $G$ contains a linearly repetitive Delone set $\Lambda$ with trivial $G$-stabilizer such that $\Omega_\Lambda$ is minimal and uniquely ergodic. If $G$ is $2$-step nilpotent, then we can moreover arrange for $\Omega_\Lambda$ to be strongly aperiodic.
\end{corollary}
It is an interesting problem, which lcsc groups admit linearly repetitive Delone sets as in the corollary. In the abelian case, the standard way to construct such sets is either by substitution methods or using the cut-and-project method of Meyer for a carefully chosen window. In the non-abelian case, the methods of the present paper provide non-periodic linearly repetitive Delone sets in substitution groups, but it is unclear how to produce such sets beyond the current setting. For example, we would like to propose the question whether the $3$-dimensional Lie group $\mathrm{SL}_2(\RR)$ admits any linearly repetitive Delone set (say, with respect to a left-invariant Riemannian metric) of unbounded pattern complexity. By analogy with the abelian case (and since the substitution method is not applicable) one might try to use Meyer's cut-and-project methods; however, at least for polytopal windows it was established in \cite{KaiserDissUnpublished} that 
the pattern counting functions of cut-and-project sets in $\mathrm{SL}_2(\RR)$ grow exponentially, so this is unlikely to work. In general, it would be desirable to obtain non-trivial lower bounds on the possible complexity and repetitivity functions of aperiodic Delone sets for various classes of non-abelian lcsc groups. 

\subsection{Related work}
Although there is a large body of literature concerning dynamical systems over abelian groups which satisfy some form of linear repetitivity, very few examples seem to be known in the non-abelian case. In fact, the only examples that we are aware of are those constructed recently by P{\'e}rez \cite{Perez} in the context of Schreier graphs associated with certain actions of spinal groups. We thus believe that the present article provides the first construction of aperiodic linearly repetitive Delone sets in non-abelian nilpotent Lie groups and thus (via the Voronoi construction) of aperiodic self-similar tilings in such groups. On the contrary, \emph{periodic} self-similar tilings of substitution groups have been a subject of study ever since Strichartz' highly influential article \cite{Str92}; see in particular the work of Gelbrich \cite{Gelbrich} for a very explicit construction of periodic self-similar tilings in the Heisenberg group.

Concerning general Delone sets in nilpotent Lie groups, not a lot seems to be known yet. Machado \cite{Machado1} has established that such Delone sets satisfy the Meyer condition if and only if they are relatively dense subsets of cut-and-project sets. The linearly repetitive Delone sets constructed in the current article are actually Meyer, but for the trivial reason that they are even relatively dense in -- and hence bounded displacement (BD) equivalent and bilipschitz (BL) equivalent to -- a lattice. On the contrary, the work of Dymarz, Kelly, Li and Lukyanenko \cite{DKLL18} shows that there are many different BD and BL equivalence classes of Delone sets in nilpotent Lie groups, which in general cannot be represented by lattices. We currently do not know whether BL equivalence classes in nilpotent Lie groups, which do not contain a lattice, can admit linearly repetitive Delone sets. 

Our examples are also special from a dynamical point of view, since the horizontal factors (in the sense of \cite{BjHa3}) of their hull dynamical systems are homogeneous. Nevertheless, it might be interesting to study their central diffraction in the sense of \cite{BjHa3} and compare it to the central diffraction of the ambient lattices.

Section~\ref{SecDilSubExamples} introduces the general framework, guided by several illustrative examples. 
In Section~\ref{Sec-SubstMap} we construct and characterize the associated substitution map and analyze the support growth under iterated applications of this map. 
Section~\ref{Sec-SubstDynamSyst} develops the corresponding subshifts and contains the proof of Theorem~\ref{MainTheorem}. 

In Section~\ref{Sec-Ex_GoodSubstData} we establish the existence of primitive and non-periodic substitution data and prove Theorem~\ref{Thm-ExPrimNonPerSubst} by introducing the notion of good substitution rules. 
Section~\ref{Sec-StrongAper} contains the proof of Theorem~\ref{Thm-strongAP}. 
These results are then applied in Section~\ref{Sec-Delone} to construct (strongly) aperiodic, linearly repetitive Delone sets in substitution groups, yielding Corollary~\ref{Cor-exLR_Del}. 

Up to this point, only a few structural facts about Lie groups are required; these are collected and developed in Section~\ref{Sec-Lie}. 
We use these Lie-theoretic results as a black box at several points, so that the exposition of this paper is accessible without specific expertise in Lie theory. 

Finally, Appendix~\ref{AppendixCensus} demonstrates that already in low dimensions a large variety of substitution groups arises.

\subsection*{Acknowledgment} We are grateful to Boris Solomyak and Enrico LeDonne for explanations concerning non-periodic substitutions in the geometric setting and related results for symbolic dynamical systems, respectively homogeneous Lie groups and their metrics. We are also grateful to Yves Cornulier and an anonymous referee for explanations concerning gradings on nilpotent Lie algebras, and in particular for pointing out the classification of maximal gradings on low-dimensional nilpotent Lie algebras by Carles \cite{Carles1, Carles2}, and to Daniel Roca Gonzalez for suggesting the axioms in Proposition~\ref{Prop-SubstitutionAxioms}. Furthermore, we thank Daniel Lenz and Jeffrey Lagarias for inspiring discussions on linearly repetitive Delone sets at the Oberwolfach Workshop 1740 and Michael Baake for a fruitful conversation on proximal pairs and aperiodicity at the Oberwolfach workshop 2101b. 

We thank the Justus-Liebig-Universit\"at Gie{\ss}en, KIT  and the Universit\"at Leipzig for providing excellent working conditions during our mutual visits.
This research was supported through the program ``Research in Pairs'' by the Mathematisches Forschungsinstitut Oberwolfach in 2020 hosting T.H. 
This work was partially supported by the Deutsche Forschungsgemeinschaft [BE 6789/1-1 to S.B.] and the German Israeli Foundation for Scientific Research and Development [I-1485-304.6/2019 to F.P.].

\section{The axiomatic framework}
\label{SecDilSubExamples}

In this section we discuss in more detail the axiomatic framework for substitutions alluded to in Subsection \ref{IntroAxioms} of the introduction. We also discuss several examples in order to illustrate the definitions.

\subsection{Dilation data} 
In this subsection we discuss in more detail the notion of a dilation datum mentioned in the introduction.
\begin{no} Throughout this article, the letter $G$ is reserved for an \emph{lcsc group}, i.e. a topological group whose topology is Hausdorff, locally compact and second countable, and $e$ denotes its neutral element.
Moreover, we denote by $ \mathrm{Aut}(G)$ the group of continuous automorphisms of $G$. If $D: (\mathbb{R}_{>0}, \cdot) \to \mathrm{Aut}(G)$, $\lambda \mapsto D_\lambda$ is a group homomorphism, then $(D_\lambda)_{\lambda > 0}$ is called a \emph{one-parameter group of automorphisms of $G$}.

By a classical theorem of Struble, every lcsc group $G$ admits a proper metric $d$ which induces the given topology and is \emph{left-invariant} in the sense that
$d(xg,xh) = d(g,h)$ for all $x,g,h \in G$.
\end{no}
\begin{definition} 
\label{Def-DilationGroup}
A \emph{dilation group} is a triple $(G, d, (D_\lambda)_{\lambda > 0})$, where
\begin{enumerate}[(D1)]
\item $G\neq \{e\}$ is an lcsc group, called the \emph{underlying group}, and $d$ is a proper left-invariant metric on $G$ which induces the given topology on $G$;
\item $(D_\lambda)_{\lambda > 0}$ is a one-parameter group of automorphisms of $G$, called the \emph{underlying dilation family}, such that 
\[
d(D_\lambda(g), D_\lambda(h)) = \lambda \cdot d(g,h) \quad \text{ for all } \lambda > 0, g,h \in G.
\]
\end{enumerate}
\end{definition}
\begin{no} Let $(G, d, (D_\lambda)_{\lambda > 0})$ be a dilation group. Given $x \in G$ and $r>0$ we denote by
$B(x,r)$ the open ball of radius $r$ around $x$ with respect to $d$; then for all $\lambda > 0$ we have
\begin{equation}
\label{Eq-OldLemma3_1}
D_\lambda(B(x,r)) = D_\lambda(x B(e,r)) = D_\lambda(x) B(e, \lambda r) = B(D_\lambda(x), \lambda r).
\end{equation}
Combining left-invariance of $d$ with the triangle inequality we also obtain
\[
d(gh,x) \leq d(gh,g) + d(g,x) \leq d(h,e) + d(g,x) \quad (g,h,x \in G),
\]
and hence the corresponding balls satisfy
\begin{equation}
\label{Eq-OldLemma3_2}
B(x,r)B(e, s) \subset B(x, r+s).
\end{equation}
In the sequel we will use these identities without further reference. We also set $|g| := d(g,e)$. 

\end{no}
\begin{definition}\label{Def-AdaptedLattice} Let $(G, d, (D_\lambda)_{\lambda > 0})$ be a dilation group and let $\Gamma < G$ be a discrete subgroup.
	\begin{enumerate}[(i)]
		\item $\Gamma$ is called a \emph{lattice} if $G/\Gamma$ admits a $G$-invariant probability measure. It is called a \emph{uniform lattice} if $G/\Gamma$ is compact.
		\item $\Gamma$ is called \emph{adapted} to $(D_\lambda)_{\lambda > 0}$ if there exists some $\lambda_0 > 1$ such that $D_{\lambda_0}(\Gamma) \subset \Gamma$.
	\end{enumerate}
\end{definition}

\begin{no}\label{DInclusionProper} Let $(G, d, (D_\lambda)_{\lambda > 0})$ be a dilation group. A subset $\Lambda \subset G$ is called \emph{$(r,R)$-Delone} for some $R\geq r>0$ if it is \emph{$r$-uniformly discrete} and \emph{$R$-relatively dense}, i.e. $\mathrm{dist}(g, \Gamma) \leq R$ for all $g \in G$ and $d(\gamma_1, \gamma_2) \geq r$ for all distinct elements $\gamma_1, \gamma_2 \in \Gamma$. Every discrete subgroup of $G$ is uniformly discrete, and hence a uniform lattice is the same as Delone subgroup. Moreover, if $G$ is non-compact, then every lattice in $G$ is infinite.

If a subgroup $\Gamma < G$ is $r$-uniformly discrete for some $r>0$, then $D_{\lambda}^n(\Gamma)$ is $\lambda^n r$-uniformly discrete. In particular, if $D_\lambda(\Gamma) = \Gamma$ for some $\lambda$, then $\Gamma$ is  $\lambda^n r$-uniformly discrete for all $n$, hence $\Gamma = \{e\}$. In particular, if $G$ is non-compact, then any adapted lattice satisfies $D_{\lambda_0}(\Gamma) \subsetneq \Gamma$ for some $\lambda_0 > 1$.
\end{no}
\begin{no} Let $\Gamma < G$ be a uniform lattice. Then the canonical projection $G \to G/\Gamma$ is a covering map and hence
there exist a bounded Borel set $V$ with non-empty interior which contains $e$ and intersects each $\Gamma$-orbit (for the left-multiplication action of $\Gamma$ on $G$) in a single point. We refer to such a set $V$ as a \emph{good fundamental domain} for $\Gamma$. Note that it is always possible to choose $V$ to contain a neighbourhood of $e$, but we do not want to assume this here, since it would exclude certain examples from the literature (as in \cite{BaBePoTe25,FraMa22}).
\end{no}
\begin{definition} 
\label{Def-DilationDatum}
A \emph{dilation datum} $\Dd = (G, d, (D_\lambda)_{\lambda >0}, \Gamma, V)$ consists of an underlying dilation group $(G, d, (D_\lambda)_{\lambda > 0})$, an adapted uniform lattice $\Gamma < G$ and a good fundamental domain $V$ for $\Gamma$.
\end{definition}
Dilation data exist both in the abelian and in the non-abelian world:
\begin{example} 
\label{Ex-EuclideanIntro}
Let $G = (\mathbb{R}^n, +)$ with Euclidean metric $d$ and let $D_\lambda \in \mathrm{Aut}(G)$ be given by
\[
D_\lambda(x_1, \dots, x_n) := (\lambda^{r_1}x_1, \dots, \lambda^{r_n}x_n) \text{ for some parameters }r_1, \dots, r_n > 0.
\]
Then  $(G, d, (D_\lambda)_{\lambda > 0})$ is a dilation group and if $r_1, \dots, r_n$ are integers, then $\Gamma := \mathbb{Z}^n$ is an adapted lattice and $V := [-\frac 1 2, \frac 1 2)^n$ is a fundamental domain. 
\end{example}

\begin{example}
\label{Ex-HeisenbergIntro} 
Let $G = \mathbb{H}_3(\RR)$ be the $3$-dimensional Heisenberg group; we recall that according to our convention this is defined as $\mathbb{R}^3$ with multiplication given by\
\[
(x,y,z)\ast(x', y', z') = (x+x', y+y', z+z' + \tfrac{1}{2} (xy' - yx')).
\]
We then have a dilation datum $\mathcal{D}_{\mathbb{H}}:=(\mathbb{H}_3(\RR), d, (D_\lambda)_{\lambda > 0},\Gamma, V)$, where 
\begin{itemize}
\item $d$ is the left-invariant metric with underlying norm
	\begin{equation}\label{CKMetric}
		|(x,y,z)| = ((x^2 + y^2)^2 + z^2)^{1/4},
	\end{equation}
\item the dilation family $(D_\lambda)_{\lambda > 0}$ is defined by $D_\lambda(x, y, z) := (\lambda x, \lambda y, \lambda^2 z)$,
\item the adapted lattice is\footnote{A more obvious choice of an adapted lattice would be given by $\Gamma' := 2\ZZ \times \ZZ \times \ZZ$, but the above choice will be technically convenient for us.
} $\Gamma:=\{(x,y,z)\in G \,|\, x,y,z\in 2\mathbb{Z}\}$,
\item and $V:=\big[-1,1)^3$.
\end{itemize}
\end{example}

While all of the results of this article are already new and of independent interest in the Heisenberg setting, that example does not fully reflect the generality of our framework. To illustrate its broader scope, we now turn to a more generic example of a \(7\)-dimensional substitution group.

\begin{example}[An infinite family of $3$-step nilpotent substitution groups]
\label{Ex-3Step} 
For every $\mu \in \RR \setminus\{0,1\}$ there is a $1$-connected $3$-step nilpotent Lie group $G_\mu$ (corresponding to the Lie algebra (147E)${}_\mu$ from the table \cite{Gong98}) which is isomorphic to $\RR^7$ with multiplication given by
\[
\begin{pmatrix} x_1\\ x_2\\ x_3\\ x_4\\ x_5\\ x_6 \\ x_7 \end{pmatrix} \ast \begin{pmatrix} y_1\\ y_2\\ y_3\\ y_4\\ y_5\\ y_6 \\ y_7 \end{pmatrix} = \begin{pmatrix} x_1 + y_1\\ x_2 + y_2\\ x_3 + y_3\\ x_4 + y_4 + \frac 1 2 (x_1y_2-x_2y_1)\\ x_5 + y_5 + \frac 1 2 (x_2y_3 - x_3y_2) \\ x_6 + y_6 - \frac 1 2 (x_1y_3 - x_3y_1) \\ x_7 + y_7 + \frac \mu 2(x_2y_6 - x_6y_2) + \frac{1-\mu} 2 (x_3y_4-x_4y_3) \end{pmatrix}.
\]
It turns out that $(G_\mu)_{\mu \in \RR}$ is an infinite family; in fact we have $G_{\mu_1} \cong G_{\mu_2}$ if and only if 
\[
\frac{(1-{\mu_1} + {\mu_1}^2)^3}{{\mu_1}^2(1-{\mu_1})^2} = \frac{(1-{\mu_2} + {\mu_2}^2)^3}{{\mu_2}^2(1-{\mu_2})^2};
\]
Moreover, $G_\mu$ admits a lattice provided $\mu \in \QQ$, so we assume from now on that $\mu = p/q$ for some coprime $p \in \ZZ$ and $q \in \mathbb N$. To construct a dilation structure on $G_\mu$ we observe that there exists $\varepsilon$ is convex in the sense that $D_t(x) \ast D_{1-t}(y) \in A_\varepsilon$ holds for all $x,y \in A_\varepsilon$ and $t\in[0,1]$ (see \cite[Theorem~B]{HeSi90}). For any such $\varepsilon$ there then exists a unique left-invariant metric $d_\varepsilon$ on $G_\mu$ with underlying norm
\[
|x|_\varepsilon = \inf\{t >0 \mid D_{1/t}(x) \in A_\varepsilon\}.
\]
If we now define a family of automorphisms of $G_\mu$ by 
\[
D_\lambda(x_1, \dots, x_7) := (\lambda x_1, \lambda x_2, \lambda x_3, \lambda^2 x_4, \lambda^2 x_5, \lambda^2 x_6, \lambda^3 x_7),
\]
then the triple  $(G_\mu, d_\varepsilon, (D_\lambda)_{\lambda>0})$ is a dilation group. In fact, this example of a dilation group stands prototypical for a large class of examples known as \emph{homogeneous dilation groups}; see
Definition \ref{Def-HomDilGroup} below for the general definition. An an adapted lattice in $G_\mu$ is given by $\Gamma_\mu :=  \ZZ \times (2q\ZZ) \times (2q\ZZ) \times \ZZ^4$, and if we choose $V := \left[-\frac 1 2, \frac 1 2 \right) \times [-q, q)^2 \times  \left[-\frac 1 2, \frac 1 2 \right)^4$, then $\mathcal{D}_{\mu}:=(G_\mu, d_\varepsilon, (D_\lambda)_{\lambda > 0},\Gamma_\mu, V)$ is a dilation datum.
\end{example}

\begin{no}\label{EPG} In all of our examples there is usually a large freedom in choosing the left-invariant metric on $G$; in order to obtain uniquely ergodic substitution systems we will have to restrict this choice slightly. 
Following \cite[Section 4.4]{Nevo} we say that a metric $d$ on a lcsc group $G$ has \emph{exact polynomial growth} with respect to $d$ of the limit $\lim_{r \to \infty} \frac{m_G(B(e,r))}{r^{\kappa}}$ exists for some $\kappa \geq 0$; here $B(e,r)$ denotes the open ball or radius $r$ around $e$ with respect to $d$ and $m_G$ denotes some choice of Haar measure on $G$.  We say that a dilation datum  $\Dd = (G, d, (D_\lambda)_{\lambda >0}, \Gamma, V)$ has \emph{exact polynomial growth} if both $d$ and the restriction $d|_{\Gamma \times \Gamma}$ have exact polynomial growth. As explained in \cite{Nevo}, exact polynomial growth is important in the context of pointwise ergodic theorems, and as explained in the companion paper \cite{BHP25-LR}, it can be used to deduce unique ergodicity from linear repetitivity. 
\end{no}
\subsection{Substitution groups} In this subsection we are concerned with the classification of lcsc groups which admit dilation data. For ease of reference we introduce the following terminology.
\begin{definition}
An lcsc group $G$ is called a \emph{substitution group} if it is the underlying group of a dilation datum.
\end{definition}
We have seen in Example \ref{Ex-EuclideanIntro}, Example \ref{Ex-HeisenbergIntro} and Example \ref{Ex-3Step} that the abelian groups $\RR^n$, the $3$-dimensional Heisenberg group $\mathbb{H}_3(\RR)$ and the groups $G_\mu$ from Example \ref{Ex-3Step} with $\mu\in \QQ$ are substitution groups. All of these examples share a lot of structural features. For example, they are nilpotent Lie groups and diffeomorphic to $\RR^n$ with multiplication given by polynomials; this is not by accident. Indeed, combining a classical theorem of Siebert \cite{Siebert86} with more recent work of Cornulier \cite{Co1} one can obtain the following characterizations of substitution groups:
\begin{theorem}[Cornulier--Siebert]\label{CornulierSiebert} For an lcsc group $G$ the following are equivalent:
\begin{enumerate}[(i)]
\item $G$ is a substitution group.
\item $G$ is the underlying group of a dilation group and contains a lattice.
\item $G$ admits a dilation datum of exact polynomial growth.
\item $G$ is a connected and simply-connected Lie group whose Lie algebra $\L g$ admits a derivation with positive eigenvalues and a basis with rational structure constants.
\end{enumerate}
Moreover, if $G$ satisfies these conditions, then there exists a homeomorphism $G \cong \RR^n$ which intertwines the group multiplication with a multiplication map $\RR^n \times \RR^n \to \RR^n$ whose coordinate functions are polynomials.
\end{theorem}
\begin{no}\label{DInclusionProper2}
The proof of Theorem \ref{CornulierSiebert} uses very different methods than most of the rest of this article (in particular, Lie theory and the theory of linear algebraic groups) and is thus deferred to Section~\ref{Sec-Lie}; see Theorem~\ref{Thm-GoodGroups} (in combination with Corollary~\ref{ExExPolGr}) for a more detailed version and Section~\ref{Sec-Lie} for proofs of the following facts: Condition (iv) implies that every substitution group admits a diffeomorphism to $\RR^n$ which intertwines the group multiplication with a multiplication map $\RR^n \times \RR^n \to \RR^n$ whose coordinate functions are polynomials. In particular, every substitution group is non-compact, which by \S \ref{DInclusionProper} implies that for any dilation datum $\Dd = (G, d, (D_\lambda)_{\lambda >0}, \Gamma, V)$ over $G$ and all $\lambda > 1$ we have $D_\lambda(\Gamma) \neq \Gamma$. In standard Lie theoretic terms\footnote{In a previous version of this article, we referred to simply-connected Lie groups with such a Lie algebras as RAHOGRASPs (rationally homogeneous groups with rational spectrum), but this seems to be at odds with standard Lie theoretic terminology, so we do not use this terminology in the present version.}, the conditions on the Lie algebra in (iv) say that $\L g$ is positively gradable and admits a $\QQ$-structure. See again Section \ref{Sec-Lie} for details.
\end{no}
\begin{remark}[Census] 
To understand the scope of our setting, one has to figure out how many substitution groups there are. In dimensions $\leq 3$ there is only a single non-abelian substitution group up to isomorphism, namely the Heisenberg group, so our setting seems rather restrictive. However, the picture changes drastically if we increase the dimension slightly. To give some impression, let us call a substitution group \emph{irreducible} if it does not split as a non-trivial product of substitution groups; then every substitution group is a product of irreducibles, and $\RR$ is the only abelian irreducible substitution group.

The number of isomorphism classes of irreducible substitution groups in dimensions $1$ to $6$ is given by $1$, $0$, $1$, $1$,  $6$ and $24$ respectively. Starting from dimension $7$, there are infinitely many isomorphism classes of irreducible substitution group. More precisely, it follows from the classification of maximal gradings of $7$-dimensional nilpotent irreducible Lie algebras in dimension $7$ due to Carles \cite{Carles1, Carles2} that 
up to isomorphism there are precisely $7$ infinite families of $7$-dimensional irreducible substitution groups depending on a rational parameter (including the one from Example \ref{Ex-3Step}) and $126$ further irreducible $7$-dimensional substitution groups not belonging to any of these families (see Appendix \ref{AppendixCensus} for details). In dimension $\geq 8$ the number of (families of) substitution groups explodes even more, so much so that any attempt of classification becomes unfeasible. In conclusion we can say that in moderately large dimensions there are hundreds of examples of substitution groups, some depending on parameters, some isolated.
\end{remark}

\subsection{Substitution data} From now on we fix a dilation datum \(\Dd = (G, d, (D_\lambda)_{\lambda >0}, \Gamma, V)\). We say that a $\lambda_0 > 1$ is a \emph{stretch factor} for $\Dd$ if $D_{\lambda_0}(\Gamma) \subset \Gamma$. By definition, every dilation datum admits a stretch factor, and if $\lambda_0$ is a stretch factor then so is $\lambda_0^n$ for any $n \in \mathbb N$; in particular, $\Dd$ admits arbitrarily large stretch factors. Not all stretch factors can be used to produce substitutions with non-empty substitution space; this is a subtlety that requires additional care.
\begin{no}
Given a stretch factor $\lambda_0 > 1$ we define sets $V_{\lambda_0}(n)$ inductively by
\begin{equation}\label{Vlambda0}
V_{\lambda_0}(0) := V \quad \text{and} \quad V_{\lambda_0}(n) := D_{\lambda_0}\big((V_{\lambda_0}(n-1)\cap\Gamma)V\big).
\end{equation}
If $\lambda_0$ is clear from context, then we also write $V(n)$ instead of $V_{\lambda_0}(n)$. The geometric meaning of these sets will become clear in Proposition~\ref{Prop-SupportFormula} below: Their intersections with $\Gamma$ describe the supports of patches obtained by iterating a substitution map with underlying fundamental domain $V$ and stretch factor $\lambda_0$ with starting seed supported at $\{e\}$. To obtain a non-empty substitution space we need these supports to grow in a controlled way; we thus make the following definition.
\end{no}
\begin{definition}\label{Def-stretch-factor} We say that a stretch factor $\lambda_0$ is \emph{$V$-sufficient} if there exist a constant \( C_- > 0 \), an integer \( s \in \NN_0 := \NN \cup \{0\} \), and \( z \in \Gamma \) such that for all \( n \in \NN_0 \), 
\[
D_{\lambda_0}^n\big( B(z, C_-) \big) = B\big( D_{\lambda_0}^n(z), C_-\lambda_0^n) \big) \subseteq V_{\lambda_0}(s+n).
\]
\end{definition}
\begin{no} In the situation of Definition \ref{Def-stretch-factor} we also say that $\lambda_0$ is $V$-sufficient with \emph{parameters} $(C_-, s, z)$. If $V$ is clear from context we simply say that $\lambda_0$ is \emph{sufficient}.
We will see in Proposition~\ref{Prop-SupportGroth}  below that the condition of $\lambda_0$ being $V$-sufficient indeed guarantees that images of finite patches under iterates of any substitution map with underlying fundamental domain $V$ and stretch factor $\lambda_0$ contain arbitrary large balls. For now let us explain how this condition can be checked in praxis.

We first consider the case where $V$ contains an identity neighbourhood. In this case there exist real numbers $r_+ \geq r_- > 0$ such that
\[
B(e, r_-) \subseteq V \subseteq B(e, r_+);
\]
we refer to $r_-$ as an \emph{inner radius} and to $r_+$ as an \emph{outer radius} for $V$. We then say that $\lambda_0$ is \emph{sufficiently large} (relative to V) if
\begin{equation}
 \lambda_0 > 1 + \tfrac{r_+}{r_-}
\end{equation} 

\end{no}
\begin{proposition}
\label{Prop-SuffCriter:suff_large-old}
If $V$ contains an identity neighbourhood and has inner radius $r_-$ and outer radius $r_+$, then every sufficiently large 
$\lambda_0$ is $V$-sufficient with parameters
\[
C_- = r_- \left( \lambda_0 - \left(1 + \tfrac{r_+}{r_-} \right) \right), \qquad s = 1 \qand z = e.
\]
\end{proposition}
\begin{no}\label{ExistenceLargeSF} Note that every substitution group admits a dilation datum $\Dd$ such that $V$ contains an identity neighbourhood; it then follows from Proposition \ref{Prop-SuffCriter:suff_large-old} that this dilation datum admits a $V$-sufficient stretch factor. If $V$ is not chosen to contain an identity neighbourhood, then the proposition does not apply since there is no inner radius for $V$. However, the outer radius $r_+$ can still be defined, and we obtain:
\end{no}
\begin{proposition}
\label{Prop-SuffCriter:suff_large}
If $r_+$ is an outer radius for $V$ and if there exist \( r_0 > \tfrac{r_+ \lambda_0}{\lambda_0 - 1} \), \( s_0 \in \NN_0 \), and \( z_0 \in \Gamma \) such that \( B(z_0, r_0) \subseteq V(s_0) \), then \( \lambda_0 \) is $V$-sufficient, with parameters
\[
C_- = (r_0 - r_+) - \tfrac{r_0}{\lambda_0}, \qquad s=s_0, \qquad z=z_0.
\]
\end{proposition}

We defer the proofs of Proposition~\ref{Prop-SuffCriter:suff_large-old} and Proposition~\ref{Prop-SuffCriter:suff_large} to Subsection \ref{SecRelativelyLarge} in favour of concrete examples. More precisely, we will construct sufficient stretch factors for the three examples considered above. These examples will show that even in the case where $V$ contains an identity neighbourhood the apparently more complicated
Proposition \ref{Prop-SuffCriter:suff_large} often produces better bounds than Proposition~\ref{Prop-SuffCriter:suff_large-old}, see the discussion in \cite{BaBePoTe25}. 

\begin{example} 
\label{Ex-EuclideanIntro-Subst}
Let $\Dd$ be the Euclidean dilation datum defined in Example~\ref{Ex-EuclideanIntro}. Then every $\lambda_0\in\NN$ with $\lambda_0\geq 2$ is a stretch factor, and we claim that it is $V$-sufficient. Our claim implies in particular that all block substitutions in the Euclidean space fall into our setting of substitution datum, e.g. the table and chair tiling substitution.
 
To prove the claim we observe that $V := [-\frac 1 2, \frac 1 2)^n$ has inner radius $r_-=\tfrac{1}{2}$ and outer radius $r_+=\tfrac{1}{\sqrt{2}} + \tfrac{1}{4}$; in view of  Proposition~\ref{Prop-SuffCriter:suff_large-old} this implies the claim for $\lambda_0 \geq 3$.
The case $\lambda_0=2$ requires additional efforts and can be established using Proposition~\ref{Prop-SuffCriter:suff_large}, see \cite[Prop.~3.1]{BaBePoTe25} for details.
\end{example}
\begin{example} Let $\Dd$ be the dilation datum for the Heisenberg group defined in Example \ref{Ex-HeisenbergIntro}. Then every $\lambda_0 \in \NN$ with $\lambda_0 \geq 2$ is a stretch factor for $\Dd$. Moreover, \( V = [-1,1)^3 \) has inner radius $r_- = 1$ and outer radius $1.5$ (since $\sqrt[4]{5} < 1.5$). Proposition~\ref{Prop-SuffCriter:suff_large-old} thus implies that any $\lambda_0 \geq 3$ is sufficient. Again it is possible to improve this to $\lambda_0 \geq 2$ by a direct computation using Proposition~\ref{Prop-SuffCriter:suff_large} and computer assistance, see~\cite{Bec21}. 
\end{example}

\begin{example} Let $\Dd$ be the dilation datum from Example \ref{Ex-3Step} for some fixed rational parameter $\mu = p/q$. Again, every $\lambda_0 \in \NN$ with $\lambda_0 \geq 2$ is a stretch factor for $\Dd$ and since $B(0,r_-)\ \subset V \subset B(0,r_+)$ with $r_- = \frac{1}{\sqrt[3]{2 \varepsilon}}$ and $r_+=\sqrt{2}\frac{q+1}{\varepsilon}$, it follows from Proposition~\ref{Prop-SuffCriter:suff_large-old} that any $\lambda_0 \geq 1+ 2^{5/6}(q+1) \cdot \varepsilon^{-\frac{2}{3}}$ is sufficiently large.
\end{example}

\begin{definition}\label{DefSubstitutionDatum} 
A \emph{substitution datum} $\mathcal S = (\mathcal{A}, \lambda_0, S_0)$ over $\Dd$ consists of 
\begin{itemize}
\item[(S1)] a finite set $\mathcal A$ called the \emph{underlying alphabet};
\item[(S2)] a $V$-sufficient stretch factor $\lambda_0$, called the \emph{underlying stretch factor};
\item[(S3)] a map $S_0: \mathcal A \to \mathcal A^{D_{\lambda_0}(V) \cap \Gamma}$ called the \emph{underlying substitution rule}.
\end{itemize}
\end{definition}
\begin{remark} If we fix a substitution group $G$ and an alphabet $\mathcal A$, then we can always find a dilation datum $\Dd$ over $G$ such that $V$ contains an identity neighbourhood. In view of Proposition~\ref{Prop-SuffCriter:suff_large-old} we can then find a substitution datum $\Ss$ over $\Dd$. Thus every substitution group admits a substitution datum, justifying the terminology. Writing out explicit substitution rules in high dimensions by hand is rather painful and better left to a computer. We thus confine ourselves to a single non-abelian example in the context of the Heisenberg group. 
\end{remark}
\begin{example}
\label{Ex-HeisenbergSubstitution}
Let $\mathcal{D}_{\mathbb{H}}:=(\mathbb{H}_3(\RR), d, (D_\lambda)_{\lambda > 0},\Gamma, V)$ be the dilation datum over the $3$-dimensional Heisenberg group from Example~\ref{Ex-HeisenbergIntro}. We define a substitution datum over $\mathcal{D}_{\mathbb{H}}$ with underlying alphabet $\mathcal{A} := \{a, b\}$ and stretch factor $\lambda_0 := 3$. First, observe that
\[
D_{\lambda_0}(V) \cap \Gamma = \{-2, 0, 2\} \times \{-2, 0, 2\} \times \{-8, -6, -4, -2, 0, 2, 4, 6, 8\}.
\]
and if we define a substitution rule $S_0:\Aa\to\Aa^{D_{\lambda_0}(V)\cap\Gamma}$ as in Figure~\ref{Fig-Heisenberg_S0}, then $(\Aa,3,S_0)$ defines a substitution datum over $\mathcal{D}_{\mathbb{H}}$.
\begin{figure}[htb]
\includegraphics[scale=0.8]{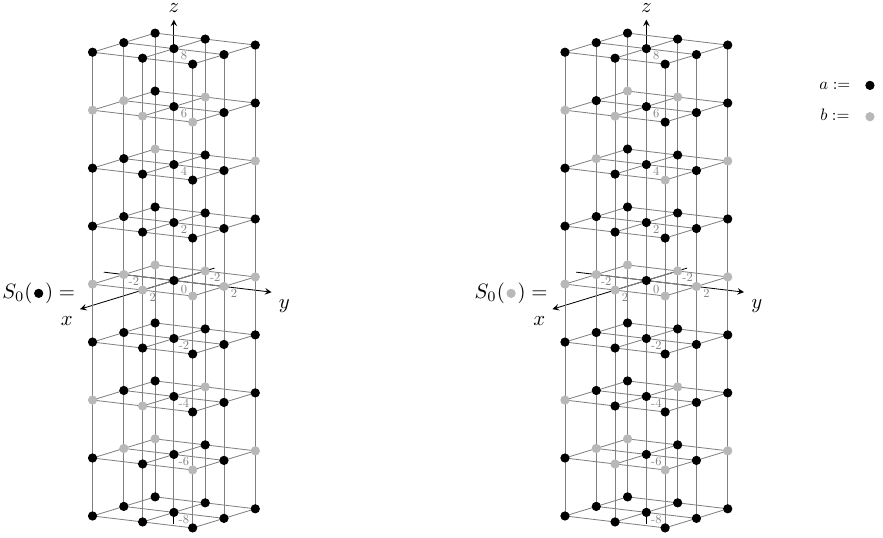}
\caption{An example of a substitution rule in the Heisenberg group with $\lambda_0=3$.}
\label{Fig-Heisenberg_S0}
\end{figure}

\end{example}
\begin{remark} 
Let us note that our framework is also of interest in the abelian case. 
Proposition~\ref{Prop-SuffCriter:suff_large}, together with Proposition~\ref{Prop-LegalFixpoints}, provides a verifiable criterion for the existence of a non-empty subshift -- a subtle point in the context of digit substitutions, see e.g.~\cite{FraMa22}. 
Digit substitutions arising from dilation (in particular diagonal) expansion maps, with $D_{\lambda_0}(V) \cap \Gamma$ used as representatives of $\Gamma / D_{\lambda_0}(\Gamma)$, fall within our setting; see~\cite{BeckusHabil25} for concrete examples. 
In general, it is delicate to determine which digit substitutions yield non-empty subshifts, but for those covered by our framework this issue is resolved in Proposition~\ref{Prop-LegalFixpoints}.
\end{remark}

\begin{example} 
Let \( \Dd := (\RR^2, d_\infty, (D_\lambda)_{\lambda > 0}, \ZZ^2, V) \) be the dilation datum, where  
\( d_\infty(x,y) := \max_{j=1,2} |x_j - y_j| \),  
\( D_\lambda(x,y) := (\lambda x, \lambda y) \),  
and 
\[
V := -\left(\tfrac{1}{4},\tfrac{1}{4}\right) 
   + \Big( [0,1)^2 \setminus [\tfrac12,1)^2 \;\cup\; [-\tfrac12,0)^2 \Big).
\]
Note that \( V \subseteq B(e,r_+) \) for \( r_+ = 1 \).  
A short computation shows that \( \lambda_0 = 2 \) is \(V\)-sufficient by Proposition~\ref{Prop-SuffCriter:suff_large}, with  
\( r_0 := 2.1 > \tfrac{r_+\lambda_0}{\lambda_0 - 1} \),  
\( z_0 = (-6,-6) \),  
and \( s_0 = 4 \), see also \cite[Sec.~3.5]{BeckusHabil25}.

For this dilation datum \( \Dd \) and $\lambda_0=2$, we have 
\[
V(1) \cap \ZZ^2 = \{(-1,-1),\; (0,0),\; (1,0),\; (0,1)\}.
\]
Now consider the substitution datum  
\( \Ss = (\mathcal A, \lambda_0, S_0) \) with alphabet \( \mathcal A=\{a,b\} \), \( \lambda_0 = 2 \), and substitution rule
\[
S_0(a) := 
\begin{array}{c|cc}
   & a &   \\
   & a & a \\
\hline
 b &   &   
\end{array},
\qquad
S_0(b) := 
\begin{array}{c|cc}
   & b &   \\
   & b & b \\
\hline
 a &   &   
\end{array}.
\]
Together with Proposition~\ref{Prop-LegalFixpoints} the previous considerations imply that the associated subshift $\Omega(S)$ defined in Section~\ref{Sec-ExisLegalFixPoints} is non-empty and all elements are linearly repetitive by Theorem~\ref{Thm-LinearRep}. 

This example illustrates the use of Proposition~\ref{Prop-SuffCriter:suff_large}, which provides a verifiable criterion for a substitution datum to give rise to a non-empty subshift with linear repetitive configurations. This is in general a subtle issue, see e.g.\@ the discussion in \cite{FraMa22}.
\end{example}

\subsection{Criteria for sufficient stretch factors}\label{SecRelativelyLarge}
In this subsection we carry out the proofs of Proposition \ref{Prop-SuffCriter:suff_large-old} and Proposition \ref{Prop-SuffCriter:suff_large}; we use this to introduce some general terminology to be used later. Throughout this subsection we fix a dilation datum $\Dd=(G, d_G, (D_\lambda)_{\lambda > 0}, \Gamma, V)$.
\begin{no}
Given a bounded subset $B \subset G$ we define its \emph{inner $\Gamma$-approximation} as 
\begin{equation}\label{InnerGammaApproximation}
B_\Gamma := (B \cap \Gamma)V.
\end{equation}
In other words, $B_\Gamma$ is obtained from $B$ by first passing to the finite sets of lattice points inside $B$ and then thickening the resulting finite set by $V$. It is thus a finite union of translates of $V$. With this notation we can then rewrite the definition of the sets $V_{\lambda_0}(n)$ as
\[
V_{\lambda_0}(0) = V_{\lambda_0} \qand V(n) = D_{\lambda_0}(V_{\lambda_0}(n-1)_\Gamma).
\]
For the remainder of this section we fix a stretch factor $\lambda_0$ and set $D := D_{\lambda_0}$ and $V(n) := V_{\lambda_0}(n)$.
\end{no}
\begin{no} For the proofs of Proposition \ref{Prop-SuffCriter:suff_large-old} and Proposition \ref{Prop-SuffCriter:suff_large} we need to estimate the size of inner approximations. Intuitively, if $B$ is a large open ball around the identity, then $B_\Gamma$ is very close to $B$. At the other extreme, if $B$ contains no lattice points, then $B_\Gamma$ is empty. To make this quantitative we 
introduce the following notation: given a bounded subset $B \subset G$ and $r>0$ we define  
\[
B_{+r} := \bigcup_{x \in B} B(x,r) \qand B_{-r} := \{x \in B \mid B(x,r) \subset B\}.
\]
Note that for all $g \in G$ and $r', r > 0$ we have 
\begin{equation}
\label{eq:Inner-outer_approx-balls}
B(g,r')_{+r}\subseteq B(g,r'+r)
\quad\text{and if } r' > r, \text{ then }\quad
B(g,r'-r)\subseteq B(g,r')_{-r}.
\end{equation}
\end{no}
\begin{lemma}[Size of inner approximation]
\label{Lem-NaiveGrowth} 
If $r_+$ denotes an outer radius for $V$, then
\[B_{-r_+} \subset B_\Gamma \subset B_{+r_+} \quad \text{for all $B \subset G$}.\]
\end{lemma}
\begin{proof} Every $x \in G$ is contained in a set of the form $\gamma V$ for a unique $\gamma \in \Gamma$. If $v := \gamma^{-1}x \in V$, then $d(x, \gamma) = d(\gamma v, \gamma)= d(v,e) < r_+$ and thus $\gamma \in B(x, r_+)$. Now if $x \in B_{-r_+}$, then $B(x, r_+) \subset B$ and hence $\gamma \in B \cap \Gamma$. This implies $x \in B_\Gamma$  and shows the first inclusion. The second inclusion is immediate from $V \subset B(e, r_+)$.
\end{proof}

\begin{proof}[Proof of Proposition \ref{Prop-SuffCriter:suff_large-old}]
Since \( \lambda_0 > 1 + \tfrac{r_+}{r_-} \), we conclude \( \lambda_0 r_- > r_+\). Let \(r:= \lambda_0 r_-\) and $\varepsilon:=\frac{\lambda_0(r-r_+)}{r} -1>0$.
Then $\lambda_0(r-r_+) = (1+\varepsilon) r$. We prove for all $k \in \NN$ that
\begin{equation}\label{IndOldSL}
B(e, \lambda_0^{k-1}\varepsilon r + r) \subset V(1+k).
\end{equation}
For $k=1$, the inclusion $B(e, r_-) \subseteq V$ leads to
\[
B(e,r) = D\big(B(e,r_-)\big) \subseteq D(V)=V(1).
\]
Thus, Lemma~\ref{Lem-NaiveGrowth} imply
\[
B(e, r-r_+) \subseteq V(1)_\Gamma 
	\implies B\big(e,\lambda_0(r-r_+)\big) = D(B(e, r-r_+))  \subseteq D(V(1)_\Gamma) = V(1+1).
\]
Since $\lambda_0(r-r_+) = (1+\varepsilon)r = \lambda_0^{1-1}\varepsilon r +r$, the induction base case is proven. Then we see by the induction hypothesis that
\begin{align*}
B\big(e,\lambda_0^{k-1}\varepsilon r + r\big) \subseteq V(n_0+k) \implies &B\big(e,\lambda_0^{k-1}\varepsilon r +r - r_+\big) \subseteq V(n_0+k)_\Gamma \\
	\implies &B\big(e,\lambda_0^{k}\varepsilon r + \lambda_0 r - \lambda_0 r_+\big) \subseteq V(n_0+k+1)
\end{align*}
while $\lambda_0 r - \lambda_0 r_+ = (1+\varepsilon) r \geq  r$. This finishes the induction and establishes \eqref{IndOldSL}.

\item Now a short computation gives
\[
\frac{\varepsilon r}{\lambda_0} = (r-r_+)-\frac{r}{\lambda_0}
	= C_-
	\implies B(e,C_-\lambda_0^k)\subseteq B(e, \lambda_0^{k-1}\varepsilon r + r)\subset V(1+k).
\]
This proves that $\lambda_0$ is sufficiently large relative to $V$ with parameters $C_-$, $s=1$ and $z=e$. 
\end{proof}

\begin{proof}[Proof of Proposition \ref{Prop-SuffCriter:suff_large}]
Let \( r_0 > \tfrac{r_+ \lambda_0}{\lambda_0 - 1} \), \( s_0 \in \NN_0 \), and \( z_0 \in \Gamma \) be such that \( B(z_0, r_0) \subseteq V(s_0) \). Since
\begin{equation}
\label{eq:proof_suff-large}
r_0 > \tfrac{r_+ \lambda_0}{\lambda_0 - 1} 
	\quad\Longleftrightarrow\quad
(r_0 - r_+)\lambda_0 > r_0
	\quad\Longleftrightarrow\quad	
\frac{r_0 - r_+}{r_0} > \frac{1}{\lambda_0},
\end{equation}
we have \( C_- > 0 \). Let \( \varepsilon = \tfrac{C_-}{r_0} > 0 \), which satisfies \(\varepsilon + \tfrac{1}{\lambda_0} = \tfrac{r_0 - r_+}{r_0}\).

We first prove inductively that
\[
\text{for all } k \in \NN: \quad
B\big(D^k(z_0),\, \lambda_0^k \varepsilon r_0 + r_0 \big) \subseteq V(s_0 + k).
\]

Let $k=1$. Lemma~\ref{Lem-NaiveGrowth} and the inclusions~\eqref{eq:Inner-outer_approx-balls} and \( B(z_0, r_0) \subseteq V(s_0) \) lead to
\[
B(z_0, r_0 - r_+) \subseteq B(z_0, r_0)_{-r_+} \subseteq V(s_0)_\Gamma.
\]
Applying the dilation \( D \) gives
\begin{align*}
B\big(D(z_0),\, \lambda_0 \varepsilon r_0 + r_0 \big) 
	%= B\big(D(z_0),\, \lambda_0(r_0 - r_+)\big) 
	= D\big(B(z_0, r_0 - r_+)\big) 
	\subseteq D\big(V(s_0)_\Gamma\big) = V(s_0 + 1). 
\end{align*}
Suppose now the claimed inclusion holds for some \( k \in \NN \). Equation \eqref{eq:proof_suff-large}, Lemma~\ref{Lem-NaiveGrowth} and the recursive definition of \( V(n) \) imply
\begin{align*}
B\big(D^{k+1}(z_0),\, \lambda_0^{k+1} \varepsilon r_0 + r_0 \big)
	&\subseteq B\big(D^{k+1}(z_0),\, \lambda_0^{k+1} \varepsilon r_0 + (r_0 - r_+)\lambda_0 \big) \\
	&= D\big(B\big(D^k(z_0),\, \lambda_0^k \varepsilon r_0 + r_0 - r_+ \big)\big) \\
	&\subseteq D\big(V(s_0 + k)_\Gamma\big) = V(s_0 + k + 1).
\end{align*}
This completes the induction. To finish the proof, observe that since \( D \) is a dilation,
\[
D^k\big(B(z_0, C_-)\big) 
	= B\big(D^k(z_0),\, \lambda_0^k \varepsilon r_0 \big)  
	\subseteq B\big(D^k(z_0),\, \lambda_0^k \varepsilon r_0 + r_0 \big) 
	\subseteq V(s_0 + k)
\]
proving the statement.
\end{proof}

\section{The substitution map}
\label{Sec-SubstMap}
In this section we discuss the construction of substitution maps from a given dilation and substitution datum. In fact, we extend any given substitution rule to a map on all patches and configurations. 
Throughout this section we fix a dilation datum  $\Dd = (G, d, (D_\lambda)_{\lambda >0}, \Gamma, V)$ and a substitution datum $\mathcal S = (\mathcal{A}, \lambda_0, S_0)$ over $\Dd$. We then abbreviate $D := D_{\lambda_0}$ and note that by \S \ref{DInclusionProper2} we have $D(\Gamma) \subsetneq \Gamma$.

\subsection{Existence and uniqueness}\label{Sec-SubDat-to-SubMap}
In order to state the axioms for a substitution map, we need the following notation.
\begin{no}
Given a subset $M \subset \Gamma$, we refer to an element $P \in \mathcal A^M$ as a  \emph{patch} with \emph{support} $\mathrm{supp}(P) := M$. A patch $P$ is called \emph{finite} if $\mathrm{supp}(P)$ is finite. Denote the space of all patches, respectively all finite patches by
\[
\mathcal A^{**}_\Gamma :=  \bigsqcup_{M \subset \Gamma} \mathcal A^M \qand \mathcal A^*_{\Gamma} := \bigsqcup_{M \subset \Gamma \text{ finite }} \mathcal A^M \subset \mathcal A^{**}_\Gamma.
\]
Then $\Gamma$ acts on $\mathcal A^{**}_\Gamma$ by
\[
\mathrm{supp}(\gamma.P) := \gamma.\mathrm{supp}(P) \qand \gamma.P(x) := P(\gamma^{-1}x), \qquad \gamma \in \Gamma,\, P \in \mathcal A^{**}_\Gamma
\]
preserving the subspaces $ \mathcal A^*_{\Gamma} \subset \mathcal A^{**}_\Gamma$ and  $\mathcal A^\Gamma \subset \mathcal A^{**}_\Gamma$; the induced action on the latter is just the action from the introduction.
Given $a \in \mathcal A$ we denote by $P_a$ the patch with 
\[
\mathrm{supp}(P_a) = \{e\}
\quad\textrm{and}\quad 
P_a(e) = a.
\] 
Moreover, given $M \subset G$, and $P \in \mathcal A^{**}_{\Gamma}$, then we denote by $P|_M := P|_{\mathrm{supp}(P) \cap M}$ the \emph{restriction} of $P$ to $\mathrm{supp}(P) \cap M$. We say that a patch $P$ \emph{extends} a patch $Q$ if $P|_{\mathrm{supp}(Q)}=Q$. 
\end{no}
Our goal is to establish the following proposition:
\begin{proposition}[Axiomatic characterization of the substitution map]
\label{Prop-SubstitutionAxioms} Given $\Dd$ and $\Ss$, there is a unique map $S: \mathcal A^{**}_\Gamma \to \mathcal A^{**}_\Gamma$ with the following properties:
%The substitution map $S: \mathcal A^{**}_\Gamma \to \mathcal A^{**}_\Gamma$ is the unique map satisfying the following properties:
\begin{itemize}
\item[(E1)] $S(P_a) = S_0(a)$ for all $a \in \mathcal A$. %(``extension'')
\item[(E2)] For all $\gamma \in \Gamma$ and $P \in \mathcal A^{**}_\Gamma$ we have 
\[S(\gamma P) = D(\gamma)S(P).\] %(``dilation-equivariance'')
\item[(E3)] For all $P\in \mathcal A^{**}_\Gamma$ and $M\subseteq \supp(P)$ the patch $S(P)$ extends the patch $S(P|_M)$.
\end{itemize}
\end{proposition}
\begin{definition} The map $S$ from Proposition \ref{Prop-SubstitutionAxioms} is called the \emph{substitution map} associated with the pair $(\Dd, \mathcal S)$.
\end{definition}
\begin{no}\label{SUnique} To see that $S(P)$ is uniquely determined by (E1)--(E3) for all $P \in \mathcal A^{**}_\Gamma$, we observe that view of (E3) the patch $S(P)$ is determined by the restrictions $P|_{\{\gamma\}}$ with $\gamma \in \Gamma$, hence we may assume that $|\mathrm{supp}(P)| = 1$. Then $\mathrm{supp}(P)$ is a translate of $\{e\}$ and hence $S(P)$ is uniquely determined by (E2) and (E1). It remains to show that $S_0$ can be extended according to (E1)--(E3) without running into contradictions. For this we will provide an explicit construction after some preliminaries.
\end{no}
\begin{no} The fact that $V$ is a fundamental domain for $\Gamma$ means that
\[
G = \bigsqcup_{\gamma \in \Gamma} \gamma V.
\]
Since  for every $\gamma \in \Gamma$ we have $\gamma V \cap \Gamma = \{\gamma\}$, this implies
\begin{equation}
MV \cap \Gamma = M \quad \text{for every }M \subset \Gamma. 
\end{equation}
Moreover, since $D$ is an automorphism, we have
\begin{equation}\label{FundDomainIterate}
G = D^n(G) = \bigsqcup_{\gamma \in \Gamma} D^n(\gamma) D^n(V).
\end{equation}
In particular, $D^n(V)$ is a fundamental domain for $D^n(\Gamma)$. Finally, for all $M \subset \Gamma$ we have
\begin{equation}\label{PreExtension}
D(MV) \cap \Gamma = \bigsqcup_{\eta \in M} \left( D(\eta)D(V)\right) \cap \Gamma = \bigsqcup_{\eta \in M} D(\eta)(D(V) \cap \Gamma).
\end{equation}
Here the last equality follows from $D(\Gamma)\subseteq \Gamma$ and the fact that $(\gamma A)\cap \Gamma = \gamma (A \cap \Gamma)$ for every $\gamma \in \Gamma$ and $A \subset G$.
\end{no}
\begin{construction}[Explicit substitution map]\label{Const-ExtensionRule}  We now construct an explicit map $S: \mathcal A^{**}_\Gamma \to \mathcal A^{**}_\Gamma$ for which we will verify the axioms (E1)--(E3).
Given a patch $P$ with support $M$ we define a new patch $S(P)$ as follows:
\begin{enumerate}
\item[(a)] The support of $S(P)$ is $\mathrm{supp}(S(P)) := D(MV) \cap \Gamma$.
\item[(b)] Now let $\gamma \in D(MV) \cap \Gamma$. By \eqref{PreExtension} there exists a unique $\eta \in M$ such that
\[
\gamma \in D(\eta)(D(V) \cap \Gamma).
\]
\item[(c)] If $\eta\in M$ is as in (b), then $P(\eta) \in \mathcal A$ and $D(\eta)^{-1}\gamma \in D(V) \cap \Gamma$, hence we may define
\[
S(P)(\gamma) := S_0(P(\eta))(D(\eta)^{-1}\gamma).
\]
\end{enumerate}
\end{construction}
\begin{proof}[Proof of Proposition \ref{Prop-SubstitutionAxioms}] Uniqueness was established in \S \ref{SUnique}, and we claim that the map $S$ from Construction \ref{Const-ExtensionRule} satisfies (E1)--(E3):

\item (E1) This is immediate from the construction.

\item (E2) Let $\gamma \in \Gamma$ and $P \in \mathcal A^{**}_\Gamma$. Set $M:= \mathrm{supp}(P)$ and note that $\mathrm{supp}(\gamma P) = \gamma M$. Since $D$ is an automorphism and $\Gamma$ is $D(\Gamma)\subseteq\Gamma$ we have
\begin{align*}
\mathrm{supp}\big( S(\gamma P) \big) 
	&= D(\gamma MV) \cap \Gamma 
	= D(\gamma) \big( D(MV) \cap \Gamma \big) 
	= D(\gamma)  \mathrm{supp}\big( S(P) \big),
\end{align*}
which shows the supports in (E2) coincide. Now let $\zeta\in \mathrm{supp}\left(S(\gamma P)\right) =  D(\gamma) D(MV) \cap \Gamma$ and set $\zeta' := D(\gamma^{-1})\zeta$. Then $\zeta' \in \mathrm{supp}(S(P))$ and  $S(P)(\zeta') = D(\gamma)S(P)(\zeta)$ follows. Thus, it suffices to prove that
\begin{equation}\label{E2ToShow}
S(\gamma P) (\zeta) = S(P)(\zeta').
\end{equation}
Let $\eta' \in M$ be the unique element such that $\zeta' \in D(\eta')(D(V) \cap \Gamma)$. Set $\eta := \gamma \eta' \in \gamma M$ so that $\zeta \in D(\eta)(D(V) \cap \Gamma)$.
By construction we then have
\[
S(\gamma P)(\zeta)= S_0(\gamma P(\eta))(D(\eta)^{-1}\zeta)\qand S(P)(\zeta') = S_0(P(\eta'))(D(\eta')^{-1}\zeta').
\]
Since $\gamma P(\eta) = P(\gamma^{-1}\eta) = P(\eta')$ and $D(\eta)^{-1}\zeta = D(\gamma \eta')^{-1} D(\gamma) \zeta' = D(\eta')^{-1}\zeta'$ this implies \eqref{E2ToShow} and finishes the proof of (E2). 

\item (E3) Let $P \in \mathcal{A}_{\Gamma}^{**}$ and $M\subseteq \operatorname{supp}(P)$.  By Construction~\ref{Const-ExtensionRule}~(a) and \eqref{PreExtension}, we have 
\[
\operatorname{supp}\big( S(P|_M) \big) =D(MV) \cap \Gamma
	= \bigsqcup_{\eta \in M} D(\eta)(D(V) \cap \Gamma).
\]
Let $\gamma \in \supp\big( S(P|_{M}) \big)$ and $\eta \in M \subseteq \supp(P)$ be such that $\gamma \in D(\eta)\big( D(V) \cap \Gamma \big)$. Then by construction we have $S(P)(\gamma) = S_0\big( P(\eta) \big) \big( D(\eta)^{-1} \gamma \big) = S\big( P|_{M} \big)(\gamma)$, which shows (E3).
\end{proof}
The following proposition collects some basic properties of the substitution map. Given $M \subset \Gamma$ we equip $\mathcal A^M$ with the product topology; then each $\mathcal A^M$ is a compact metrizable space.
\begin{proposition}
\label{Prop-SubstitutionProperties} 
The substitution map $S: \mathcal A^{**}_\Gamma \to \mathcal A^{**}_\Gamma$ satisfies the following properties.
\begin{enumerate}[(a)]
\item The restriction $S: \mathcal A^\Gamma \to \mathcal A^\Gamma$ is continuous. 
\item For each $P \in \mathcal{A}_{\Gamma}^{**}$ and all $n \in \NN$, we have $S^n(\gamma P) = D^n(\gamma) S^n(P)$.
\item If $(M_i)_{i \in I}$ is a family of subsets of $\Gamma$ and $M = \bigcup M_i$, then for all $P \in \mathcal A^M$ we have
\[
\mathrm{supp}(S(P)) = \bigcup_{i \in I}\mathrm{supp}(S(P|_{M_i})) \; \text{and} \; S(P)(\gamma) = S(P|_{M_i})(\gamma) \text{ for all }\gamma \in \mathrm{supp}(S(P|_{M_i})).
\]
\end{enumerate}
\end{proposition}

\begin{proof} 
(a) Let \( M \subseteq \Gamma \) be finite and fix an arbitrary patch \( P \in \mathcal{A}^M \). Define the open set
\(
U := U_P := \big\{ \omega \in \mathcal{A}^\Gamma \,:\, \omega|_M = P \big\}.
\)
By definition of the product topology on \( \mathcal{A}^{\Gamma} \), it suffices to show that \( S^{-1}(U) \) is open in \( \mathcal{A}^\Gamma \). Recall that
\(
\Gamma = \bigsqcup_{\eta \in \Gamma} \big( D(\eta) D(V) \cap \Gamma \big),
\)
so there exist \( m \in \mathbb{N} \) and pairwise distinct \( \eta_j \in \Gamma \), \( 1 \leq j \leq m \), such that
\[
M \subseteq \bigsqcup_{j=1}^m \big( D(\eta_j) D(V) \cap \Gamma \big).
\]
Set \( M' := \{ \eta_j : 1 \leq j \leq m \} \). For each \( \omega \in S^{-1}(U) \), define
\(
W_\omega := \big\{ \omega' \in \mathcal{A}^\Gamma \,:\, \omega'|_{M'} = \omega|_{M'} \big\}
\)
an open neighborhood of $\omega$. 
Clearly,
\[
S^{-1}(U) \subseteq \bigcup_{\omega \in S^{-1}(U)} W_\omega,
\]
and the right-hand side is open as a union of open sets. Thus, it remains to show that \( S(W_\omega) \subseteq U \) for all \( \omega \in S^{-1}(U) \).

\item Fix \( \omega \in S^{-1}(U) \) and let \( \omega' \in W_\omega \), i.e.\ \( \omega(\eta_j) = \omega'(\eta_j) \) for all \( 1 \leq j \leq m \). Then,
\[
S_0(\omega(\eta_j)) = S_0(\omega'(\eta_j)) \quad \text{for all } 1\leq j m.
\]
Now, let \( \gamma \in \bigsqcup_{j=1}^m D(\eta_j) D(V) \cap \Gamma \). Then there exists a unique \( i \in \{1, \dots, m\} \) such that \( \gamma \in D(\eta_i) D(V) \cap \Gamma \). By Construction~\ref{Const-ExtensionRule}, we have
\[
S(\omega)(\gamma) = S_0(\omega(\eta_i))\big( D(\eta_i^{-1}) \gamma \big) = S_0(\omega'(\eta_i))\big( D(\eta_i^{-1}) \gamma \big) = S(\omega')(\gamma).
\]
Since \( M \subseteq \bigsqcup_{j=1}^m D(\eta_j) D(V) \cap \Gamma \), it follows that \( S(\omega)|_M = S(\omega')|_M \), hence \( S(\omega') \in U \). Therefore, \( S(W_\omega) \subseteq U \), and we conclude that \( S^{-1}(U) \) is open.

\item (b) This follows immediately by induction from (E2) in Proposition~\ref{Prop-SubstitutionAxioms}.

\item (c) This follows immediately from (E3) in Proposition~\ref{Prop-SubstitutionAxioms}.
\end{proof}

\subsection{Support growth}\label{SecGrowth} From now on $S$ denotes the substitution map associated with $\Dd$ and $\Ss$. We will be concerned with the following problem:
\begin{figure}[hb]
\includegraphics[scale=0.45]{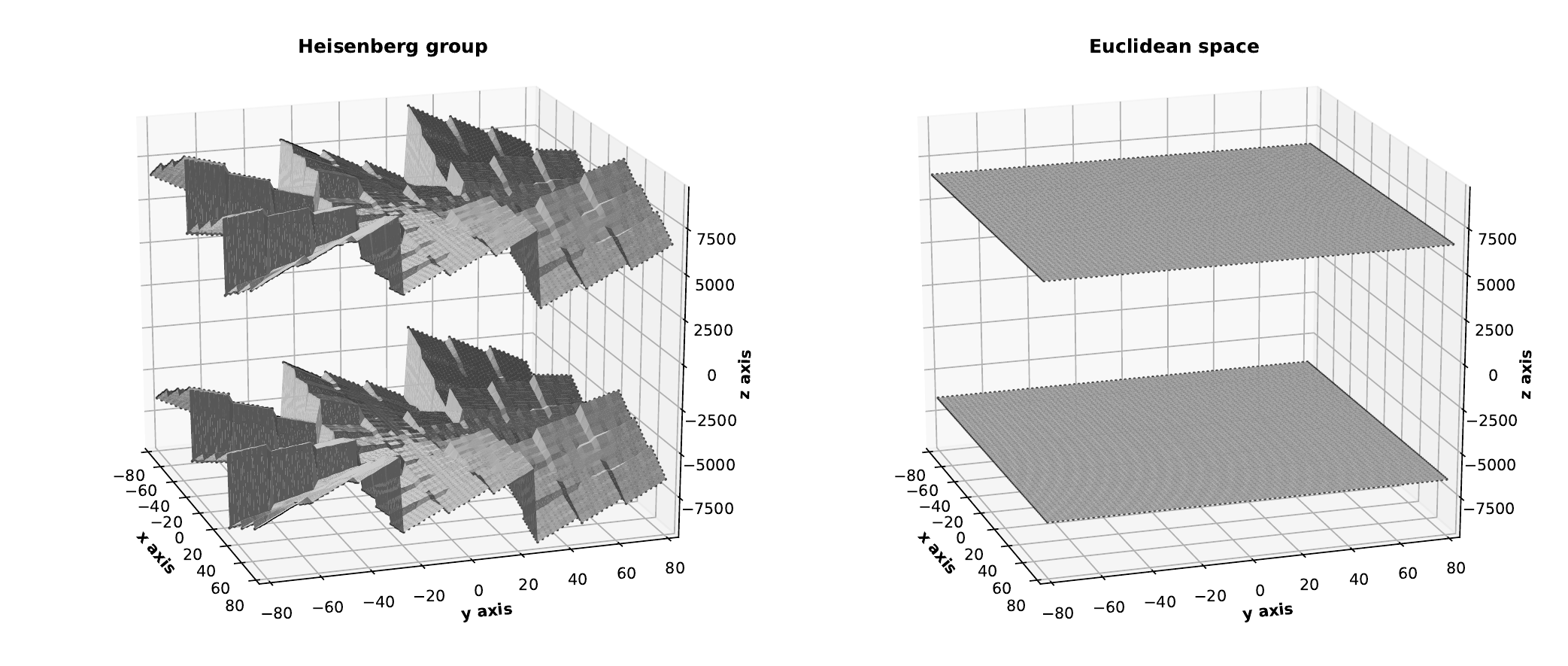}
\caption{Maximal and minimal $z$-values in the support of $S^4(P_a)$ for the ``same'' substitution in the Heisenberg group and Euclidean space.}
\label{Fig-Heisenberg_Supp}
\end{figure}
\begin{problem}[Support growth problem] 
Let \( P \in \mathcal{A}^*_\Gamma \) be a finite patch. What is the support of \( S^n(P) \)? In what sense does it grow as \( n \to \infty \)?
\end{problem}

To see what we can expect, let us consider an example:

\begin{example}[Euclidean vs.\ non-euclidean support growth]\label{Ex-HeisenbergSupport}
In Example~\ref{Ex-HeisenbergSubstitution}, we introduced a substitution datum $\Ss = (\Aa,3,S_0)$ over the Heisenberg dilation datum $\mathcal{D}_{\mathbb{H}}$. Due to the fact that the underlying space of $\mathbb H_3(\RR)$ is $\RR^3$ we can use ``the same'' rule to define a substitution over $\RR^3$. More precisely, consider the Euclidean dilation datum $\mathcal{D}_{\RR^3} := (\RR^3, d, (D_\lambda)_{\lambda > 0}, \Gamma_{\RR}, V)$, where $(\RR^3, d, (D_\lambda)_{\lambda > 0})$ is as in Example~\ref{Ex-EuclideanIntro} with $r_1 = r_2 = 1$ and $r_3 = 2$, and we take $\Gamma_{\RR} := (2\ZZ)^3$ and $V := [-1,1)^3$. Then $\Ss$ can be interpreted both as a substitution datum over $\mathcal{D}_{\mathbb{H}}$ and as a substitution datum over $\mathcal D_{\RR^3}$. Due to the differing geometry of Euclidean space and the Heisenberg group the resulting substitution maps $S_{\mathbb H}$ and $S_{\RR^3}$ are fundamentally different, and this is reflected in their support growth.

In the Euclidean setting, the support of $S_{\RR^3}^n(P_a)$ is given by the (discrete) cube
\[
([-\lambda_0^n,\lambda_0^n)^2 \times [-\lambda_0^{2n},\lambda_0^{2n})) \cap \Gamma,
\]
and these discrete cubes grow nicely as $n \to \infty$. To describe the picture in the Heisenberg case we consider the projection
\[
p_{12}: \Gamma \to (2\ZZ)^2, \; (x,y,z) \mapsto (x,y)
\]
Then the images $p_{12}(S_{\mathbb{H}}^n(P_a))$ are just given by the growing discrete squares $([-\lambda_0^{2n}, \lambda_0^{2n}) \cap 2\ZZ)^2$ as in the Euclidean case, but the fibers 
$F^n_{x,y} = S_{\mathbb{H}}^n(P_a) \cap p_{12}^{-1}(x,y)$ get shifted depending on the basepoint $(x,y)$, and these shifts lead to a more and more ragged boundary of $S_{\mathbb{H}}^n(P_a)$.
This difference is striking already for $n=4$; see Figure~\ref{Fig-Heisenberg_Supp} for a plot of the points of maximal in minimal $z$-value in each of the fibers  $F^n_{x,y}$ and the corresponding plot for the Euclidean case. These pictures seem to indicate that, despite the fact that growth is much more irregular in the Heisenberg setting, diameters of images of patches under iterations of the substitution map grow both in the Euclidean and in the Heisenberg group. We are going to establish this.
\end{example}

\begin{no}\label{VDef} To describe the support growth of general patches we need a generalization of the sets $V(n) = V_{\lambda_0}(n)$ introduced in \eqref{Vlambda0}. Given a finite subset $M \subset \Gamma$ we define subset $V(n, M) \subset G$ inductively
\[
V(0, M) := MV \qand V(n, M) := D(V(n-1, M)_\Gamma);
\]
in particular, $V(n, \{e\}) = V(n)$. 
\end{no}
\begin{proposition}[Support formula] 
\label{Prop-SupportFormula}
If $P \subset \Gamma$ is a finite patch with non-empty support $M$ 
\[
\mathrm{supp}(S^n(P)) = V(n, M) \cap \Gamma.
\]
In particular, for every $a \in \mathcal A$ we have $\mathrm{supp}(S^n(P_a)) = V_n \cap \Gamma$, and
\[
S^n(Q)|_{V(n,F)}=S^n(Q|_F)\qquad
\text{for all } Q\in\Aa^{**}_\Gamma \text { and } \emptyset\neq F\subseteq \supp(Q) \text{ finite}.
\]
\end{proposition}

\begin{proof} 
If $n=1$, then $\mathrm{supp}(S(P)) = D(MV)\cap\Gamma$ holds by definition. Since $e \in V$ we have $MV\cap\Gamma=M$ and hence $D(MV)=D((MV\cap\Gamma)V) = V(1,M)$ proving the induction base.
Since $S^{n+1}(P) = S\big( S^n(P)\big)$, we conclude by induction that
$$
\mathrm{supp}(S^{n+1}(P)) = D((V(n,M)\cap\Gamma)V)\cap\Gamma = V(n+1,M) \cap \Gamma,
$$
which proves the first statement. The second statement then follows since $\mathrm{supp}(P_a) = \{e\}$ and $V(n, \{e\}) = V(n)$, and
the equality $S^n(Q)|_{V(n,F)}=S^n(Q|_F)$ follows directly from (E3).
\end{proof}
By Proposition \ref{Prop-SupportFormula}, studying the support growth of $S$ thus reduces to the study of the growth of the subsets $V(n,M) \subset G$ which we refer to as \emph{support sets}. We thus record some elementary properties of these support sets:
\begin{lemma}\label{Lem-Prop-V(n,M)} 
Let $\gamma\in\Gamma$, $M, M' \subset \Gamma$ be finite non-empty sets and $n,m\in\NN$. Then the following hold:
\begin{itemize}
\item[(a)] $V(n,\gamma M)=D^n(\gamma)V(n,M)$,
\item[(b)] If $M \subset M'$, then $V(n,M)\subseteq V(n,M')$,
\item[(c)] $V(n,V(m,M) \cap \Gamma )=V(n+m,M)$,
\item[(d)] $V(n-1)\cap\Gamma\subseteq V(n)\cap\Gamma$.
\end{itemize}
\end{lemma}

\begin{proof}
Each of the statements follows by a simple induction argument. Here are the details. 

\item (a) For $n=1$, the identities 
$$
V(1,\gamma M) = D\big((\gamma M\cap\Gamma)V\big)
	= D(\gamma) D\big((M\cap\gamma^{-1}\Gamma)V\big)
	= D(\gamma) D\big((M\cap\Gamma)V\big)
$$
follow as $D$ is an automorphism and $\gamma\in\Gamma$. For general $n \in \NN$, we obtain
\begin{align*}
V(n+1,\gamma M) = &D\big((V(n,\gamma M)\cap\Gamma)V\big)
	= D\big((D^n(\gamma)V(n,M)\cap\Gamma)V\big)\\
	= &D^{n+1}(\gamma) D\big((V(n,M)\cap D^n(\gamma)^{-1}\Gamma)V\big)
	= D^{n+1}(\gamma) D\big((V(n,M)\cap \Gamma)V\big)
\end{align*}
using the induction hypothesis, that $D$ is an automorphism and $D^n(\Gamma)\subseteq\Gamma$.

\item (b) The equalities $V(1,M)= D(MV)\subseteq D(M'V) = V(1,M')$ prove the case $n=1$. Then the induction hypothesis leads to
$$
V(n+1,M) = D\big((V(n,M)\cap\Gamma)V\big) \subseteq D\big((V(n,M')\cap\Gamma)V) = V(n+1,M').
$$

\item (c) We first prove the statement for $m=1$ by an induction over $n\in\NN$. For $n=1$, the equality 
$$
V(1,V(1,M)\cap\Gamma) = D\big((V(1,M)\cap\Gamma)V\big) = V(2,M)
$$
follows by definition. Then the induction hypothesis $V(n,V(1,M))\cap\Gamma=V(n+1,M)\cap\Gamma$ implies
$$
V\big(n+1,V(1,M)\cap\Gamma\big) = D\big(\big( V\big(n,V(1,M)\cap\Gamma\big)\cap\Gamma\big) V \big)
	= D\big( (V(n+1,M)\cap\Gamma) V\big)
	= V(n+2,M).
$$
Having the statement for $m=1$, we conclude the desired identity for any $m\in\NN$ by induction. Specifically, the induction step follows from
$$
V(n+m+1,M) = V\big(n+m, V(1,M)\cap\Gamma \big)
	= V\big(n, V\big(m,V(1,M)\cap\Gamma \big) \cap\Gamma \big)
	= V\big(n, V(m+1,M)\cap\Gamma \big).
$$

\item (d) For $n=1$, we get $V(1)\cap\Gamma = D\big((V\cap\Gamma)V\big)\cap\Gamma \supseteq \{e\} = V\cap\Gamma$. If it holds for $n \geq 1$, then
$$
V(n+1) \cap\Gamma = D\big( (V(n)\cap\Gamma)V\big) \cap\Gamma
	\supseteq D\big( (V(n-1)\cap\Gamma) V\big)\cap\Gamma
	= V(n)\cap\Gamma
$$
proves the desired result.
\end{proof}
\begin{no}We are going to compare the support sets to suitable balls in $G$; for this we need to fix a few parameters. Since $V$ is bounded, there exist $r_+ > 0$ such that
\begin{equation}\label{OuterRadius}
V \subset B(e, r_+);
\end{equation}
we refer to $r_+$ an \emph{outer radius} for $V$. Secondly, we have assumed that $\lambda_0$ is sufficiently large. We may thus fix parameters
\( C_- > 0 \), \( s \in \NN_0  \), and \( z \in \Gamma \) such that for all \( n \in \NN_0 \) we have
\begin{equation}\label{SuffLargeConvenient}
D^n\big( B(z, C_-) \big)  \subseteq V_{\lambda_0}(s+n).
\end{equation}
The set $V(s)$ is again bounded, hence contained in $B(z, R_+)$ for some $R_+ \geq r_+$. We now define
\[
C_+ := R_+ +\tfrac{r_+\lambda_0}{\lambda_0-1}>0.
\]
\end{no}
\begin{proposition}[Support growth]\label{Prop-SupportGroth} 
For all \( n \in \NN_0 \) we have
\[
D^n(z) B\big(e, C_-\lambda_0^n\big) \big) \subseteq V(s+n) \subseteq D^n(z) B\big(e, C_+\lambda_0^n\big).
\]
\end{proposition}
\begin{proof}[Proof] Since $D=D_{\lambda_0}$ is a dilation and $d_G$ is left-invariant we have
\[
D(z) B\big(e, r\lambda_0\big) 
	= B\big(D(z) , r\lambda_0\big)
	= D\big(B(z, r)\big).
\]
Choosing $r := C_-$ then yields the first inclusion in view of \eqref{SuffLargeConvenient}. Towards the second inclusion we prove by induction that
\begin{equation}\label{Vs+nUpper}
	V(s+n) \subseteq B\big(D^n(z),\lambda_0^{n} (R_+ + r_+S_{n-1})\big)\quad \text{for all } n\in\NN_0,
\end{equation}
where $S_{-1} = 0$ and $S_n:=\sum_{k=0}^n \lambda_0^{-k}$ for all $n \in \NN_0$. Indeed, for $n = 0$ we have
\[
V(s+0) = V(s) \subseteq B(z, R_+) = B\big(D^0(z),\lambda_0^{0} (R_+ + r_+S_{-1})\big).
\]
Now if \eqref{Vs+nUpper} holds for some $n \in \NN_0$, then Lemma~\ref{Lem-NaiveGrowth} and \eqref{eq:Inner-outer_approx-balls} yield
\[
V(s+n)_\Gamma \subseteq  B\big(D^n(z),\lambda_0^{n} (R_+ + r_+S_{n-1})\big) \subseteq  B\big(D^n(z),\lambda_0^{n} (R_+ + r_+S_{n-1}) + r_+\big).
\]
Applying the dilation $D$ then yields
\begin{align*}
V(s+n+1)
	= D\big(V(s+n)_\Gamma \big)
	\subseteq &B\left(D^{n+1}(z),\lambda_0^{n+1} \big(R_+ + r_+S_{n-1} + r_+\tfrac{1}{\lambda_0^n}\big)\right)\\
	= &B\big(D^{n+1}(z),\lambda_0^{n+1} (R_+ + r_+S_{n})\big).
\end{align*}
This finishes the induction and establishes \eqref{Vs+nUpper}. Now since $\lambda_0>1$ we have $S_n \nearrow \tfrac{\lambda_0}{\lambda_0-1}$, hence we conclude from \eqref{Vs+nUpper} that
\[
	V(s+n) \subseteq B\big(D^n(z),\lambda_0^{n} (R_+ + r_+S_{n-1})\big)
		\subseteq B\big(D^n(z),\lambda_0^{n} (R_+ + \tfrac{r_+\lambda_0}{\lambda_0-1})\big).
		\qedhere
\]
\end{proof}
\begin{remark}
In the particular case that \(z = e\) (which can always be arranged if $V$ contains an identity neighbourhood), Proposition~\ref{Prop-SupportGroth} implies that the support sets \(V(n)\) eventually contain every bounded subset of $G$. If $z \neq e$, then this cannot be guaranteed, since
the set $V(n)$ may drift further away from $e$. However, they still contain a translate of every bounded subset of $G$, which is sufficient for our purposes. Alternatively, by the following corollary, we can start from a slightly larger finite patch to ensure that the corresponding support sets contain every bounded subset.
\end{remark}
\begin{corollary}
\label{Cor-SupportGrowth}
If $F_0:=z^{-1}V(s) \cap \Gamma$, then for all \( n \in \NN_0 \), 
\[
B\big(e, C_-\lambda_0^n\big) \big) \subseteq V(n,F_0) \subseteq B\big(e, C_+\lambda_0^n\big).
\]
\end{corollary}

\begin{proof} This follows from Proposition~\ref{Prop-SupportGroth} since by Lemma~\ref{Lem-Prop-V(n,M)}(a) the inclusions
\[
D^n(z) B\big(e, C_-\lambda_0^n\big) \big) \subseteq V(s+n) \subseteq D^n(z) B\big(e, C_+\lambda_0^n\big)
\]
are equivalent to the inclusions
\[
 B\big(e, C_-\lambda_0^n\big) \big) \subseteq V\big(n,z^{-1}V(s)\cap\Gamma\big) \subseteq B\big(e, C_+\lambda_0^n\big).\qedhere
\]
\end{proof}

\section{Substitution dynamical systems}
\label{Sec-SubstDynamSyst}

In this section we associate with every substitution datum a subshift; we show that these subshifts are always non-empty and provide conditions which imply that they are minimal, uniquely ergodic and weakly aperiodic. Throughout we fix a dilation datum $\Dd=(G, d, (D_\lambda)_{\lambda >0}, \Gamma, V)$ and a substitution datum $\Ss=(\Aa,\lambda_0,S_0)$ over $\Dd$ and denote by $S: \Aa^\Gamma \to \Aa^\Gamma$ the associated substitution map. We also abbreviate $D := D_{\lambda_0}$.

\subsection{Existence of legal fixpoints}
\label{Sec-ExisLegalFixPoints}
\begin{construction}[Substitution system]\label{Def-SubstitutionSystem} Given two patches $P, Q \in \mathcal{A}_\Gamma^{**}$ we say that \emph{$P$ occurs in $Q$} and write $P \prec Q$ if there exists a $\gamma\in\Gamma$ such that $\gamma P=Q|_{\gamma \mathrm{supp}(P)}$. We recall that every $a \in \Aa$ defines a patch $P_a$ with  $\mathrm{supp}(P_a) = \{e\}$ and $P_a(e) = a$. We then say that a finite patch $P \in \mathcal{A}_\Gamma^*$ is \emph{$S$-legal} if $P \prec S^n(P_a)$ for some $a \in \Aa$. Finally, a configuration $\omega \in \Aa^\Gamma$ is called \emph{$S$-legal} if for every $M \subset \Gamma$ the finite path $\omega|_M$ is $S$-legal. We then define $\Omega(S) \subset \Aa^\Gamma$ by
\[
\Omega(S) := \{\omega \in \mathcal A^\Gamma \mid \omega \text{  is $S$-legal}\}.
\]
\end{construction}
\begin{proposition}\label{Prop-Subshift}
The subset $\Omega(S) \subset \Aa^\Gamma$ is an $S$-invariant subshift.
\end{proposition}
\begin{proof} If a finite patch $P$ satisfies $P \prec S^n(P_a)$ for some $a \in \Aa$, then $S(P) \prec S^{n+1}(P_a)$ and $\gamma.P \prec S^n(P_a)$ for all $\gamma \in \Gamma$. This implies that $\Omega(S)$ is both $S$ and $\Gamma$-invariant.

\item If a sequence $(\omega_n)$ in $\Omega(S)$ converges to some $\omega \in \Aa^\Gamma$, then every patch of $\omega$ appears in $\omega_n$ for all sufficiently large $n$ and is thus $S$-legal. Thus $\Omega(S)\subset \Aa^\Gamma$ is closed, hence compact-metrizable.
\end{proof}
The proof of Proposition \ref{Prop-Subshift} did not use the hypothesis that $\lambda$ be $V$-sufficient. However, this hypothesis is crucially needed in the proof of the following non-triviality result:
\begin{proposition}
\label{Prop-LegalFixpoints} 
The subshift $\Omega(S)$ is non-empty and there exists $k \in \mathbb N$ such that $S^k$ has a fixpoint in $\Omega(S)$. 
\end{proposition}
\begin{proof} We first show that $\Omega(S)$ is non-empty: Let $\omega\in\Aa^\Gamma$ and define $\omega_n:=D^n(z^{-1})S^{n+s}(\omega)$ for $n\in\NN$. By compactness of $\Aa^\Gamma$, there exists a subsequence $\omega_{n_k}$ which converges to some $\rho \in \Aa^\Gamma$, and we claim that $\rho$ is $S$-legal.

Given $M \subset \Gamma$ we consider the corresponding $\rho$-patch $P := \rho|_M$; we have to show that $P$ is $S$-legal. Since $\lambda_0$ is $V$-sufficient there exist $z\in\Gamma$, $s\in\NN_0$ and $C_->0$ such that for all $n \in \NN$ we have
\[
B(e,C_-\lambda_0^n) \,\cap\, \Gamma \subseteq D^n(z^{-1}) V(n+s) \,\cap\, \Gamma = \mathrm{supp}\left( D^n(z^{-1})S^{n+s}(\omega|_{\{e\}}) \right),
\]
where the last equality is a consequence of Proposition~\ref{Prop-SupportFormula}.

Since $\lambda_0>1$, there is a $k_0\in\NN$ such that $M\subseteq B(e,C_-\lambda_0^{n_k})$ for all $k\geq k_0$. Furthermore, the convergence $\omega_{n_k}\to \rho$ in the product topology of $\Aa^\Gamma$ yields that there is a $k_1\geq k_0$ such that for all $k\geq k_1$ we have $\omega_{n_k}|_M=\rho|_M$ and hence
\[
P = \rho|_M
	= \omega_{n_k}|_M
	\prec \omega_{n_k}|_{B(e,C_-\lambda_0^{n_k})}
	\prec D^{n_k}(z^{-1})S^{n_k+s}(\omega_{n_k}|_{\{e\}}).
\]
If we now define $a := \omega_{n_{k_1}}(e) \in \Aa$, then we deduce that $P \prec S^{n_k+s}(P_a)$, hence $P$ is $S$-legal. This proves that $\Omega(S) \neq \emptyset$.

To see that some power of $S$ has a fixpoint we consider the subset $F_0:=z^{-1}V(s)\cap\Gamma\subseteq \Gamma$ from Corollary~\ref{Cor-SupportGrowth}; here $z\in\Gamma$ and $s\in\NN_0$ are parameters witnessing the $V$-sufficiency of $\lambda_0$. We now start from an arbitrary $\omega_0\in\Omega(S)$; since $\Aa$ and $F_0$ are finite, there are $k_0,k\in\NN$ such that $S^{k_0}(\omega_0)|_{F_0} = S^{k_0+k}(\omega_0)|_{F_0}$. Define $\omega_1:=S^{k_0}(\omega_0)$ and $\omega_{n+1}:=S^{k}(\omega_n)$ for all $n \geq 2$. We will show that $(\omega_n)_{n\in\NN}$ is convergent to some $\omega\in\Omega(S)$ and that $S^k(\omega)=\omega$.

By construction, we have $\omega_1|_{F_0}=\omega_n|_{F_0}$ for all $n\in\NN$. For all $m, n \in \NN$ we have $S^{mk}(\omega_n) = \omega_{n+m}$, and hence Proposition~\ref{Prop-SupportFormula} yields 
\begin{equation}\label{Fixpoint1}
\omega_{m+1}|_{V(mk,F_0)} 
	= S^{mk}(\omega_1)|_{V(mk,F_0)} 
	= S^{mk}(\omega_1|_{F_0}) 
	= S^{mk}(\omega_n|_{F_0}) 
	= \omega_{m+n}|_{V(mk,F_0)}.
\end{equation}
Now let $M\subseteq \Gamma$ be finite. By Corollary~\ref{Cor-SupportGrowth} there exists an $m_0\in\NN$ such that $M\subseteq V(mk,F_0)$ for $m\geq m_0$, hence \eqref{Fixpoint1} yields 
$$
\omega_{m_0+1}|_{M} = \omega_{m_0+n}|_{M}
	\qquad \text{for all $n\in\NN$}.
$$
By definition of the product topology this implies that $(\omega_n)$ converges to some $\omega \in \Omega(S)$, and by Proposition \ref{Prop-Subshift} we have $\omega \in \Omega(S)$, and by Proposition~\ref{Prop-SubstitutionProperties}(a) we have
\[
S^k(\omega) = \lim_{n\to\infty} S^k(\omega_n) = \lim_{n\to\infty} \omega_{n+1} =  \omega.\qedhere
\]
\end{proof}
%%%%%%%%%%%%%%%%%%%%%%%%%%%%%%%%%%%%%%%%%%%%%%%%%%%%%%%%%%%%%%%%%%%%%%
\subsection{From primitivity to linear repetitivity, minimality and unique ergodicity}
\label{Sec-LR+UniqErg} 
We have seen that there is a huge freedom in choosing a substitution datum over a given dilation datum, hence it is natural to ask whether one can choose substitution data with additional properties, which in turn ensure additional properties of the associated subshift.
The following notion is classical in the abelian setting:
\begin{definition} 
\label{Def-Primitive} 
A substitution datum $\Ss=(\mathcal{A}, \lambda_0, S_0)$ is called \emph{primitive} with \emph{exponent} $L \in \mathbb N$ if for all $a,b \in \mathcal A$ we have $P_a \prec S^L(P_b)$.
\end{definition}
\begin{example} The substitution given for the Heisenberg group in Example~\ref{Ex-HeisenbergSubstitution} is primitive with exponent $L=1$.
\end{example}
\begin{no} To spell out some of the main consequences of primitivity we recall that a subshift is called \emph{minimal} if every $\Gamma$-orbit is dense and \emph{uniquely ergodic} if it admits a unique $\Gamma$-invariant probability measure. Moreover, an element $\omega\in\Aa^\Gamma$ is called {\em repetitive} if, for each finite subpatch $P$ of $\omega$, there exists a radius $r>0$ such that $P\prec \omega|_{B(x,r)}$ for all $x\in \Gamma$. It is called {\em linearly repetitive} if there exists a constant $C>0$ such that for all $r\geq 1$ and each $x,y\in \Gamma$, we have $\omega|_{B(x,r)} \prec \omega|_{B(y,Cr)}$. By analogy with the abelian case we are going to establish:
\end{no}
\begin{theorem}[Consequences of primitivity]
\label{Thm-LinearRep} If $\Ss$ is primitive then the following hold:
\begin{enumerate}[(a)]
\item $\Omega(S)$ is minimal and every $\omega \in \Omega(S)$ is linearly repetitive.
\item If moreover $\Dd$ has exact polynomial growth, then $\Omega(S)$ is uniquely ergodic.
\end{enumerate}
\end{theorem}
\begin{no} The key step in the proof of Theorem \ref{Thm-LinearRep} is to find a single element $\omega \in \Omega(S)$ which is linearly repetitive; then (a) follows from general arguments and (b) is a consequence of the main result of the companion paper \cite{BHP25-LR}. The way in which primitivity is used is as follows: If $P \in \Aa_\Gamma^*$ is $S$-legal, then there exist $n_P \in \NN$ and $a_0 \in \Aa$ such that $P \prec S^{n_P}(P_{a_0})$. If now $\Ss$ is primitive with exponent $L$, then we deduce that
\begin{equation}\label{PatchOccurence} P \prec S^{n_P}(P_{a_0}) 
	\prec S^{n_P}\big(S^{n-n_P}(P_a)\big) 
	= S^{n}(P_a) \quad \text{for all }n \geq N_P := n_{p} + L.
	\end{equation}
\end{no}
\begin{lemma}\label{Lem-FixpointLR}
Let $k \in \NN$ and assume that $\omega \in \Omega(S)$ is an $S^k$-fixpoint. If $\Ss$ is primitive, then $\omega$ is linearly repetitive.
\end{lemma}
\begin{proof} 
\underline{Step 1:} We first show that $\omega$ is repetitive. For this let $P$ be a finite subpatch of $\omega$. By \eqref{PatchOccurence}, there exists an $m\in\NN$ such that $P\prec S^{mk}(P_a)$ for all $a\in\Aa$. Since $D^{mk}(\Gamma)\subseteq G$ is relatively dense there is an $r>0$ such that for each $x\in \Gamma$ there is an $\eta_x\in\Gamma$ such that $D^{mk}(\eta_x) V(mk) \subseteq B(x,r)$. By Lemma~\ref{Lem-Prop-V(n,M)} we have $D^{mk}(\eta_x) V(mk) = V(mk,\{\eta_x\})$. Proposition~\ref{Prop-SupportFormula} implies for all $x \in \Gamma$,
$$
P \prec S^{mk}(P_{\omega(\eta_x)}) 
	= S^{mk}(\omega)|_{V(mk,\{\eta_x\})}
	= \omega|_{D^{mk}(\eta_x) V(mk)}
	\prec \omega|_{B(x,r)}.
$$
Since $r$ is independent of $x$, this shows that $\omega$ is repetitive.

\underline{Step 2:} We now upgrade this statement to linear repetitivity. We start by fixing some constants depending on our substitution datum $\Ss=(\Aa,\lambda_0,S_0)$ and the underlying dilation datum $\Dd=(G, d_G, (D_\lambda)_{\lambda > 0}, \Gamma, V)$. As before we use the notation $D=D_{\lambda_0}$.

Let $r_+>0$ be chosen such that $V \subseteq B(e,r_+)$.
Since $\lambda_0>1$ is sufficiently large relative to $V$, there is a $C_->0$, an $s\in\NN_0$ and a $z\in\Gamma$ such that $D^n\big(B(z,C_-)\big)\subseteq V(s+n)$ for all $n\in\NN_0$.
Let $R_+>0$ be such that $V(s)\subseteq B(z,R_+)$ and define $C_+:=R_+ +\tfrac{r_+\lambda_0}{\lambda_0-1}>0$ as in Proposition~\ref{Prop-SupportGroth}.
Define $F_0:=z^{-1}V(s)\cap\Gamma$ like in Corollary~\ref{Cor-SupportGrowth}.

Since $\lambda_0>1$, there exists an $i\in\mathbb{N}$ such that
$$
\lambda_0^i  \geq 2 \frac{r_+}{C_-}.
$$
Since $F_0$ is finite, there is a radius $r(i)>0$ such that $V(i,F_0)\subseteq B(e,r(i))$. Define $r_0:=\max\{r_+, r(i)\}$. With these choices, we observe that the set of patches
\[
\Pp_0:= \big\{ (\gamma\omega)|_{B(e,r)} \,:\, \gamma\in\Gamma, 0<r\leq r_0 \big\}
\]
is finite since \(\Aa\) and the set \(\{ B(e,r)\cap\Gamma \mid 0<r\leq r_0 \}\) are finite by uniform discreteness of \(\Gamma\).

Then the repetitivity of $\omega$ (Step 1) yields the existence of a constant $C_1\geq 1$ such that
\begin{equation}
\label{eq:LR_Start}
\omega|_{B(x, r)} \prec \omega|_{B(y,C_1)},
	\qquad x,y\in \Gamma,\, 0<r\leq r_0.
\end{equation}
Due to Corollary~\ref{Cor-SupportGrowth}, there is a $j\in\NN$ such that $B(e,C_1)\subset V(j,F_0)$. Define
$$
C:= \max\left\{ C_1, \lambda_0^{j+k}\frac{C_+}{r_+} + \lambda_0^k \right\} \geq 1.
$$
If $1\leq r\leq r_+$, we conclude 
$$
\omega_{B(x,r)} \prec \omega|_{B(y,C_1)} \prec \omega|_{B(y,Cr)},
	\qquad x,y\in \Gamma.
$$
Since $\lambda_0>1$, we have
$$
\bigsqcup_{n\in\NN} \big[\lambda_0^{k(n-1)} r_+, \lambda_0^{kn} r_+\big) = [r_+,\infty).
$$
Let $r>r_+$. By the previous considerations, there is a unique $n_0\in\NN$ such that 
\begin{equation}
\label{eq:LR_SymbSubst}
\lambda_0^{k(n_0-1)} r_+ \leq r < \lambda_0^{kn_0} r_+.
\end{equation}
Let $x,y\in \Gamma$. Since $G=D^{kn_0}(\Gamma)D^{kn_0}(V)$, there are $\eta_x,\eta_y\in\Gamma$ and $v_x,v_y\in V$ such that
$$
x= D^{kn_0}(\eta_x)D^{kn_0}(v_x)
	\quad\text{ and }\quad
y= D^{kn_0}(\eta_y)D^{kn_0}(v_y).
$$
Since $V\subseteq B(e,r_+)$ and $D^{kn_0}(B(e,s))=B\big(e,\lambda_0^{kn_0}s\big)$, we conclude
$$
x\in B\big(D^{kn_0}(\eta_x),\lambda_0^{kn_0}r_+\big)
	\quad\text{ and }\quad
D^{kn_0}(\eta_y) \in  B\big(y,\lambda_0^{kn_0}r_+\big).
$$
Due to \eqref{eq:LR_SymbSubst} and $B(z,r)B(e,s) \subset B(z,r+s)$, we first observe
$$
B(x,r) 
	\subset B\big(D^{kn_0}(\eta_x),\lambda_0^{kn_0}r_+\big) B(e,r) 
	\subset B\big(D^{kn_0}(\eta_x),r + \lambda_0^{kn_0}r_+\big)
	\subset
	B\big(D^{kn_0}(\eta_x),\lambda_0^{kn_0} 2r_+\big)
	.
$$
By the choice of $i\in\mathbb{N}$, we derive
$$
\lambda_0^{kn_0} 2r_+ 
	= \lambda_0^{kn_0} 2 \frac{r_+}{C_-} C_-
	\leq  \lambda_0^{kn_0+i} C_-.
$$
In addition, the inclusion $B\big(e,\lambda_0^{kn_0+i}C_-\big) \subseteq V(kn_0+i,F_0)$ holds by Corollary~\ref{Cor-SupportGrowth}. 
Hence, the previous considerations and Lemma~\ref{Lem-Prop-V(n,M)}~(a) and (c) lead to
$$
B(x,r) 
	\subset D^{kn_0}(\eta_x) B(e,\lambda_0^{kn_0+i}C_-)
	\subset D^{kn_0}(\eta_x) V(kn_0+i,F_0)
	= V\big(kn_0, \eta_x(V(i,F_0)\cap\Gamma)\big).
$$
Thus, $S^{kn_0}(\omega)=\omega$ and Proposition~\ref{Prop-SupportFormula} yield
$$
\omega|_{B(x,r)}
	\prec S^{kn_0}(\omega)|_{V\big(kn_0, \eta_x(V(i,F_0)\cap\Gamma)\big)}
	= S^{kn_0}\left( \omega|_{\eta_x(V(i,F_0)\cap\Gamma)}\right).
$$
Hence, we conclude from $V(i,F_0)\subseteq B(e,r(i))$, \eqref{eq:LR_Start} and $B(e,C_1)\subset V(j,F_0)$ that
$$
\omega|_{\eta_x(V(i,F_0)\cap\Gamma)} 
	\prec \omega|_{B(\eta_x,r(i))\cap\Gamma} 
	\prec \omega|_{B(\eta_y,C_1)\cap\Gamma}
	\prec \omega|_{\eta_y (V(j,F_0)\cap\Gamma)}.
$$
If we combine this with the previous consideration, Proposition~\ref{Prop-SupportFormula} implies
$$
\omega|_{B(x,r)}
	\prec S^{kn_0}\left( \omega|_{\eta_x(V(i,F_0)\cap\Gamma)}\right)
	\prec S^{kn_0}\left( \omega|_{\eta_y V(j,F_0) \cap \Gamma}\right)
	= S^{kn_0}(\omega)|_{V(kn_0, \eta_y (V(j,F_0)\cap\Gamma))}.
$$
Focusing again on the support, Lemma~\ref{Lem-Prop-V(n,M)}~(a) and (c) lead to
$$
V\big(kn_0, \eta_y (V(j,F_0)\cap\Gamma)\big)
	= D^{kn_0}(\eta_y) V\big(kn_0,V(j,F_0)\cap\Gamma\big)
	= D^{kn_0}(\eta_y) V(kn_0+j,F_0).\\
$$
Then Corollary~\ref{Cor-SupportGrowth} and $D^{kn_0}(\eta_y) \in  B(y,\lambda_0^{kn_0}r_+)$ yield
\begin{align*}
V\big(kn_0, \eta_y (V(j,F_0)\cap\Gamma)\big)
	= D^{kn_0}(\eta_y) V\big(kn_0+j, F_0 \big)
	\subseteq &B(y,\lambda_0^{kn_0}r_+) B\big(e,\lambda_0^{kn_0+j} C_+\big)\\
	\subseteq &B\big(y,\lambda_0^{kn_0}(\lambda_0^jC_+ + r_+)\big),
\end{align*}
using again $B(z,r)B(e,s) \subset B(z,r+s)$. Then \eqref{eq:LR_SymbSubst} and the choice of $C$ give us the estimate
$$
\lambda_0^{kn_0}(\lambda_0^jC_+ + r_+)
	= \left( \lambda_0^k\frac{\lambda_0^jC_+ + r_+}{r_+} \right) \lambda_0^{k(n_0-1)} r_+ 
	\leq Cr,
$$
implying $V\big(kn_0, \eta_y (V(j,F_0)\cap\Gamma)\big) \subset B(y,Cr)$. Hence, we finally conclude
$$
\omega|_{B(x,r)}
	\prec S^{kn_0}(\omega)|_{V(kn_0, \eta_y (V(j,F_0)\cap\Gamma))}
	\prec \omega|_{B(y,Cr)},
$$
using $S^{kn_0}(\omega)=\omega$. This proves that $\omega$ is linearly repetitive.
\end{proof}

\begin{no}\label{DynamicalRep} It follows from Proposition \ref{Prop-LegalFixpoints} and Lemma \ref{Lem-FixpointLR} that if $\Ss$ is primitive, then $\Omega(S)$ contains a linearly repetitive configuration. To deduce Theorem \ref{Thm-LinearRep}(a) we need to discuss dynamical characterizations of repetitivity. Given an element $\omega \in \Aa^\Gamma$ we denote by 
\[
\Omega_\omega :=  \overline{\{\gamma \omega \mid \gamma \in \Gamma\}} \subset \Aa^\Gamma
\] 
the orbit closure of $\omega$. By definition of the product topology we have $\omega' \in \Omega_\omega$ if and only if $\omega'|_F \prec \omega$ for every finite subset $F \subset \Gamma$, i.e.\ all finite patches of $\omega'$ appear in $\omega$.
In particular, a subshift $\Omega$ is minimal if and only if all of its elements have exactly the same finite patches up to translation. Since the definition of (linear) repetitivity involves only finite patches, this implies that
if one element of a minimal subshift is (linearly) repetitive, then all of its elements are (linearly) repetitive. 
\end{no}
\begin{lemma}\label{Lemma_Minimal} A configuration $\omega \in \Aa^\Gamma$ is repetitive if and only if the orbit closure $\Omega_\omega$ is minimal.
\end{lemma}
\begin{proof} For $\Gamma$ abelian this is proved in \cite[Prop.\ 4.3]{BaakeGrimm13}. However, abelianness of $\Gamma$ is never used in the proof, see also \cite[Prop.~2.4]{BHP25-LR}. 
\end{proof}
\begin{proof}[Proof of Theorem \ref{Thm-LinearRep}](a) Using Proposition \ref{Prop-LegalFixpoints} we choose $k \in \NN$ and $\omega\in \Omega(S)$ such that $S^k(\omega) = \omega$. We claim that $\Omega(S) = \Omega_\omega$. Indeed, the inclusion $\supseteq$ follows from Proposition \ref{Prop-Subshift}. For the converse we observe that if $P$ is any finite $S$-legal patch, then \eqref{PatchOccurence}
and Proposition~\ref{Prop-SupportFormula} imply that there is an $m\in\NN$ such that 
$$
P\prec S^{mk}\big(P_{\omega(e)}\big)
	\prec S^{mk}(\omega)|_{V(mk)}
	= \omega|_{V(mk)}.
$$
Thus, every $S$-legal patch appears in $\omega$, and hence $\Omega(S) \subseteq \Omega_\omega$, which proves the claim. We then deduce with Lemma \ref{Lemma_Minimal} that $\Omega(S)$ is minimal and with \S \ref{DynamicalRep} that every element of $\Omega(S)$ is linearly repetitive.

\item (b) In the terminology of \cite{BHP25-LR}, every element in $\Omega(S)$ is symbolically linearly repetitive with respect to $d|_{\Gamma \times \Gamma}$. 
By assumption, $\Gamma$ has exact polynomial growth with respect to $d|_{\Gamma \times \Gamma}$ and hence unique ergodicity of $\Gamma \curvearrowright \Omega(S)$ follows from
\cite[Theorem~5.7~(b)]{BHP25-LR}.
\end{proof}

\subsection{A criterion for weak aperiodicity}
\label{Sec-WeakAper}
In this subsection we provide a sufficient condition on the subshift $\Omega(S)$ to be weakly aperiodic. As we will see later, this condition can be arranged for every given substitution group. 
\begin{definition}
\label{Def-S_NonPeriodic}
The substitution datum $\Ss$ is called \emph{non-periodic} if $S_0$ is injective and
\[
\big(\gamma^{-1} S(P_a)\big)|_{\gamma^{-1}D(V)\cap D(V)}\neq S(P_b)|_{\gamma^{-1}D(V)\cap D(V)}
\]
for all $\gamma\in \big( D(V)\cap\Gamma \big)\setminus\{e\}$ and $a,b\in\mathcal{A}$.
\end{definition}
The term ``non-periodic'' is justified by the following main result of this subsection:
\begin{theorem} 
\label{Thm-weakAP-ex} 
If $\Ss$ is non-periodic and $\omega \in \Omega(S)$ is an $S^k$-fixpoint for some $k\in\mathbb{N}$, then $\mathrm{Stab}_{\Gamma}(\omega) = \{e\}$. In particular, 
$\Omega(S)$ is weakly aperiodic.
\end{theorem} 
\begin{remark} 
Definition \ref{Def-S_NonPeriodic} and Theorem \ref{Thm-weakAP-ex} are inspired by results about geometric (i.e.\ non-symbolic) substitutions in the Euclidean setting. Indeed, it is well-known that injectivity of the substitution map leads to weakly aperiodic tilings in this setting \cite{Sol97,Sol98,AnPu98}; however, to carry this argument through in the symbolic setting, injectivity of $S_0$ is not enough. While there are other methods to establish weak aperiodicity of substitution systems in the abelian setting (notably via the existence of so-called proximal pairs, see \cite{BarOl14,BaakeGrimm13}), the more geometric criterion from Theorem \ref{Thm-weakAP-ex} seems to be new even in the abelian case.

As mentioned in the introduction, in the abelian case weak aperiodicity of $\Omega(S)$ can be strengthened into strong aperiodicity if $\Omega(S)$ is minimal. While this arguments breaks down in the non-abelian case, it is still possible to obtain strong aperiodicity of $\Omega(S)$ in some cases of interest; we will return to this problem in Section \ref{Sec-StrongAper}.
\end{remark}
\begin{no}
Our proof of Theorem~\ref{Thm-weakAP-ex} builds on two key observations. We first show in Lemma~\ref{lemma-injectivity} below that injectivity of $S_0$ implies that
$S^n:\mathcal{A}^{\Gamma} \to \mathcal{A}^{\Gamma} $ is injective for all $n \in \NN$, and that this injectivity is witnessed by very specific elemens of $\Omega(S)$. 
We then use this result to show that the stabilizers \(\mathrm{Stab}_\Gamma(S^n(\omega))\) shrink as \(n\) increases (see Proposition~\ref{prop-few fixedpoints}).
We recall from \S \ref{VDef} that for finitee $M \subseteq \Gamma$ and $n \in \NN_0$ the sets $V(n,m)$ are defined as
\[
V(0,M) := MV \quad \mbox{ and } \quad V(n,M) := D\big( V(n-1, M)_{\Gamma} \big).
\]
As before we abbreviate $V(n) := V(n,\{e\})$ for all $n \in \NN_{0}$; then our first key result reads:
\end{no}
\begin{lemma} \label{lemma-injectivity}

If the substitution $S_0:\mathcal{A}\to\mathcal{A}^{D(V)\cap\Gamma}$ is injective, then $S^n:\mathcal{A}^\Gamma\to\mathcal{A}^\Gamma$ is injective for all $n\in\NN$. More precisely, if $\omega_1(\eta)\neq \omega_2(\eta)$ for $\eta\in\Gamma$, then 
\[
S^n(\omega_1)|_{V(n,\{\eta\})}
	\neq S^n(\omega_2)|_{V(n,\{\eta\})} 
	\quad\quad \mbox{ for all} \quad n \in \NN.
\]
\end{lemma} 

\begin{proof}
We first show that $S$ is injective, thus settling the case $n=1$. So fix some distinct $\omega_1, \omega_2 \in \mathcal{A}$ along with $\eta \in \Gamma$ such that $\omega_1(\eta) \neq  \omega_2(\eta)$. By injectivity of $S_0$ there must be some $\tilde{\gamma} \in D(V) \cap \Gamma$ such that $S_0(\omega_1(\eta))(\tilde{\gamma}) \neq S_0(\omega_2(\eta))(\tilde{\gamma})$. Set $\gamma := D(\eta)\tilde{\gamma}$. 
The construction of the substitution map $S$ (see Construction~\ref{Const-ExtensionRule}) leads to 
\[
S(\omega_1)(\gamma) = S_0(\omega_1(\eta))(\tilde{\gamma}) \neq  S_0(\omega_2(\eta))(\tilde{\gamma}) = S(\omega_2)(\gamma).
\]
Hence, $S$ is injective. Since $\gamma \in D(\eta)D(V) = V(1,\{\eta\})$, we find $S(\omega_1)|_{V(1,\{\eta\})} \neq S(\omega_2)|_{V(1,\{\eta\})}$. This settles the case $n=1$. 
	
Proceeding by induction, suppose that the assertion of the lemma is true for some $n\geq 1$. Then $S^{n+1} = S \circ S^{n}$ is clearly injective as a composition of two injective maps. Further fix some distinct $\omega_1, \omega_2 \in \mathcal{A}$ along with $\eta \in \Gamma$ such that $\omega_1(\eta) \neq  \omega_2(\eta)$. By the induction hypothesis we find some $\tilde{\eta} \in V(n,\{\eta\}) \cap \Gamma$ such that $S^n(\omega_1)(\tilde{\eta}) \neq S^n(\omega_2)(\tilde{\eta})$. By injectivity of $S_0$ we find some $\tilde{\gamma} \in D(V) \cap \Gamma$ such that $S_0(S^n(\omega_1)(\tilde{\eta}))(\tilde{\gamma}) \neq S_0(S^n(\omega_2)(\tilde{\eta}))(\tilde{\gamma})$. Setting $\gamma:= D(\tilde{\eta}) \tilde{\gamma}$, we obtain in the same manner as in the proof for $n=1$ that $S^{n+1}(\omega_1)(\gamma) \neq S^{n+1}(\omega_2)(\gamma) $. The observation that $\gamma \in D\big( (V(n,\{\eta\}) \cap \Gamma) V \big) = V(n+1, \{\eta\})$ now finishes the proof. 
\end{proof}

We now turn to the second key result:
\begin{proposition} 
\label{prop-few fixedpoints}
Let $\Dd$ be a dilation datum and $\Ss$ a substitution datum with substitution map $S$.
If $\Ss=(\mathcal{A}, \lambda_0, S_0)$ is non-periodic, then for each $n\in\NN$, every $\gamma\in\Gamma\setminus D^n(\Gamma)$ and all $\omega_1, \omega_2 \in\mathcal{A}^\Gamma$, we have 
	$$
	\gamma^{-1}S^n(\omega_1)|_{V(n)} \neq S^n(\omega_2)|_{V(n)}.
	$$
\end{proposition} 

\begin{proof}
The statement is proven via induction over \( n \in \NN \). We begin with the induction base \( n = 1 \).
Fix \( \gamma \in \Gamma \setminus D(\Gamma) \) and \( \omega_1, \omega_2 \in \mathcal{A}^\Gamma \). By equality~\eqref{FundDomainIterate}, the union 
\(
G = \bigsqcup_{\eta \in \Gamma} D(\eta V)
\)
consists of pairwise disjoint sets. Hence, there exists a unique \( \eta \in \Gamma \) such that
\[
\gamma \in D(\eta V) \cap \Gamma = D(\eta) D(V) \cap \Gamma.
\]
Moreover, \( \gamma \neq D(\eta) \) by assumption. Since \( D(\eta) \in \Gamma \), it follows that \(D(\eta)^{-1} \gamma \in \big( D(V) \cap \Gamma \big) \setminus \{e\}\).
By non-periodicity of the substitution datum, we obtain
\[
\big( D(\eta)^{-1} \gamma \big)^{-1} S\big( P_{\omega_1(\eta)} \big)(g) 
\neq S\big( P_{\omega_2(e)} \big)(g)
\]
for some \( g \in \big( D(\eta)^{-1} \gamma \big)^{-1} D(V) \cap D(V) \cap \Gamma \). By the definition of \( S: \mathcal{A}^\Gamma \to \mathcal{A}^\Gamma \), this yields
\[
\big( D(\eta)^{-1} \gamma \big)^{-1} S\big( P_{\omega_1(\eta)} \big)(g)
= S_0\big( \omega_1(\eta) \big)\big( D(\eta)^{-1}(\gamma g) \big)
= S(\omega_1)(\gamma g)
= \gamma^{-1} S(\omega_1)(g),
\]
and since \( g \in D(V) \cap \Gamma \), we also have
\[
S\big( P_{\omega_2(e)} \big)(g) = S_0\big( \omega_2(e) \big)(g) = S(\omega_2)(g).
\]
Thus,
\[
\gamma^{-1} S(\omega_1)(g) \neq S(\omega_2)(g)
\]
and since \( g \in D(V) = V(1) \), this proves the base case.

\medskip

For the induction step, assume the proposition holds for all \( k \leq n \) and fix \( \gamma \in \Gamma \setminus D^{n+1}(\Gamma) \). We distinguish two cases:

\begin{itemize}
\item[(a)] Suppose \(\gamma \in \Gamma \setminus D^n(\Gamma)\): Then, using the induction hypothesis for \( n \), we obtain
\[
\gamma^{-1} S^{n+1}(\omega_1)|_{V(n)}
	= \gamma^{-1} S^n\big( S(\omega_1) \big)|_{V(n)}
	\neq S^n\big( S(\omega_2) \big)|_{V(n)}
	= S^{n+1}(\omega_2)|_{V(n)}.
\]
Since \( V(n) \cap \Gamma \subseteq V(n+1) \cap \Gamma \) by Lemma~\ref{Lem-Prop-V(n,M)}~(d), the desired result follows.

\item[(b)] Suppose \(\gamma \in D^n(\Gamma) \setminus D^{n+1}(\Gamma)\): Then there exists \( \tilde{\gamma} \in \Gamma \) with \( \gamma = D^n(\tilde{\gamma}) \), but \( \tilde{\gamma} \notin D(\Gamma) \). Applying the base case to \( \tilde{\gamma} \), we find
\[
\tilde{\gamma}^{-1} S(\omega_1)(\eta) \neq S(\omega_2)(\eta)
\]
for some \( \eta \in D(V) \cap \Gamma \). Since \( S_0 \) is injective, Lemma~\ref{lemma-injectivity} asserts
\[
S^n\big( \tilde{\gamma}^{-1} S(\omega_1) \big)|_{V(n, \{\eta\})}
\neq S^n\big( S(\omega_2) \big)|_{V(n, \{\eta\})}.
\]
From \( D(V) = V(1) \) and Lemma~\ref{Lem-Prop-V(n,M)}~(b) and (c), it follows that
\[
V(n, \{\eta\}) \subseteq V(n, D(V) \cap \Gamma) = V(n+1).
\]
Thus,
\[
S^n\big( \tilde{\gamma}^{-1} S(\omega_1) \big)|_{V(n+1)}
\neq S^n\big( S(\omega_2) \big)|_{V(n+1)}.
\]
Now, by Proposition~\ref{Prop-SubstitutionProperties}~(b),
\[
\gamma^{-1} S^{n+1}(\omega_1) 
= D^n( \tilde{\gamma}^{-1} ) S^n\big( S(\omega_1) \big)
= S^n\big( \tilde{\gamma}^{-1} S(\omega_1) \big),
\]
and since \( S^{n+1}(\omega_2) = S^n\big( S(\omega_2) \big) \), the desired inequality follows.\qedhere
\end{itemize}
\end{proof}

We are now ready to prove Theorem~\ref{Thm-weakAP-ex}.

\begin{proof}[Proof of Theorem~\ref{Thm-weakAP-ex}]
By Proposition~\ref{Prop-LegalFixpoints}, there exist an $\omega\in\Omega(S)$ and a $k\in\NN$ such that $S^k(\omega)=\omega$. Let $\omega\in\Omega(S)$ be such an $S^k$-fixpoint and let $h \in \Gamma \setminus \{e\}$. By uniform discreteness of $\Gamma$, we observe that there is some $n \in \NN$ such that $h \in \Gamma \setminus D^{nk}(\Gamma)$ since otherwise one would find a sequence $(\gamma_n)$ of non-trivial elements in $\Gamma$ such that $d(\gamma_n, e) = \lambda_0^{-nk} d(h,e) \to 0$ as $n \to \infty$. 
Since \( S^{k}(\omega) = \omega \), it follows from Proposition~\ref{prop-few fixedpoints} applied to \( \omega_1 = \omega_2 = \omega \) that
\[
	h^{-1} \omega|_{V(nk)} =  h^{-1} S^{nk}(\omega)|_{V(nk)} \neq S^{nk}(\omega)|_{V(nk)} = \omega|_{V(nk)}.
\]

This yields $h^{-1} \omega \neq \omega$ and since $h \in \Gamma \setminus \{e\}$ was chosen arbitrarily, we have $\mathrm{Stab}_\Gamma(\omega)=\{e\}$. 
\end{proof}

\section{Existence of good substitution data}
\label{Sec-Ex_GoodSubstData}
The purpose of this section is twofold. Firstly, we establish the existence of primitive substitution data for arbitrary dilation data. Secondly, we show that if a dilation datum belongs to a special class -- referred to as split dilation data -- then it even admits non-periodic primitive substitution data. In fact, given a split dilation datum, we provide a method to construct many explicit examples of such substitutions, which we call good substitution data. In Section~\ref{Sec-Lie}, we will prove that every group admitting a dilation datum also admits a split dilation datum. Taken together, these results thus show that every substitution group admits a non-periodic primitive substitution.

\subsection{Primitive substitution data}
Let $\mathcal{D}:=(G, d, (D_\lambda)_{\lambda >0}, \Gamma, V)$ be an arbitrary dilation datum. The following construction provides a primitive substitution datum over $\Dd$:

\begin{construction}
Since $\Dd$ is a dilation datum, there exists $\lambda > 1$ such that $D_\lambda(\Gamma) \subseteq \Gamma$. As before, we assume that $V$ is bounded and contains an open ball $B(x, r_-) \subseteq V$ for some 
$x \in V$ and $r_- > 0$. If we fix such a ball, then we obtain
\[
B(D_\lambda^N(x), \lambda^N r_-) 
	= D_\lambda^N\big(B(x, r_-)\big) 
	\subseteq D_\lambda^N(V) \quad \text{for all } N \in \NN.
\]
In particular, $D_\lambda^N(V)$ contains arbitrarily large balls as $N \to \infty$ as $\lambda > 1$. Since $\Gamma$ is relatively dense there thus exists $N \in \NN$ such that $|D_\lambda^N(V) \cap \Gamma| \geq 2$. Set $\lambda_0 := \lambda^N$, define $D := D_{\lambda_0}$, and let $\mathcal{A} := \{a, b\}$ be a two-letter alphabet. As $|D(V) \cap \Gamma| \geq 2$, we can choose two disjoint non-empty subsets $\Xi_a, \Xi_b \subseteq \Gamma$ such that $D(V) \cap \Gamma = \Xi_a \sqcup \Xi_b$. 

Define a substitution rule $S_0 \colon \mathcal{A} \to \mathcal{A}^{D(V) \cap \Gamma}$ by setting $S_0(c)(\Xi_a) = \{a\}$ and $S_0(c)(\Xi_b) = \{b\}$ for all $c \in \mathcal{A}$. Then $(\mathcal{A}, \lambda_0, S_0)$ defines a primitive substitution with $L = 1$. A similar construction works for any finite alphabet.
\end{construction}

\subsection{Non-periodic primitive substitution data}
\label{Sec-Exist_nonperiodicSubst}
As shown in the previous subsection, any given dilation datum admits a primitive substitution datum. While it remains unclear whether every dilation datum also admits a non-periodic primitive substitution datum, we will at least establish the following slightly weaker result:

\begin{theorem}\label{Thm-ExNonPer} Let $\mathcal A$ be a finite alphabet with $|\mathcal A| \geq 2$ and let $G$ be a substitution group of dimension $\geq 2$. Then there exist a dilation datum $\mathcal D = (G, d, (D_\lambda)_{\lambda >0}, \Gamma, V)$ of exact polynomial growth over $G$ and a non-periodic primitive substitution datum $\mathcal S = (\mathcal A, \lambda_0, S_0)$ over $\mathcal A$ for $\mathcal D$.
\end{theorem}

The remainder of this subsection is devoted to the proof of Theorem \ref{Thm-ExNonPer}. Our proof is constructive,  i.e.\ we are going to describe below a family of explicit non-periodic primitive substitution data for all substitution groups.

The class of substitution groups will be studied in detail in Section \ref{Sec-Lie} using tools from Lie theory. However, to understand the construction underlying Theorem \ref{Thm-ExNonPer} we only need a few basic structural features of substitution groups; the reader who is willing to take these structural features on faith should be able to follow the construction without knowing any Lie theory.
\begin{construction}[Central extensions via cocycles]
\label{Const-CentralExtension}
 If $H$ is a lcsc group and $m \in \NN$, then a continuous map $\beta: H \times H \to \RR^m$ is called a \emph{normalized cocycle}, provided 
\begin{equation}
\label{Eq-GroupCocycle}
\beta(e,e) = 0  \qand 
\beta(g_2,g_3) + \beta(g_1,g_2g_3) = \beta(g_1g_2,g_3) + \beta(g_1,g_2)
\end{equation}
for all $g_1,g_2,g_3\in H$. In this case we can define a continuous group multiplication on  $H \times \RR^m$ by
\[
(g,x)(h,y) := (gh, x+y +\beta (g,h)) \qand (g,x)^{-1} := (g^{-1},-x-\beta(g,g^{-1})), \quad g,h \in H, x,y \in \RR^m.
\]
Indeed, the group axioms are immediate from \eqref{Eq-GroupCocycle} and the following lemma.
\end{construction}

\begin{lemma}
\label{Lem-NiceCocycle}
If $\beta: H \times H \to \RR^m$ is a normalized cocycle, then 
\[
\beta(g,g^{-1}) = \beta(g^{-1},g) \quad \mbox{and} \quad \beta(g, e) = \beta(e,g) = 0 \quad \mbox{for all }
	g\in H.
\]
\end{lemma}

\begin{proof}
We insert the choices $g_1 := g_3 := g^{-1}$ and $g_2 := g$, respectively $g_1 := g$, $g_2 := g_3 := e$ into the cocycle identity~\eqref{Eq-GroupCocycle}.
\end{proof}
In the sequel, the lcsc group from Construction \ref{Const-CentralExtension} will be denoted  $H \times_\beta \RR^m$; it is called 
the \emph{central extension} of $H$ defined by the cocycle $\beta$. If $A \subset H$ and $B \subset \RR^m$ are subsets then we write $A \times_\beta B$ for $A \times B$ when considered as a subset of $H \times_\beta \RR^m$.

\begin{definition}
\label{Def-SplitDilationDatum} 
Let $H$ be a non-compact group, $m \in \NN$, $\beta: H \times H \to \RR^m$ be a normalized cocycle and $G :=  H \times_\beta \RR^m$. We say that a dilation datum $\mathcal D = (G, d, (D_\lambda)_{\lambda >0}, \Gamma, V)$ over $G$  is \emph{split} if the following hold:
\begin{enumerate}[({A}1)]
\item There exists a dilation family $(D_\lambda')_{\lambda > 0}$ on $H$ and a positive integer $\rho$ such that
\[
D_{\lambda}(h,v) = (D_{\lambda}'(h),\, \lambda^{\rho} v)
\quad \text{for all } h \in H,\; v \in \RR^m.
\]
\item There exist lattices $\Gamma_H < H$ and $\Gamma_0 < \RR^m$ such that $\beta(\Gamma_H \times \Gamma_H) \subset \Gamma_0$ and \(\Gamma = \Gamma_H \times_\beta \Gamma_0\).
\item There exist identity neighborhoods $V_H \subset H$ and $V_0 \subset \RR^m$ such that \(V = V_H \times_\beta V_0\).
\item For every $\lambda > 1$ satisfying $D_{\lambda}(\Gamma) \subset \Gamma$, define $F_0(\lambda) := \lambda^\rho V_0 \cap \Gamma_0$. Then
\[
\bigcap_{x \in F_0(\lambda)} (x + F_0(\lambda)) \neq \emptyset.
\]
\end{enumerate}
\end{definition}

Conditions (A1)--(A3) motivate the term split dilation datum, and they imply that
\begin{equation}
\label{Eq-FHF0}
D_{\lambda}(V) \cap \Gamma = F_H(\lambda) \times_\beta F_0(\lambda), \quad \text{where } F_H(\lambda) := D'_\lambda(V_H) \cap \Gamma_H \text{ and } F_0(\lambda) := \lambda^\rho V_0 \cap \Gamma_0.
\end{equation}
Condition (A4) is essential for establishing the existence of a non-periodic substitution datum; see Proposition~\ref{Prop-GoodSubst}. 

\begin{example}
\label{Ex-Heisen_SplitDil}
The dilation datum defined in Example~\ref{Ex-HeisenbergIntro} for the Heisenberg group $G=\mathbb{H}_3(\RR)$ is a split dilation datum. The normalized cocycle on \( H = \RR^2 \) is given by
\[
\beta : \RR^2 \times \RR^2 \to \RR, \qquad
\beta\big((x,y),(x',y')\big) = \tfrac{1}{2}(xy' - yx').
\]
The dilation on \( H \) is \( D'_\lambda(x, y) = (\lambda x, \lambda y) \), and \( \rho = 2 \). We set \( \Gamma_H = (2\ZZ)^2 \), \( \Gamma_0 = 2\ZZ \), and choose the fundamental domains \( V_H = [-1,1)^2 \) and \( V_0 = [-1,1) \). For each integer \( \lambda > 1 \), we have
\[
D_\lambda(\Gamma) \subsetneq \Gamma \qand F_0(\lambda) = [-\lambda^2, \lambda^2) \cap 2\ZZ.
\]
A direct computation shows that
\[
\bigcap_{x \in F_0(\lambda)} (x + F_0(\lambda)) = \{a\} \neq \emptyset,
\]
with $a=-2$ if $\lambda$ is even and $a=0$ if $\lambda$ is odd, so condition (A4) is satisfied. 
\end{example}
The following general existence result will be proved at the very end of Section 8.
\begin{theorem}[Existence of split dilation data]
\label{Thm-LieMain} 
If $G$ is a substitution group of dimension $\geq 2$, then $G \cong H \times_\beta \RR^m$ for some $m \in \NN$, non-compact Lie group $H$ and normalized cocycle $\beta$, and $H \times_\beta \RR^m$ admits a split dilation datum $\mathcal D = (H \times_\beta \RR^m, d, (D_\lambda)_{\lambda >0}, \Gamma, V)$ of exact polynomial volume growth.
\end{theorem}
\begin{no} In view of the theorem we assume from now on that \( G = H \times_\beta \RR^m \) and that \( \mathcal{D} \) is a split dilation datum. We further assume that \( \Gamma_H, \Gamma_0, V_H, V_0,\) and \( \rho \) are chosen as in Definition~\ref{Def-SplitDilationDatum}, and for any \( \lambda > 1 \), we define the sets \( F_H(\lambda) \subset \Gamma_H \) and \( F_0(\lambda) \subset \Gamma_0 \) as in equation~\eqref{Eq-FHF0}. Finally, we fix a finite alphabet \( \mathcal{A} \) with \( |\mathcal{A}| \geq 2 \).
\end{no}
\begin{lemma}[Choice of stretch factor]
\label{Lem-GoodStretchFactor}
For every split dilation datum $\mathcal D$ there exists $\lambda_0>1$ with the following properties:
\begin{enumerate}[(a)]
\item $D_{\lambda_0}(\Gamma) \subsetneq \Gamma$.
\item $\lambda_0$ is $V$-sufficient.
\item $|F_H (\lambda_0)| \geq 3$ and $|F_0(\lambda_0)| \geq |\Aa|+1$.
\end{enumerate}
\end{lemma}

\begin{proof} By \S \ref{DInclusionProper2} there exists  \( \lambda > 1 \) such that \( D_\lambda(\Gamma) \subsetneq \Gamma \), and since \( \Gamma_H \subset H \) and \( \Gamma_0 \subset \RR^m \) are both relatively dense, there exist bounded sets \( B_H \subset H \) and \( B_0 \subset \RR^m \) such that \[ |\Gamma_H \cap B_H| \geq 3 \qand |\Gamma_0 \cap B_0| \geq |\mathcal{A}| + 1.\] 
Now the product \( B_H \times_\beta B_0 \subseteq G \) is bounded, and we choose \( R > 0 \) with \(B_H \times_\beta B_0 \subseteq B(e, R)\). Since $V$ is bounded and contains an identity neighbourhood by (A3), we can choose an inner radius $r_-$ and outer radius $r_+$ for $V$. Since \( \lambda > 1 \), we can choose \( N \in \NN \) such that
\[
\lambda^N > 1 + \tfrac{r_+}{r_-}
\quad \text{and} \quad
\lambda^N r_- > R.
\]
If we now set \( \lambda_0 := \lambda^N \), then (a) holds by choice of $\lambda$ and (b) hold by Proposition~\ref{Prop-SuffCriter:suff_large-old}. Finally, 
\[
B_H \times_\beta B_0 
	\subseteq B(e, R) 
	\subseteq B(e, \lambda_0 r_-) 
	\subseteq D(B(e, r_-)) 
	\subseteq D_{\lambda_0}(V).
\]
Hence, by construction of \( B_H \) and \( B_0 \), we obtain \( |F_H(\lambda_0)| \geq 3 \) and \( |F_0(\lambda_0)| \geq |\mathcal{A}| + 1 \).
\end{proof}
\begin{no}
We refer to $\lambda_0$ as in Lemma~\ref{Lem-GoodStretchFactor} as a \emph{good stretch factor}. From now on we fix a good stretch factor $\lambda_0$; we then abbreviate
\[
D := D_{\lambda_0}, \quad F_H := F_H(\lambda_0) \qand F_0:= F_0(\lambda_0)
\]
so that by Lemma~\ref{Lem-GoodStretchFactor}~(c) and \eqref{Eq-FHF0} we have
\begin{equation}
\label{Eq-Dil_GoodSubst}
D(V)\cap\Gamma=F_H\times_\beta F_0 \quad \text{with} \quad |F_H| \geq 3 \qand |F_0| \geq |\Aa| +1.
\end{equation}
\end{no}
\begin{construction}
\label{Const-DVGammaF} 
Since \( |F_H| \geq 3 \), we can fix elements \( \gamma_1, \gamma_2 \in F_H \setminus \{e\} \) with \( \gamma_1 \neq \gamma_2 \). These choices are fixed once and for all.

By (A4), there exists an element \( x_1 \in \bigcap_{x \in F_0} (x + F_0) \), which we also fix. Since \( 0 \in F_0 \), it follows that \( x_1 \in F_0 \). Moreover, as \( |F_0| \geq |\mathcal{A}| + 1 \geq 2 \), we can choose and fix another element \( x_2 \in F_0 \setminus \{x_1\} \). Since \( \mathcal{A} \neq \emptyset \), we also fix some \( a \in \mathcal{A} \).

Again using \( |F_0| \geq |\mathcal{A}| + 1 \), we can define and fix an injective map
\[
\mathcal{A} \setminus \{a\} \hookrightarrow F_0 \setminus \{x_1, x_2\}, \quad c \mapsto x_c.
\]

With these choices fixed, we define the following subsets of \( F_H \times_\beta F_0 \):
\begin{align*}
\Xi_a &:= \big(\{e\} \times_\beta F_0\big) \,\sqcup\, \big(\{\gamma_1\} \times_\beta (F_0 \setminus \{x_1\})\big), \\
\Xi_o &:= \{(\gamma, x_1) \mid \gamma \in F_H \setminus \{e\} \}, \\
\Xi &:= \Xi_a \,\sqcup\, \Xi_o \,\sqcup\, \{(\gamma_2, x_c) \mid c \in \mathcal{A} \setminus \{a\}\} \,\sqcup\, \{(\gamma_2, x_2)\}.
\end{align*}
By construction, we have \( \Xi \subseteq F_H \times_\beta F_0 \), and thus, using~\eqref{Eq-Dil_GoodSubst}, also \( \Xi \subseteq D(V) \cap \Gamma \).
\end{construction}

\begin{definition}
\label{Def-GoodSubst} 
Let $\mathcal D$ be a split dilation datum with substitution datum $\Ss=(\Aa,\lambda_0,S_0)$ where $\lambda_0$ is a good stretch factor.
The substitution rule $S_0: \Aa \to \Aa^{D(V) \cap \Gamma}$ is called \emph{good} if the following conditions are satisfied, using the notation from Construction~\ref{Const-DVGammaF}:
\begin{enumerate}[(G1)]
\item For every $b \in \mathcal A$ we have $S_0(b)(\Xi_a) = \{a\}$ and $S_0(b)(\Xi_o) \in\Aa\setminus\{a\}$.
\item For every $b \in \Aa$ and $c\in\Aa\setminus\{a\}$, we have $S_0(b)(\gamma_2,x_c) = c$.
\item For every $b \in \Aa$ we have $S_0(b)(\gamma_2,x_2) = b$.
\end{enumerate}
\end{definition}

\begin{remark} 
For every split dilation datum, every alphabet $\mathcal{A}$ with $|\mathcal{A}| \geq 2$, and every good stretch factor $\lambda_0$, there exist numerous choices for good substitution rules, as only a small subset of the values of $S_0$ is prescribed. In particular, if $|\mathcal{A}| \geq 3$, the values $S_0(b)(\Xi_o)$ can be chosen arbitrarily, provided they differ from the distinguished letter $a$.

Condition (G1) ensures that the substitution rule is non-periodic. Conditions (G2) and (G3) guarantee primitivity and injectivity of $S_0$, respectively. These latter conditions may be relaxed if primitivity or injectivity is ensured by other means.
\end{remark}

Theorem~\ref{Thm-LieMain} is now a consequence of the following observation:

\begin{proposition} 
\label{Prop-GoodSubst}
Let $\mathcal D$ be a split dilation datum with substitution datum $\Ss=(\Aa,\lambda_0,S_0)$ where $\lambda_0$ is a good stretch factor.
If $S_0$ is a good substitution rule, then the associated substitution $\Ss$ is primitive and non-periodic.
\end{proposition}

\begin{proof}
Firstly, we show that $(\mathcal{A}, \lambda_0, S_0)$ is primitive with exponent $L = 1$, i.e.\ that for all $b,c \in \mathcal{A}$ we have $P_c \prec S^L(P_b)$. This follows since for each $b \in \mathcal{A}$, by (G1) and (G2),
\[
S_0(b)(e,0) = a \qand S_0(b)(\gamma_2, x_c) = c, \quad \text{for } c \in \mathcal{A} \setminus \{a\}, 
\]
where we used that $(e,0) \in \{e\} \times_\beta F_0 \subseteq \Xi_a$.

\medskip

Secondly, we claim that $S_0$ is injective. This follows directly from condition (G3) stating $S_0(b)(\gamma_2, x_2) = b$ for all $b \in \mathcal{A}$.

\medskip

Thirdly, let $\eta = (g, x) \in (D(V) \cap \Gamma) \setminus \{(e,0)\}$ and $b,c \in \mathcal{A}$. It remains to show that
\[
\exists\, \eta' \in \eta^{-1} D(V) \cap D(V) \cap \Gamma : \quad (\eta^{-1} S(P_c))(\eta') \neq S(P_b)(\eta').
\]

Assume first that $g = e$, so $x \neq 0$. Since $D(V) \cap \Gamma = F_H \times_\beta F_0$, we then have $x \in F_0 \setminus \{0\}$. As $(\gamma_1, x_1) \in D(V) \cap \Gamma$, define
\[
\eta' := (\gamma_1, x_1 - x) = \eta^{-1} (\gamma_1, x_1) \in \eta^{-1} (D(V) \cap \Gamma) = \eta^{-1} D(V) \cap \Gamma.
\]
By (A4), $x_1 \in x + F_0$ and thus $x_1 - x \in F_0$, so $\eta' = (\gamma_1, 0)(e, x_1 - x) \in F_H \times_\beta F_0 = D(V) \cap \Gamma$. This shows that $\eta' \in \eta^{-1} D(V) \cap D(V) \cap \Gamma$. Furthermore, since $x \neq 0$, we have $\eta' \in \{\gamma_1\} \times_\beta (F_0 \setminus \{x_1\}) \subseteq \Xi_a$, while $(\gamma_1, x_1) \in \Xi_o$. We deduce that
\[
(\eta^{-1} S(P_c))(\eta') = S(P_c)(\eta \eta') = S(P_c)(\gamma_1, x_1) \in \mathcal{A} \setminus \{a\} \not\ni a = S(P_b)(\eta').
\]
This finishes the case $g = e$.

\medskip

Now assume $g \neq e$. Set $\eta' := (e, x_1 - x)$. Again by (A4), $x_1 \in x + F_0$ implies $x_1 - x \in F_0$, hence $\eta' \in \Xi_a \subseteq D(V) \cap \Gamma$. Moreover,
\[
\eta \eta' = (g, x_1) \in \Xi_o \subset D(V) \cap \Gamma,
\]
since $g \in F_H \setminus \{e\}$. Thus $\eta' \in \eta^{-1} D(V)$, and using (G1) we obtain
\[
(\eta^{-1} S(P_c))(\eta') = S(P_c)(\eta \eta') = S(P_c)(g, x_1) \in \mathcal{A} \setminus \{a\} \not\ni a = S(P_b)(\eta'),
\]
which completes the proof.
\end{proof}

\begin{figure}[htb]
\includegraphics[scale=0.8]{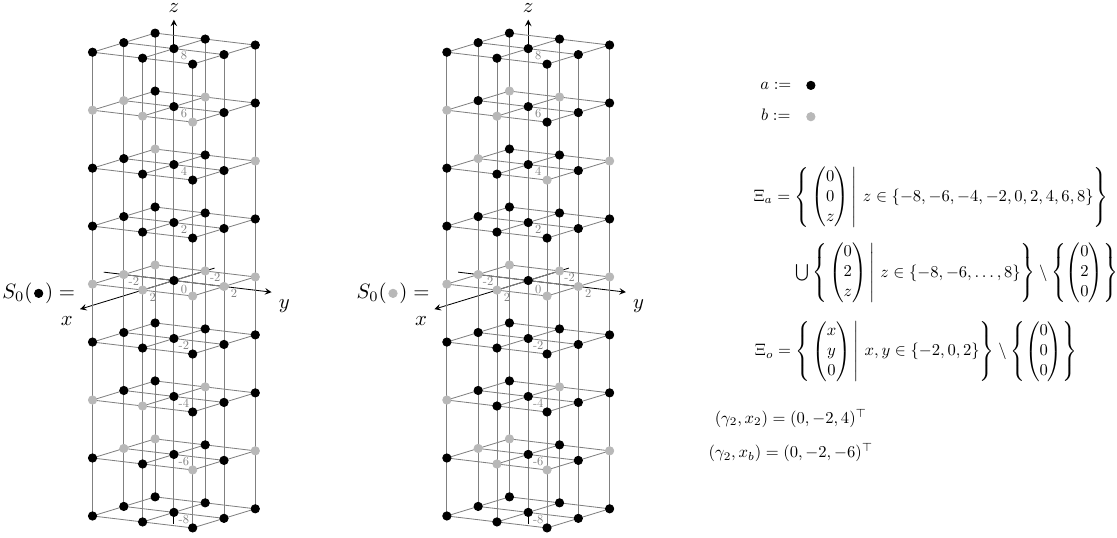}
\caption{The substitution rule $S_0$ is plotted together with the sets $\Xi_a$ and $\Xi_o$. This shows that $S_0$ is a good substitution rule.}
\label{Fig-Heisenberg_GoodSubst}
\end{figure}

\begin{example}
\label{Ex-HeisenbergGoodSubst} 
We now show that the substitution datum \( \Ss = (\Aa, \lambda_0, S_0) \) defined in Example~\ref{Ex-HeisenbergSubstitution} is a good substitution rule for the Heisenberg group \( \mathbb{H}_3(\RR) \). Recall that \( \Aa = \{a, b\} \), \( \lambda_0 = 3 \), and the substitution rule \( S_0 \) is given in Figure~\ref{Fig-Heisenberg_GoodSubst}. We use the splitting from Example~\ref{Ex-Heisen_SplitDil}. As shown in Example~\ref{Ex-HeisenbergSubstitution}, the stretch factor \( \lambda_0 = 3 \) is sufficiently large relative to \( V \).
For this choice of \( \lambda_0 \), we compute:
\[
D(V)\cap \Gamma = \{-2, 0, 2\} \times \{-2, 0, 2\} \times \{-8, -6, -4, -2, 0, 2, 4, 6, 8\},
\]
so in the notation introduced earlier we have
\[
F_H = \{-2, 0, 2\}^2, \qquad F_0 = \{-8, -6, -4, -2, 0, 2, 4, 6, 8\}.
\]
Since \( |F_H| \geq 3 \), \( |F_0| \geq |\Aa| + 1 = 3 \), and \( \bigcap_{x \in F_0} (x + F_0) = \{0\} \), it follows that \( \lambda_0 \) is a good stretch factor.

We now fix the following elements: 
\[
\gamma_1 := (0, 2)^\top, \qquad \gamma_2 := (0, -2)^\top \quad \text{in } F_H \setminus \{e\},
\]
\[
x_1 := 0 \in \bigcap_{x \in F_0} (x + F_0), \qquad x_2 := 4, \qquad x_b :=-6\in F_0 \setminus \{x_1\}.
\]
With this setup, it is straightforward to verify that the substitution rule \( S_0 : \Aa \to \Aa^{D(V) \cap \Gamma} \) from Example~\ref{Ex-HeisenbergSubstitution} is good in the sense of Definition~\ref{Def-GoodSubst}, see also Figure~\ref{Fig-Heisenberg_GoodSubst}. Thus, by Proposition~\ref{Prop-GoodSubst}, the substitution datum \( (\Aa, 3, S_0) \) over \(\mathcal{D}_{\mathbb{H}} := (\mathbb{H}_3(\RR), d, (D_\lambda)_{\lambda > 0}, \Gamma, V)\) is primitive and non-periodic.
\end{example}

\section{Strong aperiodicity}
\label{Sec-StrongAper}
\begin{no} Recall from Section \ref{Sec-SubstDynamSyst} that if $\Ss$ is a non-periodic substitution datum over a dilation datum $\Dd$ with associated substitution map $S$, then the associated subshift $\Omega(S)$ is weakly aperiodic. Here we are interested in criteria which guarantee that it is actually strongly aperiodic in the sense that every element of $\Omega(S)$ has trivial stabilizer. In the abelian case, strong aperiodicity follows immediately from weak aperiodicity and minimality. In fact, we have the following more general result:
\end{no}
\begin{lemma}
\label{Lem-APcenter}
Let $H$ be a group with center $Z$, acting minimally by homeomorphisms on a space $\Omega$. If there exists an $\omega \in \Omega$ such that $\mathrm{Stab}_Z(\omega) = \{e\}$, then $\mathrm{Stab}_Z(\rho) = \{e\}$
for all $\rho \in \Omega$. In particular, any weakly aperiodic minimal action of an abelian lcsc group $H$ is automatically strongly aperiodic.
\end{lemma}
\begin{proof}
Let $\omega \in \Omega$ be as in the lemma and fix $\rho \in \Omega$. 
Since the action $H \curvearrowright \Omega$ is minimal, there exists a sequence $(h_n)$ in $H$ with $h_n \rho \to \omega$ as $n \to \infty$. 
By continuity of the action, for any $z \in Z$ with $z \rho = \rho$ we have
\[
z \omega 
= z \big( \lim_{n \to \infty} h_n \rho \big) 
 = \lim_{n \to \infty} (z h_n)\rho 
 = \lim_{n \to \infty} (h_n z)\rho 
 = \lim_{n \to \infty} h_n (z \rho) 
 = \lim_{n \to \infty} h_n \rho 
 = \omega.
\]
Since $\mathrm{Stab}_Z(\omega)\cap Z=\{e\}$ by assumption, this forces $z=e$, as claimed.
\end{proof}
\begin{corollary}\label{AbelianStrong} If $\Ss$ is a non-periodic primitive substitution datum over a dilation datum $\Dd$ with associated substitution map $S$ over an abelian lcsc group, then $\Gamma \curvearrowright \Omega(S)$ is strongly aperiodic.
\end{corollary}

\begin{remark} Clearly the proof of Corollary \ref{AbelianStrong} does not extend directly to non-abelian situations; nevertheless Lemma \ref{Lem-APcenter} is useful also in non-abelian situations, since it guarantees absence of periods in central directions. Since nilpotent Lie groups (in particular, substitution groups) arise as iterated central extensions of abelian Lie groups, one can hope to establish a version of Corollary \ref{AbelianStrong} for general substitution groups by induction on the nilpotency degree. Since the details of such an inductive argument become rather technical, we confine ourselves to the case of $2$-step nilpotent groups.
\end{remark}
\begin{no}\label{SpecialSub} From now on let $G$ be a substitution group of dimension $d \geq 2$ which is $2$-step nilpotent. Given $g,h \in G$ we denote by $[g,h] := ghg^{-1}h^{-1}$ their commutator; these commutators then generate a normal subgroup $[G,G]$ of $G$, called the \emph{commutator subgroup}. We now need a result from the structure theory of $2$-step nilpotent substitution groups (see Example \ref{Ex-Splitting_2Step} below): Up 
to isomorphism we may assume that $G$ is a central extension of the form $G = H \times_\beta \RR^m$ such that
$H \cong (\RR^{d-m}, +)$ is abelian and non-compact, and the commutator subgroup $[G,G]$ is contained in the central subgroup $\{e\} \times_\beta \RR^m$. We refer to a split dilation datum $\mathcal D = (G, d, (D_\lambda)_{\lambda >0}, \Gamma, V)$ with respect to this special kind of splitting as a \emph{special dilation datum}. 
\end{no}
\begin{example}\label{HeisenbergSpecial} The Heisenberg group is two-step nilpotent, and the split dilation datum from Example~\ref{Ex-Heisen_SplitDil} is special. Indeed, if we write the Heisenberg group as $H \times_\beta \RR$ with $\beta$ as in Example~\ref{Ex-Heisen_SplitDil}, then $H\cong (\RR^2, +)$ is abelian and the commutator subgroup is given by $\{0\} \times_\beta \RR$, which coincides with the center.
\end{example}
\begin{no}
From now on we fix a splitting $G = H \times_\beta \RR^m$ as above, a special substitution datum $\mathcal D = (G, d, (D_\lambda)_{\lambda >0}, \Gamma, V)$ and a  primitive non-periodic substitution datum $\Ss = (\mathcal A, \lambda_0, S_0)$ over $\Dd$. We write $H$ additively and in particular denote by $0_H$ ist neutral element. By assumption the dilations $D_\lambda$ split as 
\[
D_{\lambda}(h,v) = (D_{\lambda}'(h),\, \lambda^{\rho} v)
\]
for some $\rho > 0$ and we have $D_{\lambda_0}(\Gamma) \subsetneq \Gamma$. Our goal is to establish Theorem \ref{Thm-strongAP} from the introduction, which claims that $\Gamma \curvearrowright \Omega(S)$ is strongly aperiodic. We need the following observation:
\end{no}
\begin{lemma}
\label{lem-Jenga}
Let $Z$ denote the center of $G$ and $\lambda>1$ be a stretch factor. Then for every $\gamma \in \Gamma \setminus Z$ there exists $N_0 \in \NN$ such that
\[
\gamma\, D_{\lambda}^{n}\big(\{0_H\} \times V_0\big) \,\cap\, D_{\lambda}^{n}(\Gamma) \;=\; \emptyset
\quad \text{for all } n \geq N_0.
\]
\end{lemma}

\begin{proof}
Fix $\gamma = (\gamma_H, \gamma_0) \in \Gamma \setminus Z$ and set $D:=D_\lambda$.
Since $\gamma \notin Z$, we have $\gamma_H \neq 0_H$, and hence there exists $N_0 \in \NN$ such that 
$\gamma_H \notin (D_{\lambda}')^{n}(\Gamma_H)$ for all $n \geq N_0$. 
Now fix any $n \ge N_0$ and observe that $D^{n}(\Gamma) = (D'_{\lambda})^n(\Gamma_H) \times \lambda^{\rho n} \Gamma_0$.
Since $\beta$ is a normalized cocycle, we have $\beta(e_H, t) = 0$ for all $t \in \RR^m$, see Lemma~\ref{Lem-NiceCocycle}. Thus,
\[
\gamma\, D^{n}\big(\{0_H\} \times V_0\big) 
	= (\gamma_H, \gamma_0)\big(\{0_H\} \times \lambda^{\rho n} V_0\big) 
	= \{\gamma_H\} \times (\gamma_0 + \lambda^{\rho n} V_0\big).
\]
Since $\gamma_H \notin (D_{\lambda}')^n(\Gamma_H)$, the lemma follows.
\end{proof}

\begin{proof}[Proof of Theorem~\ref{Thm-strongAP}]
Since $\mathcal S$ is primitive, $\Omega(S)$ is minimal by Theorem~\ref{Thm-LinearRep}. 
By Proposition~\ref{Prop-LegalFixpoints}, there exist $\omega \in \Omega(S)$ and $k \in \NN$ with $S^{k}(\omega) = \omega$, and we fix such a pair $(\omega, k)$ once and for all. 
Recall from Theorem~\ref{Thm-weakAP-ex} that $\omega$ has trivial $\Gamma$-stabilizer, and fix $\rho \in \Omega(S)$ and $\gamma \in \mathrm{Stab}_\Gamma(\rho)$. We need to show that $\gamma = e$. 

If $\gamma \in Z$, then $\gamma = e$ by Lemma~\ref{Lem-APcenter} and minimality of $\Omega(S)$; hence assume for contradiction that $\gamma \in \Gamma \setminus Z$. 
By Lemma~\ref{lem-Jenga}, choose $k_0 \in \NN$ such that, for $n := k_0 k$, 
\[
S^{n}(\omega) = \omega 
\qquad\text{and}\qquad
\gamma D^{n}\big(\{0_H\}\times V_0\big) \,\cap\, D^{n}(\Gamma) = \emptyset.
\]

Since $\Omega(S)$ is minimal, there exists $(\eta_j)_{j\in\NN} \subset \Gamma$ with $\eta_j^{-1}\omega \to \rho$. 
As $\Gamma = D^n(\Gamma)\big(D^n(V)\cap\Gamma\big)$ by~\eqref{PreExtension} and $D^n(V)\cap\Gamma$ is finite, we may assume $\eta_j = D^n(\tilde{\eta}_j)\eta_0$ for some $\tilde{\eta}_j\in\Gamma$, $\eta_0\in D^n(V)\cap\Gamma$. 
Define
\[
z := [\gamma^{-1},\eta_0] = \gamma^{-1}\eta_0\gamma\eta_0^{-1}.
\]
Then $z \in [G,G]\subseteq \{0_H\}\times\RR^m$ and, since $z\in\Gamma$, also $z\in\{0_H\}\times \Gamma_0$. 
By the splitting of the dilation and since $V=V_H\times_\beta V_0$, there exist $z_0\in D^n(\{0_H\}\times V_0)$ and $\tilde{z}\in\{0_H\}\times \Gamma_0$ such that $z = D^n(\tilde{z})z_0$.
 
\item Since $z_0\in [G,G]\subseteq Z(G)$ commutes with all elements in $G$ we have
\[
\eta_0\gamma\eta_0^{-1}
	= \eta_0\gamma\eta_0^{-1}z^{-1}z
	= \eta_0\gamma\eta_0^{-1}(\eta_0\gamma^{-1}\eta_0^{-1}\gamma)z
	= \gamma D^n(\tilde{z})z_0
	= \gamma z_0 D^n(\tilde{z}).
\]

Combining this with $\gamma\rho=\rho$, $S^n(\omega)=\omega$, $S^n(g\omega)=D^n(g)S^n(\omega)$ (see Proposition~\ref{Prop-SubstitutionProperties}) and continuity of the $\Gamma$-action yields
	\begin{align*}
	S^n\big(\tilde{\eta}_j^{-1}\omega\big) = D^n(\tilde{\eta}_j^{-1}) S^n(\omega) = \eta_0 \eta_0^{-1} D^n(\tilde{\eta}_j^{-1}) \omega = \eta_0 \eta_j^{-1}\omega 
		\;\overset{j\to\infty}{\longrightarrow}\; &\eta_0 \rho\\	
							& \;\; \rotatebox{90}{=}\\
	\gamma z_0  S^n\big(\tilde{z}\tilde{\eta}_j^{-1} \omega\big) = \eta_0 \gamma \eta_0^{-1} D^n(\tilde{\eta}_j^{-1}) S^n(\omega) = \eta_0 \gamma \eta_j^{-1}\omega 
		\;\overset{j\to\infty}{\longrightarrow}\; &\eta_0 (\gamma \rho).
	\end{align*}

By definition of the product topology from $\Aa^\Gamma$, there exists $j_0\in\NN$ such that 
\[
\big(\gamma z_0 S^n(\tilde{z}\tilde{\eta}_j^{-1}\omega)\big)|_{V(n)}
	= \big(S^n(\tilde{\eta}_j^{-1}\omega)\big)|_{V(n)}, \qquad j\ge j_0.
\]
Set $\omega_1 := \tilde{z}\tilde{\eta}_{j_0}^{-1}\omega$, $\omega_2 := \tilde{\eta}_{j_0}^{-1}\omega$, and $\eta := (\gamma z_0)^{-1}$. Then
\[
(\eta^{-1}S^n(\omega_1))|_{V(n)} = (S^n(\omega_2))|_{V(n)}.
\]
Since $\gamma D^n(\{e_H\}\times V_0)\cap D^n(\Gamma)=\emptyset$ and $z_0\in D^n(\{e_H\}\times V_0)$, we have $\gamma z_0\notin D^n(\Gamma)$, hence $\eta^{-1}\in\Gamma\setminus D^n(\Gamma)$ and $\eta\in\Gamma\setminus D^n(\Gamma)$. 
This contradicts Proposition~\ref{prop-few fixedpoints}, since $(\Aa,\lambda_0,S_0)$ is non-periodic.
\end{proof}
At this point we have also established Theorem~\ref{Thm-HeisenbergMain} from the introduction:
\begin{proof}[Proof of Theorem~\ref{Thm-HeisenbergMain}]
The dilation datum from Example~\ref{Ex-Heisen_SplitDil} is special, by Example \ref{HeisenbergSpecial}, and by Lemma~\ref{Lem-GoodStretchFactor} there exists a good stretch factor $\lambda_0$. Via Construction~\ref{Const-DVGammaF}, we obtain a good substitution datum. Now by Proposition~\ref{Prop-GoodSubst}, the associated substitution $\mathcal{S}$ is primitive and non-periodic, and the associated subshift is
minimal, linearly repetitive and uniquely ergodic by Theorem~\ref{Thm-LinearRep} and even strongly aperiodic by Theorem~\ref{Thm-strongAP}. 
\end{proof}

\section{Applications to Delone dynamical systems} 
\label{Sec-Delone}

\begin{no}\label{WeightedDelone} So far, all our results were concerned with subshifts; however, there is a standard way to transfer our results to the context of Delone dynamical systems, which we briefly recall. We denote by $\mathrm{Del}(G)$ the space of Delone sets in $G$, equipped with the Chabauty--Fell topology, and refer to \cite{BjHaPo18} for general background on Delone sets in groups. We fix a lcsc group $G$, uniform lattice $\Gamma < G$ and a finite alphabet $\Aa$ together with an embedding $\iota: \Aa \hookrightarrow (0, \infty)$. For concreteness' sake let us assume that actually $\Aa = \{a,b\}$ with $\iota(a) = 1$ and $\iota(b) = 2$. If we denote by $\mathcal R(\Gamma)$ the space of Radon measures on $\Gamma$, then we obtain a continuous $\Gamma$-equivariant embedding
\[\Aa^\Gamma \hookrightarrow \mathcal R(\Gamma), \quad \omega \mapsto \Lambda_\omega := \sum_{\gamma \in \Gamma} \iota(\omega(\gamma)) \delta_\gamma,\]
whose image consists of ``weighted Delone sets''. For every $\omega \in \Aa^\Gamma$ we thus obtain an isomorphism of topological dynamical systems (cf.\ \cite[Lemma~5.1 and Proposition~5.2]{BHP25-LR})
\begin{equation}\label{GammaHulls1}
(\Gamma \curvearrowright \Omega_\omega) \cong (\Gamma \curvearrowright \Omega^\Gamma_{\Lambda_\omega}),\end{equation}
where $\Omega_\omega$ and $\Omega^\Gamma_{\Lambda_\omega}$ denote the orbit closures of $\omega$ and $\Lambda_\omega$ respectively. To get rid of the weights we make the following observation. 
\end{no}
\begin{proposition}\label{Prop-NoWeights} Let $\omega \in \Aa^\Gamma$ be linearly repetitive.
\begin{enumerate}[(a)]
\item For every $\omega' \in \Omega_\omega$ the subset $\Lambda'_{\omega'}  := \{\gamma \in \Gamma \mid \omega'(\gamma) = a\} \subset \Gamma$ is a linearly repetitive Delone set in $G$.
\item If $\Omega^\Gamma_{\Lambda'_\omega}$ denotes the $\Gamma$-orbit closure of $\Lambda'_\omega$ in $\mathrm{Del}(G)$, then there is an isomorphism 
\begin{equation}\label{GammaHulls2}(\Gamma \curvearrowright \Omega_\omega) \cong (\Gamma \curvearrowright \Omega^\Gamma_{\Lambda'_\omega}), \quad \omega' \mapsto \Lambda'_{\omega'}.\end{equation}
\end{enumerate}
\end{proposition}
\begin{proof} (a) Let $\omega' \in \Omega$; then $\omega'$ (and hence $\Lambda'_{\omega'}$) is linearly repetitive by Theorem \ref{Thm-LinearRep}, hence there exists a radius $R>0$ such that $P_a \prec \omega|_{B(x,R)}$ for all $x\in G$. This implies that $\Lambda'_{\omega}$ is $R$-relatively dense, and since $\Gamma$ is a discrete subgroup, $\Lambda'_{\omega'} \subset \Gamma$ is uniformly discrete.

\item (b) Denote by $\Lambda''_\omega = \sum_{\lambda \in \Lambda'_\omega} \delta_\lambda \in \mathcal R(\Gamma)$ the associated Dirac comb of $\Lambda'_\omega$; then there is a $\Gamma$-equivariant homeomorphism $\Omega^\Gamma_{\Lambda'_\omega} \to \Omega^\Gamma_{\Lambda''_\omega}$ which maps every Delone set to its associated Dirac comb (cf.\ \cite[Proposition~3.9]{BjHaPoII}). On the other hand,
the map
\[
\Omega^\Gamma_{\Lambda_\omega} \to \mathcal R(\Gamma), \quad \mu \mapsto \sum_{\{x \in H \mid \mu(x) = 1\}} \delta_x
\]
is continuous, $\Gamma$-equivariant, injective and maps $\Lambda_\omega$ to $\Lambda''_\omega$, hence induces a $\Gamma$-equivariant homeomorphism $\Omega^\Gamma_{\Lambda_\omega} \to \Omega^\Gamma_{\Lambda''_\omega}$. Combining these identifications yields the claim.
\end{proof}
\begin{no} Combining the isomorphisms \eqref{GammaHulls1} and \eqref{GammaHulls2} we also obtain an isomorphism
\[
(\Gamma \curvearrowright \Omega^\Gamma_{\Lambda'_\omega}) \cong (\Gamma \curvearrowright \Omega_\omega) \cong  (\Gamma \curvearrowright \Omega^\Gamma_{\Lambda_\omega}).
\]
for every linearly repetitive configuration $\omega \in \Aa^\Gamma$. In this case we can also consider $\Lambda_\omega$ as a Radon measure on $G$ and $\Lambda'_\omega$ as a Delone set in $G$. If we now denote by $\Omega_{\Lambda_\omega}$ and $\Omega_{\Lambda'_\omega}$ the corresponding $G$-orbit closures, then the same argument as in the proof of Proposition \ref{Prop-NoWeights} provides isomorphisms
\[
(G \curvearrowright \Omega_{\Lambda'_\omega}) \cong (G \curvearrowright \Omega_{\Lambda_\omega}).
\]
Moreover, these systems are isomorphic to the $G$-suspension $G \curvearrowright (G \times \Omega_\omega)/\Gamma$ of $\Gamma \curvearrowright \Omega_\omega$.
\end{no}
With this at hand, we can now prove the remaining corollaries from the introduction.
\begin{proof}[Proof of Corollary~\ref{Cor-exLR_Del}]
Since $G$ is a substitution group, it admits a dilation datum $\Dd$ of exact polynomial growth together with a primitive and non-periodic substitution datum $\Ss$ over the alphabet $\mathcal A = \{a,b\}$ (see Theorem~\ref{Thm-ExPrimNonPerSubst}).
If $G$ is $2$-step nilpotent, then we will ensure that $\Dd$ is special in the sense of \S \ref{SpecialSub}. We then denote by $S$ the associated substitution map of $(\Dd, \Ss)$ and by $\Gamma \curvearrowright \Omega(S)$ the associated subshift. By Theorem \ref{MainTheorem} this subshift is minimal and uniquely ergodic and contains an element $\omega$ with trivial $\Gamma$-stabilizer; we fix such an element once and for all. If we now set $\Lambda := \Lambda'_{\omega} \in \mathrm{Del}(G)$, then $\Omega(S) = \Omega_\omega \cong \Omega_{\Lambda}^\Gamma$ and hence $\Omega_{\Lambda} = (G \times \Omega(S))/\Gamma$. Since $d$ has exact polynomial growth it now follows from \cite[Theorem~5.7~(b)]{BHP25-LR} that $\Omega_\Lambda$ is minimal and uniquely ergodic.

Under the isomorphism $\Omega^G_{\Lambda} \cong (G \times \Omega(S))/\Gamma$, the basepoint $\Lambda$ corresponds to the class of $(e,\omega)$. 
If $g \in G$ stabilizes this class, then the first factor gives $g \in \Gamma$, and the second implies that $g$ stabilizes $\omega \in \Omega(S)$. 
Since $\omega$ has trivial $\Gamma$-stabilizer, the $G$-stabilizer of $\Lambda$ is trivial. This shows that $\Omega_\Lambda$ is weakly aperiodic.

If $G$ happens to be $2$-step nilpotent and $\Dd$ is special, then $\Gamma \curvearrowright \Omega(S)$ is free by Theorem~\ref{Thm-strongAP}, and hence the induced $G$-action on $(G \times \Omega(S))/\Gamma$ is free. 
Consequently, $G \curvearrowright \Omega_{\Lambda}$ is strongly aperiodic. This completes the proof.
\end{proof}
\begin{proof}[Proof of Corollary~\ref{Cor-UniqErg_Min_StrAper-Heisenberg}] 
Since the Heisenberg group $\mathbb{H}_3(\RR)$ is a $2$-step nilpotent group substitution group (Example~\ref{Ex-Heisen_SplitDil}), this is a special case of Corollary~\ref{Cor-exLR_Del}.
\end{proof}

\section{Substitution groups and their lattices}\label{Sec-Lie}

In this final section we describe the general structure of substitution groups and their lattices. 
As an application, we establish Theorem~\ref{Thm-LieMain}, which ensures the existence of split dilation data for dilation groups of dimension $\ge 2$. We also establish the existence of special dilation data for $2$-step nilpotent substitution groups as claimed in \S \ref{SpecialSub}. Unlike the preceding sections, this part requires some basic Lie theory; standard references such as \cite{HilgertNeeb, Knapp} cover the necessary background.
Since substitution groups turn out to be nilpotent Lie groups, we first recall some basic facts on nilpotent Lie algebras.

\subsection{Nilpotent and positively-gradable Lie algebras}

%%%%%%%%%%%%%%%%%%%%%%%%%%%%%%
\begin{no}
Let $\KK$ be a field. A \emph{Lie ring} is an abelian group $\L g$ together with a biadditive alternating map $[\cdot, \cdot]: \mathfrak{g}\times \mathfrak{g} \to \mathfrak{g}$ (called the \emph{Lie bracket}) satisfying the \emph{Jacobi identity} \[[X,[Y,Z]] + [Y,[Z,X]] + [Z,[X,Y]] = 0 \quad \text{for all $X,Y, Z \in \L g$.}\] It is called a \emph{$\KK$-Lie algebra} if moreover $\L g$ is a $\KK$-vector space and the Lie bracket is $\KK$-bilinear. If $\L g$ is a $\KK$-Lie algebra, then a $\KK$-vector subspace $\L h < \L g$ is called a \emph{Lie subalgebra} if it is closed under Lie brackets. If $\L h$ and $\L k$ are $\KK$-linear subspaces of a $\KK$-Lie algebra $\L g$, then we define
\[
[\L h, \L k] := \mathrm{span}_{\KK}\left(\{[X,Y] \mid X \in \L h, Y \in \L k\}\right).
\] 
If $(X_1, \dots, X_n)$ is an ordered basis of a $\KK$-Lie algebra $\L g$, then the Lie bracket on $\L g$ is determined by the scalars $\alpha_{ij}^k \in \KK$ such that
\[
[X_i, X_j] = \sum_{k=1}^n \alpha_{ij}^k X_k.
\]
These scalars are referred to as the \emph{structure constants} of $\L g$ with respect to the chosen basis. 
\end{no}
%%%%%%%%%%%%%%%%%%%%%%%%%%%%%%

From now on let $\L g$ be a finite-dimensional $\KK$-Lie algebra over some field $\KK$. We will be mostly interested in the case where $\KK \in \{\QQ, \RR, \CC\}$.
\begin{definition}
A $\KK$-linear map $\L d: \L g \to \L g$ is called a \emph{derivation} if $\L d([X,Y]) = [\L d(X), Y] + [X, \L d(Y)]$ for all $X, Y \in \L g$. A $\KK$-linear map $D: \L g \to \L g$ 
is called an \emph{automorphism} if it is invertible and satisfies $[D(X), D(Y)] = [X,Y]$ for all $X, Y \in \L g$. 
\end{definition}

%%%%%%%%%%%%%%%%%%%%%%%%%%%%%%
\begin{no}
The derivations of $\L g$ form a Lie algebra, denoted $\mathrm{Der}(\L g)$, and the automorphisms of $\L g$ form a group, denoted
$\mathrm{Aut}(\L g)$. If $\KK \in \{\RR, \CC\}$ then $\exp(X) \in \mathrm{Aut}(\L g)$ for every $X \in \mathrm{Der}(\L g)$.

Every $X \in \L g$ defines a derivation $\mathrm{ad}(X)$ given by $\mathrm{ad}(X)(Y) := [X,Y]$, and these are called \emph{inner derivations}. If $\KK \in \{\RR, \CC\}$, then every $X \in \L g$ also defines an \emph{inner automorphism} $\exp(\mathrm{ad}(X)) \in \mathrm{Aut}(\L g)$. 

A Lie subalgebra $\L h < \L g$ is called \emph{characteristic} (respectively an \emph{ideal}) if it is invariant under all (respectively all inner) derivations; we write $\L h \lhd \L g$ to indicate that $\L h$ is an ideal in $\L g$. The \emph{lower central series} of $\L g$ is the sequence of characteristic subalgebras defined by
\[
L^1(\mathfrak{g}) := \mathfrak{g} \qand L^{k+1}(\mathfrak{g}) := [ \mathfrak{g},  L^{k}(\mathfrak{g})], \quad k \geq 1.
\]
We say that $\mathfrak{g}$ is \emph{at most $s$-step nilpotent} if $L^{s+1}(\mathfrak{g}) = \{0\}$  and \emph{$s$-step nilpotent} if moreover $s=0$ or $L^{s}(\mathfrak{g}) \neq \{0\}$. If $\mathfrak{g}$ is at most $1$-step nilpotent, then it is called \emph{abelian}. 
\end{no}
%%%%%%%%%%%%%%%%%%%%%%%%%%%%%%

\begin{definition}
Let $S$ be an abelian semigroup (written additively). A family $\gamma = (\L g_s)_{s \in S}$ of subspaces of $\L g$ is called a \emph{grading} of $\L g$ \emph{over} $S$ (or an
 \emph{$S$-grading} for short) if
\[
\L g = \bigoplus_{s \in S} \L g_s \quad \text{and} \quad [\L g_{s}, \L g_{t}] \subset \L g_{s+t} \quad \text{for all }s,t \in S.
\]
A $\KK$-Lie algebra $\L g$ together with an $S$-grading is called an \emph{$S$-graded $\KK$-Lie algebra}. If such a grading exists, then $\L g$ is called \emph{$S$-gradable}.  

\item If $\gamma = (\L g_s)_{s \in S}$ is an $S$-grading of $\L g$, then the elements of \[\mathcal W_\gamma := \{s \in S \mid \L g _s \neq \{0\}\} \]
are called the \emph{weights} of $\gamma$. If $s \in \mathcal W_\gamma$ then $\L g_s$ is called the associated \emph{weight space}, and elements of $\L g_s$ are called \emph{homogeneous} of weight $s$.
\end{definition}
Note that if $\gamma$ is an $S$-grading of $\L g$, then $\mathcal W_\gamma$ is finite, since $\L g$ is finite-dimensional.

%%%%%%%%%%%%%%%%%%%%%%%%%%
\begin{no}
\label{Par-GradingDerivation} 
Given a subsemigroup $S \subset \KK$, there is a bijection between $S$-gradings of $\L g$ and $\KK$-diagonalizable derivations of $\L g$ with  $\mathrm{spec}(\L d) \subset S$, where $\mathrm{spec}(\L d)$ denotes the collection of eigenvalues of $\L d$.
 Indeed, if $\L g$ is such a derivation then there is a unique $S$-grading $\gamma_{\L d}$ of $\L g$ whose weights are the eigenvalues of $\L d$ and whose weight spaces are the corresponding eigenspaces. 

Conversely, if $\gamma$ is an $S$-grading of $\L g$, then there exists a unique derivation $\L d_\gamma$ such that $\L d_\gamma(X) = sX$  for all $s \in S$ and $X \in \L g_s$, and this derivation is
$\KK$-diagonalizable with $\mathrm{spec}(\L d_\gamma) = \mathcal W_\gamma \subset S$. In the sequel we refer to $\L d_\gamma$ as the \emph{associated derivation} of $\mathfrak{g}$.
\end{no}
%%%%%%%%%%%%%%%%%%%%%%%%%%

\begin{construction} 
Let $A$, $B$ be semigroups and let $\gamma = (\L g_\alpha)_{\alpha \in A}$ be an $A$-grading of $\L g$. We can then construct new gradings of $\L g$ as follows:
\begin{itemize}
\item If $\mathcal W_\gamma \subset A' \subset A$ for some subsemigroup $A' \subset A$ and $f: A' \to B$ is a morphism, then $f_*\gamma := (\bigoplus_{\alpha \in f^{-1}(\beta)}\L g_\alpha)_{\beta \in B}$ is a $B$-grading of $\L g$.
\item If $u \in \mathrm{Aut}(\L g)$, then $u(\gamma) := (u(\L g_\alpha))_{\alpha \in A}$ is another $A$-grading of $\L g$.
\end{itemize}
We observe that if $f: A \to B$ is a morphism and $u \in \mathrm{Aut}(\L g)$, then $u(f_*\gamma) = f_*(u(\gamma))$.
\end{construction}

\begin{definition} 
We say that two gradings $\gamma$ and $\gamma'$ of $\L g$ over semigroups $A$ and $B$ respectively are \emph{equivalent} if $\gamma' = u(f_*\gamma)$ for some isomorphism $f: A \to B$ and some $u \in \mathrm{Aut}(\L g)$. If we can choose $u = \mathrm{Id}$, then they are called \emph{strictly equivalent}.
\end{definition}

For the following lemma we denote by $\NN \subset \QQ_{>0} \subset \RR_{>0}$ the subsemigroups of $(\RR, +)$ given by all positive integers, positive rational numbers and positive real numbers respectively.

\begin{lemma}
\label{Lem-PosGrad} 
We have equivalences
\[
\L g \text{ is $\NN$-gradable} \iff \L g \text{ is $\QQ_{>0}$-gradable} \iff \L g \text{ is $\RR_{>0}$-gradable}.
\]
Moreover, for every $\RR_{>0}$-grading there is an $\NN$-grading with the same weight spaces.
\end{lemma}

\begin{proof} 
The forward implications are clear. If $\gamma$ is a $\QQ_{>0}$-grading, then $W_\gamma \subset \frac{1}{m}\NN$ for some $m \in \NN$ and 
if $f: \frac{1}{m}\NN \to \NN$ denotes multiplication by $m$, then $f_*\gamma$ is an $\NN$-grading with the same weight spaces as $\gamma$. The lemma then follows from the fact that a linear system of equations with $\QQ$-coefficients has a rational solution if it has a real solution and that the rational solutions are dense in the real solution.
\end{proof}

\begin{definition} 
\label{Def-PostiveGradable}
A Lie algebras satisfying the equivalent conditions of Lemma \ref{Lem-PosGrad} is called \emph{positively gradable}.
\end{definition}

%%%%%%%%%%%%%%%%%%%%%%%%%%
\begin{no}
\label{Par-Strat} 
An important class of $\NN$-gradings is given by so-called stratifications. Here an $\NN$-grading $\gamma$ is called a
a \emph{stratification} of $\L g$ if $\gamma = (\L g_1, \dots, \L g_n)$ with \[\L g_j \neq \{0\} \text{ for all $j \in \{1, \dots, n\}$} \qand [\L g_1, \L g_j] = \L g_{j+1} \text{ for all $j \in \{1, \dots, n-1\}$},\] and $\L g$ is called \emph{stratifiable} if it admits such a stratification. While a positively gradable Lie algebra may admit many different $\NN$-gradings, it always admits at most one stratification up to equivalence (see \cite[Rem.\ 1.3]{LeDonne_Primer}).
\end{no}
%%%%%%%%%%%%%%%%%%%%%%%%%%
\begin{example}\label{2stepStrat} Let $\L g$ be a $2$-step nilpotent Lie algebra; set $\mathfrak{g}_2 :=[\L g, \L g]$ and let $\mathfrak{g}_1$ be an arbitrary vector space complement of $\mathfrak{g}_2$ in $\mathfrak{g}$ so that $\mathfrak{g} = \mathfrak{g}_1 \oplus \mathfrak{g}_2$. Since $\mathfrak{g}$ is $2$-step nilpotent, we have $[X,Y] = 0$ for all $X \in \mathfrak{g}$ and $Y \in \mathfrak{g}_2$, which implies that $[\mathfrak{g}_1, \mathfrak{g}_1] = [\mathfrak{g}, \mathfrak{g}] = \mathfrak{g}_2$. This shows that $(\L g_1, \L g_2)$ is a stratification of $\mathfrak{g}$. 

In the context of this example, the uniqueness statement of \S \ref{Par-Strat} together with the fact that $[\L g, \L g] < \L g$ is characteristic implies the following: Every stratification of a $2$-step nilpotent Lie algebra $\L g$ is of the form $(\L g_1, \L g_2 = [\L g, \L g])$ for some complement $\L g_1$ of $[\L g, \L g]$ in $\L g$, and $\mathrm{Aut}(\L g)$ acts transitively on all vector space complements of $[\L g, \L g]$.
\end{example}
%%%%%%%%%%%%%%%%%%%%%%%%%%
\begin{no}
\label{Par-NilpotentStratGradable} 
The classes of nilpotent, stratifiable and positively gradable Lie algebras are related as follows: By definition, every stratifiable Lie algebra is positively gradable.
If a Lie algebra $\L g$ admits an $\NN$-grading $\gamma = (\L g_n)_{n \in \NN}$, then
\[
L^k(\L g) \subset \bigoplus_{n \geq k} \L g_n.
\]
In particular, if we set $s := \max \mathcal W_\gamma$, then $\L g$ is at most $s$-step nilpotent. This shows that \emph{every positively gradable Lie algebra is nilpotent}. 
\item Not every nilpotent Lie algebra is positively gradable (see e.g.\ \cite[Ex.~1.8]{LeDonne_Primer}). However, \emph{every $2$-step nilpotent Lie algebra is stratifiable, hence positively gradable} by Example \ref{2stepStrat}.
\end{no}
%%%%%%%%%%%%%%%%%%%%%%%%%%

\begin{remark} 
By combining various results from the literature, it is possible to obtain a complete classification of positively-gradable real Lie algebras in dimensions $\leq 7$; since this classification does not appear anywhere explicitly, we provide the details in Appendix \ref{AppendixCensus}. We can roughly summarize the classification as follows: Up to isomorphism, there are $33$ indecomposable nilpotent real Lie algebras of dimension $\leq 6$ (including $1$ abelian one) and all of these are positively-gradable. In dimension $7$ there are infinitely many isomorphism classes of nilpotent Lie algebras, some of which are positively-gradable and some are not. More precisely we have the following: every $7$-dimensional nilpotent real Lie algebra which decomposes as a direct sum of two lower-dimensional Lie algebras is positively-gradable. The indecomposable $7$-dimensional nilpotent real Lie algebra have been classified and grouped into $9$ infinite families and $140$ isolated examples \cite{Gong98}, and $7$ of the $9$ families and $126$ of the $140$ isolated examples are positively gradable.
\end{remark}

\begin{example}
\label{Exa-TwoMainExamples} 
We give three examples of $\NN$-graded real Lie algebras:
\begin{enumerate}[(a)]
\item If $\L g$ is an $n$-dimensional abelian real Lie algebra then any vector space decomposition $\L g = \L g_{\alpha_1} \oplus \dots \oplus \L g_{\alpha_k}$ with $k \leq n$ and $\alpha_1, \dots, \alpha_k \in \NN$ defines an $\NN$-grading. This shows that $\NN$-gradings are in general very far from unique (even up to equivalence).
\item The \emph{Heisenberg Lie algebra} is the three-dimensional Lie algebra $\L g$ with basis $(X,Y,Z)$ and Lie brackets determined by $[X,Y] = Z$ and $[X,Z] = [Y,Z] = 0$. This is $2$-step nilpotent, hence stratifiable with stratification $\L g = \L g_1 \oplus \L g_2$, where $X,Y \in \L g_1$ and $Z \in \L g_2$. In the basis $(X,Y,Z)$ the derivation associated with this stratification is given by the diagonal matrix $\mathrm{diag}(1,1,2)$ with diagonal entries $1$, $1$ and $2$.
\item We return to Example~\ref{Ex-3Step}  and consider one of the 9 infinite families of $7$-dimensional nilpotent real Lie algebras, which is labelled (147E)${}_\mu$ in the table \cite{Gong98}. For every $\mu \in \RR$ this family contains a Lie algebra $\L g_\mu$
with generators $X_1, \dots, X_7$, and all non-zero Lie brackets between basis elements are given (up to antisymmetry) by 
\[
[X_1, X_2] = X_4, \; [X_1, X_3] = -X_6, \; [X_1, X_5] = -X_7, \; [X_2, X_3] = X_5
\]
\[
[X_2, X_6] = \mu X_7, \quad [X_3, X_4] = (1-\mu) X_7,
\]
All of these Lie algebras are $3$-step nilpotent, and we have $\L g_{\mu_1} \cong \L g_{\mu_2}$ if and only if 
\[
\frac{(1-{\mu_1} + {\mu_1}^2)^3}{{\mu_1}^2(1-{\mu_1})^2} = \frac{(1-{\mu_2} + {\mu_2}^2)^3}{{\mu_2}^2(1-{\mu_2})^2}.
\]
Each $\L g_\mu$ is positively gradable; in fact, there is a unique grading such that $X_1, X_2, X_3$ are homogeneous of degree $1$, $X_4, X_5, X_6$ are homogeneous of degree $2$ and $X_7$ is homogeneous of degree $3$. In the basis $(X_1, \dots, X_7)$ the derivation associated with this grading is represented by the diagonal matrix $\mathrm{diag}(1,1,1,2,2,2,3)$.
\end{enumerate}
\end{example}

%%%%%%%%%%%%%%%%%%%%%%%%%%
\begin{no}
If $\LL/\KK$ is a field extension, then $\L g_{\LL} :=  \L g \otimes \LL$ is an $\LL$-Lie algebra with respect to the Lie bracket which on elementary tensors is given by
\begin{equation}
\label{Eq-TensorBracket}
[X \otimes \lambda, Y \otimes \mu] = [X,Y] \otimes \lambda \mu \quad (X, Y \in \L g, \lambda, \mu \in \LL).
\end{equation}
We call $\L g_{\LL}$ the \emph{$\LL$-ification} of $\L g$. Conversely, if $\L h_{\LL}$ is an $\LL$-Lie algebra, then a subset $\L h \subset \L h_{\LL}$ is called a \emph{$\KK$-form} of $\L h_{\LL}$ if it is a $\KK$-Lie subalgebra and if the bilinear map $\L h \times \LL \to \L h_{\LL}$, $(X, \lambda) \mapsto \lambda X$ induces a linear isomorphism $ \L h \otimes \LL \to \L h_{\LL}$. 

We say that an $\LL$-Lie algebra is \emph{definable over $\KK$} if it admits a $\KK$-form; equivalently, it admits a basis with structure constants $\alpha_{ij}^k \in \KK$, and any such basis is called a \emph{$\KK$-basis}.
It is obvious that a $\L g$ is nilpotent (of degree $s$) if and only if its $\LL$-ification $\L g_{\LL}$ is nilpotent (of degree $s$). We note that if $\gamma=(\L g_n)_{n \in \NN}$ is an $\NN$-grading of $\L g$, then $\gamma_{\LL} := (\L g_n \otimes_{\KK} \LL)_{n \in \NN}$ is an $\NN$-grading of $\L g_{\LL}$, hence if $\L g$ is positively gradable, then so is $\L g_{\LL}$. In characteristic $0$, also the converse holds:
\end{no}
%%%%%%%%%%%%%%%%%%%%%%%%%%

\begin{theorem}[{Cornulier, \cite[Thm.~3.25]{Co1}}]\label{Thm-Cornulier1}
If $\LL/\KK$ is a field extension of characteristic $0$, then $\L g$ is positively gradable iff $\L g_{\LL}$ is positively gradable.
\end{theorem}

\subsection{Dilation groups} 
In this subsection we are going to explain how positively gradable real Lie algebras gives rise to dilation groups in the sense of Definition \ref{Def-DilationGroup}.

%%%%%%%%%%%%%%%%%%%%%%%%%%
\begin{no}
\label{Par-BCH} 
Every real Lie algebra $\L g$ is the Lie algebra of a $1$-connected (i.e.\ connected and simply-connected) real Lie group $G$, which is unique up to isomorphism of Lie groups (i.e.\ smooth group isomorphism). 
Conversely, if $G$ is a $1$-connected real Lie group with nilpotent Lie algebra $\L g$, then $\exp: \L g \to G$ is a global diffeomorphism; we define
\[
\log := \exp^{-1}: G \to \L g \qand X \ast Y := \log(\exp(X)\exp(Y)) \quad (X,Y \in \L g).
\] 
Then $G(\L g) := (\L g, \ast)$ is a Lie group and $\exp$ defines an isomorphism of Lie groups $G(\L g) \to G$. The product $\ast$ on $\L g$ is called the \emph{Baker--Campbell--Hausdorff (BCH) product}.
\end{no}
%%%%%%%%%%%%%%%%%%%%%%%%%%

%%%%%%%%%%%%%%%%%%%%%%%%%%
\begin{no} It turns out that the Baker--Campbell--Hausdorff product is given by a universal power series. To describe this, we introduce the following notation. Let $\L g$ be a Lie ring and let $k,m \in \NN$.
Given $p = (p_1, \dots, p_k) \in \NN_0^k$ and  $q = (q_1, \dots, q_k) \in \NN_0^k$ we define
\[
c(k,m; p,q) := (k+1)(q_1 + \dots + q_k + 1)p_1! q_1! \cdots p_k!q_k! m! \in \NN
\]
and
\[
f[k,m; p,q](X,Y) := \mathrm{ad}(X)^{p_1}\mathrm{ad}(Y)^{q_1} \cdots \mathrm{ad}(X)^{p_k}\mathrm{ad}(Y)^{q_k} \mathrm{ad}(X)^m(Y) \in L^{2k+m+1}(\L g).
\] 
Then the Baker--Campbell--Hausdorff product of a nilpotent real Lie algebra $\L g$ is given by
\begin{equation}
\label{Eq-BCH}
X \ast Y = X+\underset{p_i + q_i >0}{\sum_{k,m \geq 0}} \frac{(-1)^k}{c(k,m; p,q)} f[k,m; p,q](X,Y).
\end{equation}
Note that this sum is finite, since $f[k,m; p,q](X,Y) = 0$ if $2k+m+1$ is larger then the nilpotency degree of $\L g$. In particular, if we choose a basis and identify $\L g$ with $\RR^n$ for some $n \in \NN$, then $\ast$ corresponds to a polynomial map $\RR^n \times \RR^n \to \RR^n$. In other words, $G(\L g)$ is isomorphic to (the group of $\RR$-points of) a linear algebraic group defined over $\RR$.
\end{no}
%%%%%%%%%%%%%%%%%%%%%%%%%%

%%%%%%%%%%%%%%%%%%%%%%%%%%
\begin{no}
\label{Par-StrongSubring}
If $\Lambda \subset \L g$ is a Lie subring of a nilpotent real Lie algebra, then for all $X,Y \in \Lambda$ we have $X \ast Y \in \frac{1}{M_s}\Lambda$, where 
\begin{equation}
\label{Eq-Mi}
M_i := \mathrm{lcm}\{m_1, \dots, m_i\} \quad \text{with} \quad m_i :=  \mathrm{lcm}(\{c(k,m; p,q) \mid 2k+m+1 =i\}).
\end{equation}
In particular, if $\L g_{\QQ} \subset \L g$ is a $\QQ$-Lie subalgebra of $\L g$, then $G(\L g_{\QQ}) := (\L g_{\QQ}, \ast) < G(\L g)$ is a subgroup (and in fact the group of $\QQ$-points of a linear algebraic group defined over $\QQ$ if $\L g_{\QQ}$ is finite-dimensional). 
\end{no}
%%%%%%%%%%%%%%%%%%%%%%%%%%

\begin{example}
For small nilpotency degrees we obtain $m_1 = 1$, $m_2 = 2$ and $m_3 = 12$, and
\[
X \ast Y = \left\{\begin{array}{ll} X+Y, & \text{if $\L g$ is abelian,}\\
 X + Y + \tfrac{1}{2} [X,Y],&  \text{if $\L g$ is $2$-step nilpotent,}\\
 X + Y + \frac{1}{2}[X, Y] + \frac{1}{12}[X,[X,Y]] + \frac{1}{12}[Y,[Y,X], & \text{if $\L g$ is $3$-step nilpotent.}
 \end{array} \right.
\]
\end{example}

\begin{definition} If $G$ is a $1$-connected nilpotent Lie group with Lie algebra $\L g$, then an $\NN$-grading $\gamma$ on $\L g$ is called an \emph{infinitesimal $\NN$-grading} of $G$. In this case, the pair $(G,\gamma)$  is called an \emph{$\NN$-graded} group, and we say that $G$ is \emph{$\NN$-gradable}.

\item If $(G,\gamma)$ is an $\NN$-graded group with associated derivation $\L d = \L d_\gamma$, then the trace $\kappa := \mathrm{tr}(\L d)$ is called the \emph{homogeneous dimension} of $G$. 
\end{definition}

\begin{remark}
The terminology varies in the literature: what we call an $\NN$-gradable, respectively $\NN$-graded group is simply called a \emph{gradable group}, respectively \emph{graded group} in \cite{LeDonne_Primer}. Siebert \cite{Siebert86} uses the term \emph{contractible group} instead of $\NN$-gradable group. The more analytically minded literature (e.g.\ \cite{FischerRuzhansky}) seems to prefer the term \emph{homogeneous group} instead of $\NN$-graded group, which apparently goes back to the influential monograph of Folland and Stein \cite{FollandStein}.
\end{remark}

\begin{construction}[Dilation family of an $\NN$-graded group]
\label{Const-DilFam} 
Let $(G,\gamma)$ be an $\NN$-graded group with Lie algebra $\L g$ and denote  by $\L d = \L d_\gamma$ the associated derivation of $\L g$. By \S \ref{Par-BCH} we may assume that $G = G(\L g)$, and we will do so to simplify notation. For $\lambda>0$ we define $D_{\lambda} \in \mathrm{Aut}(\L g)$ by
\[
D_{\lambda} := \exp(\log(\lambda) \cdot \L d).
\]
Then $D_{\lambda \mu} = D_{\lambda}D_{\mu}$ for all $\lambda, \mu > 0$ and since $D_{\lambda}$ preserves the Lie bracket for each $\lambda > 0$, it also preserves the BCH product on $\L g$. Consequently, $(D_{\lambda})_{\lambda > 0}$ is a one-parameter group in $\mathrm{Aut}(G(\L g))$, called the \emph{associated dilation family} of the $\NN$-graded group $(G,\gamma)$.
\end{construction}

Wr record the following formulas, which are immediate from the construction.

\begin{lemma}
\label{Lem-DilF1}
If $(G,\gamma)$ is an $\NN$-graded group with associated dilation family $(D_{\lambda})_{\lambda > 0}$, then 
\[\pushQED{\qed}
\lim_{\lambda \to 0} D_\lambda(x) = 0, \quad D_1(x) = x \quad \text{and} \quad \lim_{\lambda \to \infty} D_\lambda(x) = \infty \quad \text{for every }x \in G.
\]
\end{lemma}

The following lemma is \cite[p.~100, Equality~(3.6)]{FischerRuzhansky}. Recall that every nilpotent Lie group $G$ is unimodular, hence we do not need to distinguish between left- and right-Haar measures.

\begin{lemma}\label{HomDim} If $(G,\gamma)$ is an $\NN$-graded group of homogeneous dimension $\kappa$ with Haar measure $m_G$ and associated dilation family $(D_{\lambda})_{\lambda > 0}$, then for any Borel subset $S \subset G$ we have
\[
\pushQED{\qed}m_G(D_\lambda(S)) = \lambda^{\kappa} \cdot m_G(S).
\]
\end{lemma}

\begin{example}
\label{Exa-TwoMainGroups}
We return to Example~\ref{Exa-TwoMainExamples}.
\begin{enumerate}[(a)]
\item If $\L g$ is the real Heisenberg algebra, then using coordinates with respect to the basis $X,Y,Z$ as in Example \ref{Exa-TwoMainExamples}~(b) we get $G(\L g) \cong \RR^3$ with multiplication given by
\[
(x,y,z) \ast (x', y', z') = \left(x+x', y+y', z+z' + \frac 1 2  (xy' - yx') \right),
\]
and the dilation family associated with the stratification is given by
\[
D_{\lambda}(x,y,z) = (\lambda x, \lambda y, \lambda^2 z).
\]
\item If $\L g_\mu$ is as in Example \ref{Exa-TwoMainExamples}~(c), then $G(\L g_\mu) \cong \RR^7$ with multiplication given by
\[
\begin{pmatrix} x_1\\ x_2\\ x_3\\ x_4\\ x_5\\ x_6 \\ x_7 \end{pmatrix} \ast \begin{pmatrix} y_1\\ y_2\\ y_3\\ y_4\\ y_5\\ y_6 \\ y_7 \end{pmatrix} = \begin{pmatrix} x_1 + y_1\\ x_2 + y_2\\ x_3 + y_3\\ x_4 + y_4 + \frac 1 2 (x_1y_2-x_2y_1)\\ x_5 + y_5 + \frac 1 2 (x_2y_3 - x_3y_2) \\ x_6 + y_6 - \frac 1 2 (x_1y_3 - x_3y_1) \\ x_7 + y_7 + \frac \mu 2(x_2y_6 - x_6y_2) + \frac{1-\mu} 2 (x_3y_4-x_4y_3) \end{pmatrix}.
\]
Moreover, the dilation structure associated with the $\NN$-grading from Example \ref{Exa-TwoMainExamples}~(c) is given by
\[
D_\lambda(x_1, \dots, x_7) = (\lambda x_1, \lambda x_2, \lambda x_3, \lambda^2 x_4, \lambda^2 x_5, \lambda^2 x_6, \lambda^3 x_7).
\]
\end{enumerate}
\end{example}

\begin{definition} 
Let $(G,\gamma)$ be an $\NN$-graded group with associated dilation family $(D_{\lambda})_{\lambda > 0}$. A continuous non-negative function $|\cdot|: G \to \mathbb{R}_{\geq 0}$ is called a \emph{homogeneous quasinorm} for $(G,\gamma)$ if for all $x \in G$,
\begin{itemize}
\item[(H1)] $|x| = 0$ if and only if $x = e$.
\item[(H2)] $|x^{-1}| = |x|$.
\item[(H3)] $|D_{\lambda}(x)| = \lambda |x|$.
\end{itemize}
It is called a \emph{homogeneous norm} if moreover for all $x,y \in G$,
\begin{itemize}
\item[(H4)] $|xy| \leq |x| + |y|$.
\end{itemize}
\end{definition}

%%%%%%%%%%%%%%%%%%%%%%%%%%
\begin{no} 
We recall some basic results concerning homogeneous (quasi-)norms; see  \cite[Sec.\ 3.1.6]{FischerRuzhansky} for proofs and references. Firstly, for every homogeneous quasinorm there is a constant $C \geq 1$ such that
\[
|xy| \leq C \cdot (|x| + |y|) \quad \text{for all }x, y \in G;
\]
however, in general we cannot choose $C=1$. If $|\cdot|$ is a homogeneous norm for $(G,\gamma)$ then the we obtain a left-invariant metric on $G$ by setting
\[
d_{|\cdot|}: G \times G \to \mathbb{R}_{\geq 0}, \quad (x, y) \mapsto |x^{-1}y|.
\]
We refer to such a metric as a \emph{homogeneous metric} for $(G,\gamma)$. It follows from (H3) that if $d$ is a homogeneous metric for $(G,\gamma)$, then
\begin{equation}
\label{Eq-DilationMetricCompatible}
d(D_\lambda(x), D_\lambda(y)) = \lambda d(x,y) \quad \text{for all }x,y \in G.
\end{equation}
If $|\cdot|$ is merely a homogeneous quasinorm, then we can still define the function $d_{|\cdot|}$ as above, but it may fail to satisfy the triangle inequality. Instead there exists a constant $C > 0$ such that
\[
d_{|\cdot|}(x,y) \leq C(d_{|\cdot|}(x,z) + d_{|\cdot|}(z,y)) \quad (x,y,z \in G).
\]
Using this inequality one can show that the ``balls'' 
\[
B^{|\cdot|}(x, R) = \{y \in G \mid |x^{-1}y| < R\}
\]
with $x \in G$ and $R>0$ are open and relatively compact and generate the topology of $G$. Moreover, \eqref{Eq-OldLemma3_1} still holds for homogeneous quasinorms, whereas \eqref{Eq-OldLemma3_2}
does not (since its proof involves the triangle inequality).
\end{no}
%%%%%%%%%%%%%%%%%%%%%%%%%%

\begin{example}[Explicit homogeneous quasinorms]
Let $(X_1, \dots, X_n)$ be an eigenbasis of $\mathfrak{g}$ with respect to $\L d_\gamma$ with corresponding eigenvalues $(\mu_1, \dots, \mu_n)$. We then obtain a family of homogeneous quasi-norms for $(G,\gamma)$ by setting
\[
\left|\exp\left(\sum_{i=1}^n \alpha_i X_i\right) \right|_\infty := \max_{i=1, \dots, n} |\alpha_i|^{1/\mu_i} \qand \left|\exp\left(\sum_{i=1}^n \alpha_i X_i\right) \right|_p := \left(\sum_{i=1}^n |\alpha_i|^{p/\mu_i}\right)^{1/p}  
\]
for $1 \leq p < \infty$. In particular, homogeneous quasinorms exist for every $\NN$-graded group. The following proposition shows that every $\NN$-gradable group even admits a homogeneous norm, but it comes at the cost of not being able to write down an explicit formula.
\end{example}

\begin{proposition}[Hebisch-Sikora, \cite{HeSi90}] 
\label{Prop-HebischSikora}
If $(G,\gamma)$ is an $\NN$-graded group, then for any homogeneous quasinorm $|\cdot|$ for $(G,\gamma)$ there exist $C \geq 1$ and a homogeneous norm $\|\cdot\|$ for $(G, \gamma)$ such that
\begin{equation}
\label{Eq-NormQuasinorm}
C^{-1} |x| \leq \|x\|\leq C|x| \quad \text{ for all }x \in G.
\end{equation}
In fact, \eqref{Eq-NormQuasinorm} holds for every homogeneous norm $\|\cdot\|$ for $(G, \gamma)$.
In particular, there exists a homogeneous metric for $(G,\gamma)$, and all homogeneous metrics on $G$ are left-invariant, continuous, proper and bi-Lipschitz to each other.
\end{proposition}

From Lemma \ref{Lem-DilF1} and \eqref{Eq-DilationMetricCompatible} we deduce:

\begin{corollary}
\label{Cor-HomogeneousToDilation} 
If $(G,\gamma)$ is an $\NN$-graded group with associated dilation family $(D_\lambda)_{\lambda > 0}$ and $d$ is a homogeneous metric for $(G,\gamma)$, then $(G, d, (D_\lambda)_{\lambda > 0})$ is a dilation group.
\end{corollary}

\begin{definition} 
\label{Def-HomDilGroup}
We refer to $(G, d, (D_\lambda)_{\lambda > 0})$ as in Corollary~\ref{Cor-HomogeneousToDilation} as a \emph{homogeneous dilation group} with underlying  $\NN$-graded group $(G,\gamma)$.
\end{definition}

\begin{lemma}[Volume growth of homogeneous quasinorms]
\label{Lem-VolGrowthG} 
If $(G,\gamma)$ is an $\NN$-graded group of homogeneous dimension $\kappa$ and $|\cdot|$ is a homogeneous quasinorm for $(G,\gamma)$, then there exists $C>0$ such that for all $x,y \in G$ and $r>0$ we have
\[
m_G(B^{|\cdot|}(x, r)) =C \cdot r^{\kappa} \qand \lim_{s\to \infty} \frac{m_G(B^{|\cdot|}(x, r+s))}{m_G(B^{|\cdot|}(y, s))} = 1.
\]
\end{lemma}

\begin{proof} Set $B^{|\cdot|} := B^{|\cdot|}(e, 1)$ and $C := m_G(B^{|\cdot|})$. Observe that \eqref{Eq-OldLemma3_1} still holds for homogeneous quasinorms, hence
it follows with Lemma \ref{HomDim} that  for all $x \in G$ and $r > 0$ we have
\[
m_G(B^{|\cdot|}(x, r)) = m_G(x \cdot B^{|\cdot|}(e, r)) = m_G(B^{|\cdot|}(e, r)) = m_G(D_r(B^{|\cdot|}) =  r^\kappa \cdot m_G(B^{|\cdot|})= C \cdot r^{\kappa}.
\]
Since $\frac{C \cdot (r+s)^\kappa}{C \cdot s^\kappa} \to 1$ as $s \to \infty$, the lemma follows.
\end{proof}

\begin{corollary}[Exact polynomial growth]
\label{Cor-exapolygrowth} 
If $(G, d, (D_\lambda)_{\lambda > 0})$ is a homogeneous dilation group, then $(G,d)$ has exact polynomial growth.
\end{corollary}

We conclude this section by stating the classification of underlying groups of dilation groups, due to Siebert \cite{Siebert86}.

\begin{theorem}[Siebert \cite{Siebert86}]
\label{Thm-Siebert}
For a connected lcsc group $G$ the following are equivalent.
\begin{enumerate}[(i)]
\item $G$ is the underlying group of a dilation group $(G, d, (D_\lambda)_{\lambda > 0})$.
\item There exists a group homomorphism $D: \RR^{> 0} \to \mathrm{Aut}(G)$, $\lambda \mapsto D_\lambda$ such that for every $g \in G$ we have
\[
\lim_{\lambda \to 0} D_\lambda(g) = e, \quad D_1(g) = g \quad \text{and} \quad \lim_{\lambda \to \infty} D_\lambda(g) = \infty.
\]
\item There exists an automorphism $D \in \mathrm{Aut}(G)$ such that $D^{-k}(g) \to e$ for all $g \in G$.
\item $G$ is an $\NN$-gradable group, hence in particular a $1$-connected nilpotent Lie group.
\end{enumerate}
\end{theorem}

\begin{proof} The non-trivial implication is (iii)$\implies$(iv), and this is contained in \cite[Cor.\ 2.4]{Siebert86}. Now assume (iv) and fix an $\NN$-grading on $G$. If $(D_\lambda)_{\lambda > 0}$ denotes the associated dilation family, then (ii) holds by Lemma \ref{Lem-DilF1}  and (i) holds by Corollary~\ref{Cor-HomogeneousToDilation}. Finally, both (i) or (ii) imply (iii) by setting $D := D_\lambda$ for some $\lambda > 1$.
\end{proof}

\begin{remark} 
If we combine Theorem~\ref{Thm-Siebert} with \cite[Theorem~D]{CKLNO21}, then we can obtain a purely metric characterization of the underlying metric groups of dilation groups: suppose that $(X,d_X)$ is any connected locally compact metric space whose isometry group acts transitively on $X$ and such that there exists a bijection $D: X \to X$ and $\lambda > 1$ such that $d_X(D(x),D(y)) = \lambda d_X(x,y)$ for all $x,y \in X$. Then there exists a dilation group $(G, d, (D_\lambda)_{\lambda > 0})$ and an isometry $f: (X,d_X) \to (G, d)$ which intertwines $D$ with an automorphism of $G$ (which then necessarily satisfies Condition (iii) from Theorem~\ref{Thm-Siebert}).

It follows from Theorem~\ref{Thm-Siebert} and Corollary \ref{Cor-HomogeneousToDilation} that if a group underlies a dilation group, then it underlies a \emph{homogeneous} dilation group. This does not mean, however, that every dilation group is homogeneous: Every real $(n \times n)$-matrix $D$ with positive spectrum defines a dilation family on the abelian Lie group $(\RR^n, +)$ via $D_\lambda := \exp(\log(\lambda) D)$, but the corresponding dilation group is homogeneous iff the matrix is diagonalizable with spectrum in $\NN$. 

Nevertheless we will focus on homogeneous dilation groups, since this will allow us to make use of Corollary~\ref{Cor-exapolygrowth}.
\end{remark}

\subsection{Adapted lattices in homogeneous dilation groups} 
In this subsection we are going to show that every homogeneous dilation group which admits a lattice even admits an \emph{adapted} lattice in the sense of Definition~\ref{Def-AdaptedLattice}. 

%%%%%%%%%%%%%%%%%%%%%%%%%%
\begin{no}
\label{Par-StrongShit}
From now on let $G$ be a $1$-connected $s$-step nilpotent Lie group with Lie algebra $\L g$. Since $G$ is nilpotent, every lattice in $G$ is automatically cocompact, and we are interested in classifying these lattices up to commensurability. We say that a discrete additive subgroup $\Lambda \subset \L g$ is cocompact if $\L g/\Lambda$ is compact. Equivalently, there exists an $\RR$-basis $(X_1, \dots, X_d)$ of $\L g$ such that $\Lambda = \ZZ X_1 \oplus \dots \oplus \ZZ X_d$, and hence \[\Lambda_{\QQ} := \mathrm{span}_{\QQ}(\Lambda) = \QQ X_1 \oplus \dots \oplus \QQ X_d\]
is a $\QQ$-form of $\L g$. By basic linear algebra, two cocompact discrete subgroups of $\L g$ are commensurable if and only if they span the same $\QQ$-form.

Now let $\Lambda \subset \L g$ be an additive subgroup. Following Cornulier \cite[Def.\ A.2]{Co2}, we say that $\Lambda$ is a \emph{strong subring} of $\L g_{\QQ}$ if for all $i \geq 2$ and $Y_1, \dots, Y_i \in \Lambda$ we have
\begin{equation}
\label{Eq-CornulierStrong}
\mathrm{ad}(Y_1) \circ \dots \circ \mathrm{ad}(Y_{i-1})(Y_i) \in m_i \cdot \Lambda,
\end{equation}
where $m_i$ is as in \eqref{Eq-Mi}. Note that for $i=2$ this says that $\Lambda$ is closed under $(X,Y) \mapsto \frac 1 2[X,Y]$, and hence every strong subring is a Lie subring. 

Note that if $\Lambda \subset \L g$ is a strong subring, then $\Lambda$ is closed under $\ast$ by \eqref{Eq-BCH} and hence $\exp(\Lambda) \subset G$ is a subgroup. In particular, if $\Lambda \subset \L g$ is a cocompact strong subring, then $\exp(\Lambda) < G$ is a lattice; we refer to such lattices as \emph{strong lattices}. By definition, strong lattices in $G$ are in bijection with cocompact strong subrings of $\L g$.

If an additive subgroup $\Lambda \subset \L g$ is generated by finitely many elements $X_1, \dots, X_n$, then it is strong if \eqref{Eq-CornulierStrong} holds for all $Y_1, \dots, Y_i \in \{X_1, \dots, X_n\}$, and this condition can always be achieved by replacing $X_1, \dots, X_n$ by integer multiples. We apply this as follows:

If $(X_1', \dots, X_n')$ is any rational basis for $\L g$ then we can find integers $m_1, \dots, m_n$ such that the elements $X_j := m_j X_j'$ have the following properties: 
$(X_1, \dots, X_n)$ is a $\ZZ$-basis of a strong cocompact subring $\Lambda \subset \L g$ and a $\QQ$-basis of the rational form $\L g_{\QQ} =  \mathrm{span}_{\QQ}(X_1', \dots, X_n')$ of $\L g$, and $\exp(\Lambda) < G$ is a strong lattice. We then refer to $(X_1, \dots, X_n)$ as a \emph{strong rational basis} of $\L g$.

To summarize, every rational form of $\L g$ admits a strong rational basis and hence contains a strong cocompact subring of $\L g$; such strong cocompact subrings correspond bijectively to strong lattices in $G$.
We combine this with the following result, which is contained in \cite[Lemma A. 4]{Co2}.
\end{no}
%%%%%%%%%%%%%%%%%%%%%%%%%%

\begin{lemma}
\label{Lem-CoStrong}
Let $G$ be a $1$-connected nilpotent Lie group with Lie algebra $\L g$.
\begin{enumerate}[(a)]
\item If $\Lambda \subset \L g$ is a cocompact discrete Lie subring, then there exist cocompact strong subrings $\Lambda_1, \Lambda_2 \subset \L g$ and $C \in \NN$ such that 
\[
\Lambda_1 \subset \Lambda \subset \Lambda_2 \qand [\Lambda_2:\Lambda_1] \leq C.
\]
\item If $\Gamma<G$ is a lattice, then there exist strong lattices $\Gamma_1, \Gamma_2 < G$ and $C \in \NN$ such that
\[
\Gamma_1 \subset \Gamma \subset \Gamma_2 \qand [\Gamma_2:\Gamma_1] \leq C.
\]
\end{enumerate}
\end{lemma}

%%%%%%%%%%%%%%%%%%%%%%%%%%
\begin{no}
\label{Par-MoreStrong}
Cornulier actually proves in \cite[Lemma A.4]{Co2} the stronger statement that the constants $C$ in (i) and (ii) depend only on the dimension of $\L g$, but we will not need this fact here. From \S \ref{Par-StrongShit} and Lemma~\ref{Lem-CoStrong} we deduce that we have natural bijections between the following sets:
\begin{itemize}
\item the set of $\QQ$-forms of $\L g$;
\item the set of commensurability classes of cocompact discrete Lie subrings of $\L g$;
\item  the set of commensurability classes of cocompact discrete strong subrings of $\L g$;
\item the set of commensurability classes of strong lattices in $G$;
\item the set of commensurability classes of lattices in $G$.
\end{itemize}
Explicitly, if $\L g_{\QQ}$ is a $\QQ$-form of $\L g$, then we obtain a (strong) lattice $\Gamma < G$ by choosing a basis of $\L g_{\QQ}$, rescaling this basis into a strong rational basis $(X_1, \dots, X_n)$ of $\L g$ and then taking $\Gamma := \exp(\mathrm{span}_{\ZZ}(X_1, \dots, X_n))$. This construction produces all lattices in $G$ up to commensurability. In particular, $G$ contains a lattice if and only if is definable over $\QQ$.
\end{no}
%%%%%%%%%%%%%%%%%%%%%%%%%%

\begin{construction}[Adapted lattices in $\NN$-gradable groups]
\label{Con-BuildALattice}
Let $(G, \gamma)$ be an $\NN$-graded group with Lie algebra $\L g$ and assume that $\L g$ is definable over $\QQ$. Note that these assumptions imply in particular that $G$ is a $1$-connected nilpotent Lie group and that $G$ contains a lattice by \S \ref{Par-MoreStrong}.

Now choose a $\QQ$-form $\L g_{\QQ}$ of $\L g$. By Theorem \ref{Thm-Cornulier1} there then exists an $\NN$-grading $\gamma_{\QQ}$ on $\L g_{\QQ}$ which induces an $\NN$-grading $\gamma := (\gamma_{\QQ})_{\RR}$ on $\L g$. We denote by $\L d = \L d_{\gamma}$ the associated derivation and by  $(D_{\lambda})_{\lambda >0}$ the associated dilation family. 

We now choose an eigenbasis basis $(X_1', \dots, X_n')$ of $\L g_{\QQ}$ for $\L d$ and rescale it into a strong rational basis $(X_1, \dots, X_n)$ for $\L g$; then $(X_1, \dots, X_n)$ is still an eigenbasis for $\L d$ and freely generates $\Lambda := \mathrm{span}_{\ZZ}(X_1, \dots, X_n)$. In particular, $\L d(\Lambda) \subset \Lambda$ and hence $\Gamma := \exp(\Lambda)$ is a strong lattice in $G$ such that $D_\lambda(\Gamma) \subset \Gamma$ for all $\lambda \in \NN$. In particular, if $d$ is a homogeneous metric for $(G, \gamma)$, then $\Gamma$ is an adapted lattice for $(G, d, (D_{\lambda})_{\lambda >0})$. We refer to a lattice arising from this construction as a \emph{strong adapted lattice} for $(G, d, (D_{\lambda})_{\lambda >0})$.
\end{construction}

\begin{theorem}
\label{Thm-GoodGroups} 
For an lcsc group $G$ the following are equivalent.
\begin{enumerate}[(i)]
\item $G$ is a $1$-connected Lie group whose Lie algebra admits both a basis with rational structure constants and a derivation with positive eigenvalues.
\item $G$ is a $1$-connected nilpotent Lie group whose Lie algebra $\L g$ is definable over $\QQ$ and positively gradable. 
\item $G$  is a $1$-connected nilpotent Lie group underlying a homogeneous dilation group with an adapted lattice.
\item $G$ underlies a dilation group with an adapted lattice.
\item $G$ underlies a dilation group and contains a lattice.
\end{enumerate}
\end{theorem}

\begin{proof} 
The equivalence (i)$\iff$(ii) follows from \S \ref{Par-GradingDerivation} and  \S \ref{Par-NilpotentStratGradable}, and (ii)$\implies$(iii) follows from Construction \ref{Con-BuildALattice}. The implications (iii)$\implies$(iv)$\implies$(v) are obvious.
Finally, if $G$ underlies a dilation group, then it is a $1$-connected nilpotent Lie group by Theorem~\ref{Thm-Siebert}, and if it furthermore contains a lattice, then its Lie algebra must be definable over $\QQ$ by \S \ref{Par-MoreStrong}. This establishes (v)$\implies$(ii) and finishes the proof.
\end{proof}
By definition, the groups satisfying the equivalent conditions of Theorem \ref{Thm-GoodGroups} are precisely the \emph{substitution groups} as defined in the introduction. For lattices in substitution groups (not necessarily adapted) we have the following lattice point counting result: 

\begin{lemma}[Lattice point counting]
\label{Lem-VolGrowthGamma}
Let $(G,\gamma)$ be an $\NN$-graded group of homogeneous dimension $\kappa$ and let $\Gamma$ be a lattice in $G$. Set $B_r^\Gamma := \{ \gamma\in\Gamma \,|\, \|\gamma\|<r \}$, where $\|\cdot\|$ is a fixed homogeneous norm on $G$.
Then there exist a constants $C_1 >0$ and a function $\vartheta:(r_0, \infty) \to [0, \infty)$ with $\lim_{r \to \infty} \vartheta(r)  = 0$ such that
	\[
		C_1 \cdot r^{\kappa} - \vartheta(r) \cdot r^{\kappa} \leq \big| B_r^\Gamma \big| \leq C_1 \cdot r^{\kappa} + \vartheta(r) \cdot r^{\kappa} \quad \mbox{ for all }  r > r_0
		\] 
 In particular, $\lim_{s \to \infty}\frac{|B_{r + s}^\Gamma|} { |B_s^\Gamma|}  = 1$ and $\Gamma$ has exact polynomial growth with respect to the restriction of any homogeneous metric on $G$.
\end{lemma}

\begin{proof}  Since $G$ is nilpotent, the lattice $\Gamma$ is uniform and thus we may fix a relatively compact fundamental domain $V$ containing $e$ as an inner point. Let $C$ be the constant from Lemma \ref{Lem-VolGrowthG} and set $C_1 := \frac{C}{m_G(V)}$. Given $r>0$ we now abbreviate $B_r := \{ x\in G \,|\, \|x\|<r \}$ and $(B_r)_{\Gamma} := B_r^\Gamma V$. Since $V$ is relatively compact, there exists $r_0$ such that $V \subset B_{r_0}$, and since $V$ is a fundamental domain we have
\begin{equation}
\label{Eq-InnerApproximationBla}
(B_r)_{\Gamma} = \bigsqcup_{\gamma \in (B_r)_\Gamma} \gamma V, \quad \text{hence} \quad |B_r^\Gamma| = \frac{m_G((B_r)_\Gamma)}{m_G(V)} = \frac{m_G((B_r)_\Gamma)}{\mathrm{covol}(\Gamma)}.
\end{equation}
It follows from \eqref{Eq-OldLemma3_2} that $(B_r)_\Gamma \subset B_{r}V \subset B_{r}B_{r_0} \subset B_{r+r_0}$. Similarly one obtains $B_{r-r_0} \subset (B_r)_\Gamma$ for all $r \geq r_0$. We thus deduce from \eqref{Eq-InnerApproximationBla} and Lemma~\ref{Lem-VolGrowthG} that
	\[
	\frac{C (r-r_0)^\kappa}{\mathrm{covol}(\Gamma)} \leq \big| B_r^\Gamma | \leq \frac{C (r+r_0)^\kappa}{\mathrm{covol}(\Gamma)} \quad \mbox{ for all }  r > r_0. 
	\]
If we now define $\vartheta: (r_0 , \infty) \to [0, \infty)$ by
	\[
	\vartheta(r) := \frac{C}{\mathrm{covol}(\Gamma)} \cdot \Bigg( \frac{(r +r_0)^\kappa-(r -r_0)^\kappa}{r^\kappa} \Bigg),
	\]
then $\vartheta(r) = O(1/r)$ and the first statement follows. For the second statement we then observe that
	\begin{align*}
	\limsup_{r \to \infty} \frac{\big| B_{r+s}^\Gamma \big|}{\big| B_r^\Gamma \big|}  \leq \limsup_{r \to \infty}  \frac{\frac{C}{m_G(V)} (r + s)^{\kappa} + \vartheta(r)(r + s)^{\kappa}}{\frac{C}{m_G(V)} r^{\kappa} - \vartheta(r) r^{\kappa}} = 1
	\leq \liminf_{r \to \infty} \frac{\big| B_{r+s}^\Gamma \big|}{\big| B_r^\Gamma \big|}.
	\end{align*}
This finishes the proof.
\end{proof}
\begin{corollary}\label{ExExPolGr} Every substitution group admits a dilation datum of exact polynomial growth.\qed
\end{corollary}

\subsection{Split dilation data for substitution groups} 
In this subsection we are going to show that if $G$ is a substitution group of dimension $\geq 2$, then
it always admits a split dilation datum in the sense of Definition \ref{Def-SplitDilationDatum}, thereby establishing Theorem \ref{Thm-LieMain}. As a first step we show:
\begin{proposition}
\label{Prop-GroupsSplit}
If $G$ is a substitution group of dimension $\geq 2$, then $G \cong H \times_\beta \RR^m$ for some $m \in \NN$, non-compact Lie group $H$ and normalized cocycle $\beta$.
\end{proposition}

\begin{remark}[The abelian case]
\label{Rem-AbelianSplittingTrivial}
Note that if $\dim G = 1$, then $G \cong \RR$ does not admit a splitting as in the theorem, hence the assumption $\dim G \geq 2$ is necessary. If $G$ is abelian of dimension $d \geq 2$, then $G \cong \RR \times \RR^{d-1}$, hence Proposition \ref{Prop-GroupsSplit} holds. We now assume that $G$, and hence its Lie algebra $\L g$, is non-abelian.
\end{remark}

%%%%%%%%%%%%%%%%%%%%%%%%%%
\begin{no}
\label{Con-SplittingSetting}
From now on let $G$ be a non-abelian substitution group with Lie algebra $\L g$. We assume without loss of generality that $G = G(\L g)$ and fix a rational form $\L g_{\QQ}$ of $\L g$. By Theorem~\ref{Thm-Cornulier1} the Lie algebra $\L g_{\QQ}$ is $\NN$-gradable; we fix an $\NN$-grading $\gamma_{\QQ}$ of $\L g_{\QQ}$ with associated derivation $\L d$. We then denote by $(D_{\lambda})_{\lambda > 0}$ the associated dilation family and choose a compatible homogeneous metric $d$ on $G$.

As in Construction~\ref{Con-BuildALattice} we can choose a strong eigenbasis $\cB = (X_1, \dots, X_d)$ of $\L g_{\QQ}$ for $\L d$ with eigenvalues $\mu_1, \dots, \mu_d$. We assume that the order is chosen such that
\[\mu_1 \leq \dots \leq \mu_d\] and hence the spectral radius of $\L d$ is given by $\rho(\L d) := \mu_d$. We also set
\[j_o := \min\{j \in \{1, \dots, d\}\mid \mu_j = \rho(A)\} \qand m := d-j_o+1 > 0.\]
\end{no}
%%%%%%%%%%%%%%%%%%%%%%%%%%

\begin{notation} If $\L h$ is a Lie algebra, $V$ is a vector space and $\omega: \L h \times \L h \to V$ is a \emph{Lie algebra cocycle}, i.e.\ an alternating bilinear map satisfying
\[
\omega([[X,Y], Z]) = \omega([X,Z], Y) - \omega([Y,Z], X),
\]
then we denote by $\L h \oplus_{\omega} V$ the Lie algebra with underlying vector space is $\L h \oplus V$ and Lie bracket
\[
[(X_1, v_1), (X_2, v_2)] = ([X_1, X_2], \omega(X_1, X_2)).
\]
\end{notation}

\begin{construction}
\label{Con-LieAlgSplitting}
Using our strong eigenbasis $(X_1, \dots, X_d)$ of $\L g_{\QQ}$ we decompose $\L g$ as follows: We denote by $\mathfrak{g}_V = \mathrm{span}_{\RR}(X_{j_o}, \dots, X_{d})$ the $\L d$-eigenspace of $\mathfrak{g}$ with eigenvalue $\mu_d = \rho(\L d)$. For every $j \in \{1, \dots, d\}$ and $Y \in \mathfrak{g}_V$ we then have 
\[
[X_j, Y] \in \mathfrak{g}_{\rho(\L d) + \mu_j} =\{0\}, \quad \text{and hence}\quad [X_j, Y] = 0.
\]
This shows that $\mathfrak{g}_V$ is a central ideal of $\mathfrak{g}$; the quotient $\mathfrak{g}_H := \mathfrak{g}/\mathfrak{g}_V$ has basis $\overline{X}_1, \dots, \overline{X}_{j_o-1}$,where $\overline{X}_j$ denotes the image of $X_j$ in $\mathfrak{g}_H$ under the canonical projection,  and there exists a unique cocycle $\omega: \L g_H \times \L g_H \to \L g_V$ such that 
\begin{equation}
\label{Eq-Embedd-Horiz+Vert}
\iota_{\mathcal{B}}: \mathfrak{g} \to \mathfrak{g}_H \oplus_\omega \mathfrak{g}_V, \quad 
	X_j \mapsto 
		\left\{\begin{array}{ll} 
			(\overline{X}_j, 0), & j < j_o,\\ 
			(0, X_j), & j \geq j_o, 
		\end{array}\right.
\end{equation}
is an isomorphism. We refer to $ \mathfrak{g}_H$ and $\mathfrak{g}_V$ as the \emph{horizontal} and \emph{vertical part} of $\L g$ respectively. We observe that $\L g_H$ inherits a natural $\NN$-grading $\gamma_H$ from $\gamma$ by declaring that $\overline{X}_j$ is of homogeneous degree $\mu_j$. We also note that $\L g_H$ is non-trivial, since $\L g$ is non-abelian.
\end{construction}

\begin{construction}
\label{Con-GroupSplitting} 
We can lift Construction~\ref{Con-LieAlgSplitting} to the group level: If we set $G_H := G(\L g_H)$, then $G_H \cong \L g_H$ is non-compact, and the canonical projection $\L g \to \L g_H$ lifts to a canonical projection
\[
p_H: G \to G_H, \quad p_H\left(\sum_{i=1}^d \alpha_i X_i \right) = \sum_{i=1}^{j_o-1} \alpha_i \overline{X}_i.
\]
We now define a section of $p_H$ by
\[
\sigma: G_H \to G, \quad  \sum_{i=1}^{j_o-1} \alpha_i \overline{X}_i \mapsto  \sum_{i=1}^{j_o-1} \alpha_i {X}_i,
\]
which defines a cocycle $\beta: G_H \times G_H \to \RR^m$ as follows: Since for all $g_H, h_H \in G_H$ we have
$\sigma(g_H h_H)^{-1} \sigma(g_H) \sigma(h_H) \in \ker(p_H)$, we can define
\[
\beta(g_H, h_H) := (\alpha_{j_o}, \dots, \alpha_n)^\top, \quad \text{where} \quad \sigma(g_H h_H)^{-1} \sigma(g_H) \sigma(h_H) = \sum_{j= j_o}^d \alpha_j X_j.
\]
One checks that this defines indeed a cocycle and that there is an an isomorphism
\begin{equation}
\label{Eq-icB}
i_{\cB}: G \to G_H \times_\beta \RR^m, \quad \sum_{j=0}^d \alpha_j X_j \mapsto \left( \sum_{i=1}^{j_o-1} \alpha_i \overline{X}_i, (\alpha_{j_o}, \dots, \alpha_n)^\top\right).
\end{equation}
\end{construction}

\begin{proof}[Proof of Proposition~\ref{Prop-GroupsSplit}] Combine Remark~\ref{Rem-AbelianSplittingTrivial} and Construction~\ref{Con-GroupSplitting}.
\end{proof}

%%%%%%%%%%%%%%%%%%%%%%%%%%
\begin{no} Let $G$ be a non-abelian substitution group with Lie algebra $\L g$ and let $(G,d, (D_{\lambda})_{\lambda > 0})$ be the dilation group from \S \ref{Con-SplittingSetting}, and identify $G$ with $G_H \times_\beta \RR^m$ via the isomorphism $i_{\cB}$ from \eqref{Eq-icB}.
Explicitly,
\[
D_\lambda(g_H, v) = (D^H_{\lambda}(g_H), \lambda^{\rho(\L d)} v),
\]
where $(D^H_{\lambda})_{\lambda>0}$ is the dilation family on $G_H$ corresponding to the induced grading $\gamma_H$ of $\L g_H$ and $\rho(\L d) \in \NN$ is the spectral radius of $\L d$.
\end{no}
%%%%%%%%%%%%%%%%%%%%%%%%%%

\begin{construction} 
With notation as in \S \ref{Con-SplittingSetting} we now define $\Lambda := \mathrm{span}_{\ZZ}(X_1, \dots, X_d)$; by construction, this is a strong cocompact subgroup of $\L g$, and hence $\Gamma := \exp(\Lambda)<G$ is a lattice in $G$. Since $\cB$ is an eigenbasis for $\L d$, this lattice is adapted.

If $p_H: G \to G_H$ is as in Construction~\ref{Con-GroupSplitting}, then $\Gamma_H := p_H(\Gamma) = \exp(\mathbb{Z} \overline{X}_{1} + \dots + \mathbb{Z} \overline{X}_{j_o -1})$. Consequently, the section $\sigma$ restricts to a section of $p_H|_{\Gamma}: \Gamma \to \Gamma_H$. Since $\cB$ is a strong basis we have $\beta(\Gamma_V, \Gamma_V) \subset \ZZ^m$ and
$i_{\cB}$ restricts to an isomorphism $\Gamma \cong \Gamma_H \times_{\beta} \ZZ^m$. Moreover, for every $\lambda \in \NN$ we have $D_\lambda(\Gamma) \subset \Gamma$. We have thus established the following result.
\end{construction}

\begin{proposition}
\label{Prop-JustLatticeAndDilation}
Let $G$ be a non-abelian substitution group. Then $G \cong H \times_\beta \RR^m$ for some $m \in \NN$, non-compact Lie group $H$ and normalized cocycle $\beta$, and $H \times_\beta \RR^m$ underlies a dilation group and contains an adapted lattice $\Gamma$ satisfying conditions (A1) and (A2) from Definition \ref{Def-SplitDilationDatum} with $\Gamma_0= \ZZ^m$.
\end{proposition}

\begin{example}[The $2$-step nilpotent case]\label{Ex-Splitting_2Step} We now explain why every (non-abelian) $2$-step nilpotent substitution group $G$ admits a special substitution datum in the sense of \S \ref{SpecialSub}. For this we recall that 
the Lie algebra $\L g$ admits a stratification $\mathfrak g = \mathfrak g_1 \oplus [\mathfrak g, \mathfrak g]$ which is unique up to automorphisms. The corresponding derivation $\L d$ has spectrum $\{1,2\}$ and thus spectral radius $2$, and if $d := \dim \L g$ and $m := \dim [\L g,\L g]$, then with notation as above $\L g_H \cong \RR^{d-m}$ is abelian. The proof of Proposition \ref{Prop-JustLatticeAndDilation} then shows that there exist
\begin{itemize}
\item a normalized cocycle $\beta:\RR^{d-m} \times \RR^{d-m} \to \mathbb{R}^m$ such that $G \cong \RR^{d-m} \times_\beta \RR^m$;
\item a dilation family $D_{\lambda}$ on $G$ which corresponds to $D_\lambda(u, v), = (\lambda u, \lambda^2 v)$;
\item  a lattice $\Gamma_H < \RR^{d-m}$ such that $\beta(\Gamma_H \times \Gamma_H) \subset \ZZ^m$.
\end{itemize}
Note that since $\Gamma_H < \RR^{d-m}$ is a subgroup we automatically have $\lambda \Gamma_H < \Gamma_H$ for every $\lambda \in \NN$; this implies that the lattice $\Gamma < G$ corresponding to $\Gamma_H \times_\beta \ZZ^m$
satisfies $D_\lambda(\Gamma) \subset \Gamma$ for every $\lambda \in \NN$, and hence is split adapted to $((D_\lambda)_{>0}, d)$, where $d$ is a homogeneous metric for $(G, (D_{\lambda})_{\lambda > 0})$. %Finally, $\Gamma$ has exact polynomial volume growth by Lemma \ref{Lem-VolGrowthGamma}.
\end{example}

\begin{remark}
\label{Rem-AbelianLatticeDilation}
Proposition~\ref{Prop-JustLatticeAndDilation} also holds if $G$ is abelian of dimension $m \geq 2$. Indeed, in this case $G \cong \RR^{m-1} \times \RR$ and we can choose $\Gamma = \ZZ^{m-1} \times \ZZ$ and the dilation family whose underlying derivation is just the identity.
\end{remark}
To complete the proof of Theorem~\ref{Thm-LieMain} it remains to construct good fundamental domains for lattices in central extensions.

%%%%%%%%%%%%%%%%%%%%%%%%%%
\begin{no}\label{A1A2OK}
From now on let $G$ be a group of the form $G = H \times_\beta \RR^m$, where $H$ is a non-compact lcsc group and $\beta: H \times H \to \RR$ is a normalized cocycle. We assume that we are given a dilation group $(G,d, (D_{\lambda})_{\lambda > 0})$ and a lattice $\Gamma < G$ satisfying conditions (A1) and (A2) from Definition~\ref{Def-SplitDilationDatum} with $\Gamma_0= \ZZ^m$. We now discuss a construction of fundamental domains for the $\Gamma$-action on $G$ by left-multiplication.
\end{no}
%%%%%%%%%%%%%%%%%%%%%%%%%%

\begin{construction}
\label{Con-V0Explicit} 
We first consider the fundamental domain $V_0 := [-1/2, 1/2)^m$ for the action of $\ZZ^m$ on $\RR^m$ and fix $\lambda_0 > 1$ such that $D_{\lambda_0}(\Gamma) \subset \Gamma$. The latter assumption implies that $\lambda_0^\rho \ZZ^m \subset \ZZ^m$ and hence $\lambda_0^{\rho} \in \ZZ$. We now define
\[
F(\lambda_0) := \lambda_0^{\rho} V_0 \cap \ZZ^m = \begin{cases} 
		\{- j , - j +1,\,  \dots,  j -2,  j -1\}^m,\qquad 	&\lambda_0^{\rho} = 2j \text{ even},\\  
		\{- j , -j +2,\,  \dots, j-2, j \}^m,			& \lambda_0^{\rho} = 2j+1 \text{ odd},
	\end{cases}
\]
and claim that
\begin{equation}
\label{Eq-IntersectionV0}
\bigcap_{x \in  F(\lambda_0)} (x+ F(\lambda_0)) \neq \emptyset.
\end{equation}
Indeed, if $\lambda_0^{\rho}$ is even (respectively, odd), then the intersection in \eqref{Eq-IntersectionV0} contains $(-1,\ldots,-1)^\top$ (respecticely  $(0,\ldots,0)^\top$). This shows that $V_0$ satisfies Property (A4) from Definition \ref{Def-SplitDilationDatum}.
\end{construction}

%%%%%%%%%%%%%%%%%%%%%%%%%%
\begin{no} Let $G = H \times_\beta \RR^m$ be a non-abelian substitution group. By \S \ref{A1A2OK} we find a dilation group $(G,d, (D_{\lambda})_{\lambda > 0})$ and a lattice $\Gamma < G$ satisfying conditions (A1) and (A2), and by Construction \ref{Con-V0Explicit} we may assume that $\Gamma_0 < \RR^m$ admits a bounded fundamental domain $V_0 \subset \RR^m$ which contains a $0$-neighbourhood and satisfies (A4). Since $\Gamma_H \subset H$ is a uniform lattice, there exist an open identity neighbourhood $U \subset H$ and a compact set $V \subset H$ containing $U$, such that the multiplication map $m: \Gamma_H \times H \to H$ is injective on $\Gamma \times U$ and surjective on $\Gamma \times V$. Since $m$ has countable fibers, this implies that there exists a Borel set $V_H$ with $U \subset V_H \subset V$ such that  $m$ restricts to a bijection $\Gamma \times V_H \to H$, i.e.\ $V_H$ is a bounded identity neighbourhood and a fundamental domain for $\Gamma_H$ in $H$. We now define $V := V_H \times_{\beta} V_0$. Then $V$ is bounded, contains an open identity neighbourhood in $G$, and it is a fundamental domain for $\Gamma$ by the following general principle. \end{no}
%%%%%%%%%%%%%%%%%%%%%%%%%%

\begin{lemma}[Fundamental boxes]
\label{Lem-FundCell} 
If $V_H \subset H$ and $V_0 \subset \RR^m$ are fundamental domains for $\Gamma_H$ and $\ZZ^m$ respectively, then $V := V_H \times_\beta V_0 \subset G$ is a fundamental domain for $\Gamma$.
\end{lemma}
\begin{proof} Let $g = (g_H, x) \in G$. We choose $\gamma_H \in \Gamma_H$ such that $v_H := \gamma_Hg_H \in V_H$. Then there exists $y \in \RR^m$ such that $(\gamma_H,0)(g_H, x) = (v_H, y)$, and we choose $\gamma_V \in \Gamma_0$ such that  $y+ \gamma_V \in V_0$. Now set $\gamma := (\gamma_H, \gamma_V) = (e, \gamma_V)(\gamma_H, 0)$ so that
\[
\gamma g = (e, \gamma_V)(\gamma_H, 0)(g_H, x)=(e, \gamma_V)(v_H, y) = (v_H, y+\gamma_V) \in V_H \times V_0.
\]
Conversely, if $v = (v_H, v_0) \in V$ and $\gamma = (\gamma_H, \gamma_V)$ in $\Gamma$ with $\gamma v = (\gamma_H v_H, \gamma_V +v_0 + \beta(\gamma_H, v_H)) \in V$, i.e. $\gamma_H v_H \in V_H$ and $ \gamma_V +v_0 + \beta(\gamma_H, v_H) \in V_0$, then $\{v_H, \gamma_H v_H\} \subset V_H$, and since $V_H$ is a fundamental domain for $\Gamma_H$ we obtain $\gamma_H = e$. This in turn implies by Lemma~\ref{Lem-NiceCocycle} that $\beta(\gamma_H, v_H) = 0$ and hence $\{v_0, \gamma_V+v_0\} \in V_0$. Since $V_0$ is a fundamental domain for $\Gamma_0$ we deduce that $\gamma_V = 0$. Thus $\gamma$ is trivial, and hence $V$ is a fundamental domain.
\end{proof}
\begin{proof}[Proof of Theorem \ref{Thm-LieMain}] Combine Lemma \ref{Lem-VolGrowthGamma}, Proposition \ref{Prop-JustLatticeAndDilation}, Remark \ref{Rem-AbelianLatticeDilation}, Construction \ref{Con-V0Explicit} and Lemma \ref{Lem-FundCell}.
\end{proof}

\appendix

\section{On the classification of low-dimensional positively gradable $\RR$-Lie algebras}
\label{AppendixCensus}

%%%%%%%%%%%%%%%%%%%%%%%%%%
\begin{no} 
We have seen that substitution groups are precisely those $1$-connected Lie groups whose Lie algebra is positively gradable and definable over $\QQ$. The purpose of this appendix is to point out that, at least in low dimensions, such Lie algebras exist in abundance. By Theorem
\ref{Thm-Cornulier1} a real Lie algebra is positively gradable iff the complexification, which is a complex nilpotent Lie algebra, is positively gradable. This motivates a closer study of gradings on complex (nilpotent) Lie algebras.
\end{no}
%%%%%%%%%%%%%%%%%%%%%%%%%%

From now on until Corollary~\ref{Cor-ComplexClassification}, $\L g$ denotes a complex Lie algebra; for the moment we do not assume that $\L g$ is nilpotent. We study gradings of $\L g$ over \emph{finitely-generated torsion-free abelian} (fgtfa) groups.

\begin{construction} A Lie subalgebra $\L t < \mathrm{Der}(\L g)$ is called \emph{toral} if it is abelian and every $T \in \L t$ acts on $\L g$ by a diagonalizable endomorphism. Equivalently, $\L t \subset \mathrm{End}(\L g)$ is a simultaneously diagonalizable subalgebra. The term ``toral'' alludes to the fact that these subalgebras are precisely the Lie algebras of linear algebraic subgroups $T < \mathrm{Aut}(\L g)$ which are isomorphic (as linear algebraic groups over $\CC$) to some \emph{split torus} $(\CC^\times)^r$ for some $r \geq 0$. If $\L t < \mathrm{Der}(\L g)$ is toral and $\alpha \in \L t^*$ we define 
\[
\L g_\alpha(\L t) := \{X \in \L g \mid \forall\, T \in \L t: T(X) = \alpha(T) X\} \qand \mathcal W_{\L t} := \{\beta \in \L t^* \mid \L g_\beta \neq \{0\} \}.
\]
The elements of $\mathcal W_{\L t}$ are called the \emph{weights} of $\L t$ and $\L g_\alpha(\L t)$ is called the \emph{weight space} of $\alpha$. The weights generate a fgtfa group $A_{\L t} := \langle \mathcal W_{\L t}  \rangle < \L t^*$, and since 
$\L t \subset \mathrm{End}(\L g)$ is simultaneously diagonalizable we obtain an $A_{\L t}$-grading $\L t$ of $\L g$ by
\[
\gamma_{\L t} := (\L g_{\alpha}(\L t))_{\alpha \in A_{\L t}}.
\]
\end{construction}

\begin{lemma}
\label{Lem-GradingsAreToral} 
Every grading of $\L g$ over a fgtfa group $A$ is strictly equivalent to a grading of the form $\gamma_{\L t}$ for an toral subalgebra $\L t \subset \mathrm{Aut}(\L g)$. Moreover:
\begin{enumerate}[(i)]
\item If $\L t$ is toral and $\L t' = \mathrm{Ad}(u)(\L t)$ for some $u \in \mathrm{Aut}(\L g)$, then $\gamma_{\L t}$ and $\gamma_{\L t'}$ are equivalent.
\item If $\L t' \subset \L t$, then there exist a morphism $f: A_{\L t} \to A_{\L t'}$ such that $\gamma_{\L t'} = f(\gamma_{\L t})$.
\end{enumerate}
\end{lemma}

\begin{proof} We may assume that $A = \ZZ^r$ for some $r \geq 0$ and denote by $f_1, \dots, f_r: A \to \ZZ$ the coordinate projections. Then each $\gamma_j = (f_j)_*(\gamma)$ is a $\ZZ$-grading of $\L g$ and hence defines an associated derivation $\L d_i$. Since $\L d_1$, \dots, $\L d_r$ are diagonalizable and preserve each other's eigenspaces, they commmute and thus generate a toral subalgebra
 $\L t := \langle \L d_1, \dots, \L d_r \rangle < \mathrm{Der}(\L g)$ such that $\gamma$ and $\gamma_{\L t}$ are strictly equivalent.
 
 \item (i) One checks that if  $\L t' = \mathrm{Ad}(u)(\L t)$, then $\gamma_{\L t'} = u(\gamma_{\L t})$.
 \item (ii) The restriction $\L t^* \to (\L t')^*$ induces the desired map $f$.
\end{proof}

\begin{definition} A grading $\gamma$ of $\L g$ over a fgtfa group $A$ is called a \emph{maximal grading} if the following hold:
\begin{itemize}
\item $\mathcal W_\gamma$ generates $A$ as a group.
\item If $\gamma'$ is another grading of $\L g$ over another fgtfa group $B$, then $\gamma' = u(f_*(\gamma))$ for some $f: A \to B$ and $u \in \mathrm{Aut}(\L g)$.
\end{itemize}
\end{definition}

%%%%%%%%%%%%%%%%%%%%%%%%%%
\begin{no} It is a classical fact in a linear algebraic group over $\CC$ every torus is contained in a maximal torus and that any two maximal tori are conjugate \cite[Thm.\ 6.4.1]{Springer}. In particular, if $\L t_o < \mathrm{Der}(\L g)$ is a maximal toral subalgebra, then for any other toral subalgebra $\L t < \mathrm{Der}(\L g)$, then there exists $u \in \mathrm{Aut}(\L g)$ so that $\L t \subset \mathrm{Ad}(u)(\L t_o)$. In view of Lemma~\ref{Lem-GradingsAreToral} this implies that $\gamma_{\L t_o}$ is a maximal grading. On the other hand, it is immediate from the definition that any two maximal gradings are equivalent. To summarize:
\end{no}
%%%%%%%%%%%%%%%%%%%%%%%%%%

\begin{theorem}
\label{Thm-ExMaxGrad}
Every complex Lie algebra admits a maximal grading, which is unique up to equivalence of gradings.
\end{theorem}

%%%%%%%%%%%%%%%%%%%%%%%%%%
\begin{no} 
If $\L g$ is a complex Lie algebra with maximal grading $\gamma = (\L g_\alpha)_{\alpha \in A}$, then $A$ is a lattice in $A_{\RR} := A \otimes_\ZZ \RR$ and the latter is a finite-dimensional vector spaces. By Theorem~\ref{Thm-ExMaxGrad}, the rank of $A$ (or, equivalently, the dimension of $A_{\RR}$) depends only on $\L g$ (but not on the choice of maximal grading); it is called the \emph{toral rank} of $\L g$.
\end{no}
%%%%%%%%%%%%%%%%%%%%%%%%%%

\begin{example} If $\L g$ is a complex semisimple Lie algebra, then every derivation of $\L g$ is inner \cite[Thm.~5.3]{Humphreys}. Since $\L g \cong \mathrm{ad}(\L g) \cong \mathrm{Der}(\L g)$, gradings of $\L g$ are parametrized by toral subalgebras of $\L g$, and maximal gradings are induced by maximal toral subalgebras of $\L g$, which are precisely the Cartan subalgeba (see  \cite{Humphreys}). In the terminology of \cite{Humphreys} this means that the toral rank of $\L g$ is the rank of $\L g$ and that maximal gradings are precisely the \emph{root gradings} induced by Cartan subalgebras.

However, if $\L g$ is nilpotent, then any Cartan subalgebra of $\L g$ is equal to $\L g$, hence does not provide a non-trivial grading. In this case, a possible replacement for Cartan subalgebras of $\L g$ is given by maximal torus subalgebras of $\mathrm{Aut}(\L g)$; in many cases (i.e.\ if the toral rank is non-zero) this provides non-trivial gradings on $\L g$.
 
To summarize, maximal gradings are generalizations of roots gradings of semisimple Lie algebras which also provide some insight in the nilpotent case. In particular, they can be used to completely characterize positive gradability:
\end{example}

\begin{proposition}[Criterion for positive gradability]
\label{Prop-PosConv}
A complex Lie algebra $\L g$ is positively gradable if and only if for some (hence any) maximal grading $\gamma = (\L g_\alpha)_{\alpha \in A}$ of $\L g$ we have
\[0 \not \in \mathrm{conv}(\mathcal W_\gamma) \subset A_{\RR}.\]
\end{proposition}

\begin{proof} If $0 \in \mathrm{conv}(\mathcal W_\gamma) \subset A_{\RR}$ for a maximal grading $\gamma$, then $0 \in \mathrm{conv}(\mathcal W_{\gamma'}) \subset {\RR}$ for any $\ZZ$-grading $\gamma'$ of $\L g$, since the latter is strictly equivalent to $f_*(\gamma)$ for some $f: A \to \ZZ$. This implies that any $\ZZ$-grading $\gamma'$ of $\L g$ must have both non-negative and non-positive weights. In particular, $\L g$ does not admit an $\NN$-grading.

\item Conversely, if $0 \not \in \mathrm{conv}(\mathcal W_\gamma) \subset A_{\RR}$, then by basic convex geometry there is a linear form $f: A_{\RR} \to \RR$ such that
$f(\mathcal W_\gamma)  \subset \RR_{>0}$, and hence $f_*g$ is an $\RR_{>0}$-grading on $\L g$.
\end{proof}

%%%%%%%%%%%%%%%%%%%%%%%%%%
\begin{no} We now apply Proposition \ref{Prop-PosConv} to the case of low-dimensional nilpotent complex Lie algebras. All nilpotent Lie algebras of dimension $\leq 7$ have been classified, using different sets of invariants. The most cited list is that of Gong \cite{Gong98}, correcting earlier lists of Seeley, Ancochea-Goze and Romdhani, which lists Lie algebras by upper central series dimensions. For our purposes the most useful list is the list compiled by Carles \cite{Carles1, Carles2} with later corrections by Magnin \cite{Magnin} which lists the Lie algebras together with their maximal grading, sorted by toral rank. Using this list one can directly check the condition from Proposition~\ref{Prop-PosConv} in each case. The result is as follows:
\end{no}
%%%%%%%%%%%%%%%%%%%%%%%%%%

\begin{corollary}[Classification of positively gradable complex Lie algebras in dimension $\leq 7$]
\label{Cor-ComplexClassification}
Let $\L g$ be an indecomposable complex nilpotent Lie algebra of dimension $\leq 7$.
\begin{enumerate}[(i)]
\item If $\dim_{\CC} \L g \leq 6$ then $\L g$ is positively gradable. (Consequently, every decomposable Lie algebra of dimension $7$ is also positively gradable.)
\item If $\dim_{\CC} \L g = 7$ and $\L g$ has toral rank $\geq 2$, then $\L g$ is positively gradable.
\item If $\dim_{\CC} \L g = 7$ and $\L g$ has toral rank $1$, then $\L g$ is positively gradable unless is isomorphic to one of the four Lie algebras $\L n_{7,23}$, $\L n_{7,26}$, $\L n_{7,43}$ or $\L n_{7,47}$ from Carles' list \cite[p.10]{Carles1}, which correspond to the Lie algebras 12357B, 12457B, 12457K and 13457G in Gong's list. 
\item If $\dim_{\CC} \L g = 7$ and $\L g$ has toral rank $0$, then $\L g$ is not positively gradable. These Lie algebras are listed in \cite[pp.~7-9]{Carles1} and correspond to the Lie algebras 123457E, 123457F, 123457H, 12457J, 13457I and 12457I and the infinite family 12457N in Gong's list.
\end{enumerate}
\end{corollary}

\begin{remark}
We emphasize that in general there is no reason for a nilpotent complex Lie algebra of toral rank $\geq 2$ to be positively gradable; this is just an accident in dimension $7$.
\end{remark}

%%%%%%%%%%%%%%%%%%%%%%%%%%
\begin{no}
\label{Par-GongReal}
We now return to the real case; thus from now on $\L g$ denotes a real Lie algebra. Conveniently, Gong \cite{Gong98} also provides a list of all $7$-dimensional real Lie algebras up to isomorphism. As far as the complex Lie algebras listed in (iii) and (iv) of Corollary~\ref{Cor-ComplexClassification} are concerned,
\begin{itemize}
\item each Lie algebra in the infinite family $12457N$ has two isomorphism classes of real forms (denoted 12457N and 12457N${}_2$), except the one with parameter $\lambda =1$ which has a third isomorphism class of real forms (denoted 12457N${}_1$);
\item each of the Lie algebras 12457J, 12357B and 123457H has two isomorphism classes of real forms;
\item the remaining Lie algebras 12457B, 12457I, 12457K, 13457G, 13457I, 123457E and 123457F  have a single isomorphism class of real forms.
\end{itemize}
In view of Theorem \ref{Thm-Cornulier1} this establishes the following result.
\end{no}
%%%%%%%%%%%%%%%%%%%%%%%%%%

\begin{corollary}[Classification of positively gradable real Lie algebras in dimension $\leq 7$]
\label{Cor-ClassificationReal}\quad\quad
Let $\L g$ be a real nilpotent Lie algebra of dimension $\leq 7$.
\begin{enumerate}[(i)]
\item If $\dim \L g \leq 6$ or $\L g$ is decomposable, then $\L g$ is positively gradable.
\item Exactly $126$ of the $140$ isolated indecomposable real nilpotent Lie algebras of dimension $7$ in Gong's list are positively gradable.
\item Exactly $7$ of the $9$ infinite families of indecomposable  real nilpotent Lie algebras of dimension $7$ in Gong's list are positively gradable.
\end{enumerate}
Moreover,  the indecomposable real nilpotent Lie algebras of dimension $7$, which are not positively gradable, are precisely those listed in \S \ref{Par-GongReal}.
\end{corollary}

%%%%%%%%%%%%%%%%%%%%%%%%%%
\begin{no} 
All of the $140$ isolated indecomposable real nilpotent Lie algebras of dimension $7$ in Gong's list are defined over $\QQ$; in fact, the bases given in \cite{Gong98} already have rational structure constants. This yields 126 pairwise non-isomorphic examples of rationally $\NN$-gradable groups. As for the infinite families, these come with a real parameter $\lambda$, and at least for rational choices of $\lambda \in \QQ$ we obtain rational structure constants. We thus obtain $7$ infinite families of pairwise non-isomorphic examples of rationally $\NN$-graded groups. 
\end{no}
%%%%%%%%%%%%%%%%%%%%%%%%%%

%%%%%%%%%%%%%%%%%%%%%%%%%%%%%%%%%%%%%%%%%%%%%%%%%%%%%
%%%%%%%%%%           References            %%%%%%%%%%
%%%%%%%%%%%%%%%%%%%%%%%%%%%%%%%%%%%%%%%%%%%%%%%%%%%%%
\bibliographystyle{amsalpha}

\newcommand{\etalchar}[1]{$^{#1}$}
\def\cprime{$'$} \def\cprime{$'$} \def\cprime{$'$}
\providecommand{\bysame}{\leavevmode\hbox to3em{\hrulefill}\thinspace}
\providecommand{\MR}{\relax\ifhmode\unskip\space\fi MR }
\providecommand{\MRhref}[2]{%
  \href{http://www.ams.org/mathscinet-getitem?mr=#1}{#2}
}
\providecommand{\href}[2]{#2}

\end{document}